\documentclass[letterpaper, 10pt, dvipsnames]{article}

\usepackage{type1ec}
\usepackage[T1]{fontenc}
\usepackage[utf8]{inputenc}    
\usepackage[english]{babel}
\usepackage[normalem]{ulem}
\usepackage{dsfont, bm, pifont, mathrsfs}
\usepackage{stmaryrd}
\usepackage{bbm}
\usepackage{enumitem}
\usepackage{bold-extra}
\usepackage{float, graphicx, wrapfig, tikz}
\usetikzlibrary{positioning, fit}
\usepackage[margin=0.8in]{subcaption}
\usepackage[font=small, margin=0.6in]{caption}

\usepackage{amsthm, amsmath, amsfonts, amssymb, mathrsfs, mathtools}
\usepackage[alphabetic]{amsrefs}

\usepackage{hyperref, xcolor, titlesec} 
\hypersetup{colorlinks = true, urlcolor = RoyalBlue!70!Black,linkcolor=BrickRed, citecolor=ForestGreen, bookmarksopen = true}
\usepackage[nomarginpar]{geometry}
\numberwithin{equation}{section}
\usepackage[boxsize = .3em]{ytableau}
\usepackage{cleveref}

\DeclareMathOperator{\Id}{Id}

\DeclareMathOperator{\Var}{Var}
\renewcommand{\Im}{\operatorname{Im}}

\DeclareMathOperator{\Tr}{Tr}
\newcommand{\tr}{\Tr}

\newcommand{\cob}[1]{\textcolor{RoyalBlue}{#1}}
\newcommand{\cor}[1]{\textcolor{BrickRed}{#1}}

\newcommand{\N}{\mathbb{N}}

\newcommand{\Rbb}{\mathbb{R}}
\newcommand{\C}{\mathbb{C}}

\newcommand{\e}{\mathrm{e}}

\let\d\relax
\newcommand{\d}{\mathrm{d}}
\newcommand{\dif}{\operatorname{d}}

\newcommand{\A}{\mathcal{A}}
\newcommand{\Au}{\underline{\mathcal{A}}}
\newcommand{\B}{\mathcal{B}}
\newcommand{\D}{\mathcal{D}}

\newcommand{\Q}{\mathcal{Q}}
\newcommand{\R}{\mathcal{R}}

\newcommand{\T}{\mathcal{T}}

\newcommand{\ebf}{\mathbf{e}}

\renewcommand{\u}{\mathbf{u}}
\renewcommand{\v}{\mathbf{v}}
\newcommand{\w}{\mathbf{w}}
\newcommand{\x}{\mathbf{x}}
\newcommand{\y}{\mathbf{y}}
\newcommand{\z}{\mathbf{z}}

\renewcommand{\epsilon}{\varepsilon}
\renewcommand{\phi}{\varphi}
\renewcommand{\hat}{\widehat}
\renewcommand{\tilde}{\widetilde}
\newcommand{\wt}{\widetilde}

\let\O\relax
\newcommand{\O}[1]{\mathcal{O}\left(#1\right)}

\renewcommand{\P}{\mathds{P}}
\newcommand{\E}{\mathds{E}}
\newcommand{\Eb}{\mathbf{E}}

\newcommand{\bma}{\begin{bmatrix}}
\newcommand{\ema}{\end{bmatrix}}
\def\bet{\begin{thm}}
\def\eet{\end{thm}}
\def\bel{\begin{lem}}
\def\eel{\end{lem}}
\def\bas{\begin{ass}}
\def\eas{\end{ass}}
\def\bed{\begin{defn}}
\def\eed{\end{defn}}
\def\bep{\begin{prop}}
\def\eep{\end{prop}}
\def\beq{\begin{equation}}
\def\eeq{\end{equation}}
\def\bea{\begin{equation*}}
\def\eea{\end{equation*}}
\def\bex{\begin{ex}}
\def\eex{\end{ex}}
\def\bp{\begin{proof}}
\def\ep{\end{proof}}
\def\benr{\begin{enumerate}[label=(\roman*)]}
\def\eenr{\end{enumerate}}

\newcommand{\nt}{\tilde{n}}

\newcommand{\unn}[2]{[\![#1,#2]\!]}

\newtheorem{ccounter}{ccounter}[section]
\newtheorem{thm}[ccounter]{Theorem}
\newtheorem{lem}[ccounter]{Lemma}
\newtheorem{corollary}[ccounter]{Corollary}
\newtheorem{defn}[ccounter]{Definition}
\newtheorem{prop}[ccounter]{Proposition}
\newtheorem{ass}[ccounter]{Assumption}
\newtheorem{ex}[ccounter]{Example}
\theoremstyle{definition}
\newtheorem{rmk}[ccounter]{Remark}

\titleformat{\section}[block]{\normalfont\filcenter}{\Large\thesection .}{.7em}{\Large\scshape}
\titleformat{\subsection}[runin]{\normalfont}{\large \bf \thesubsection .}{.5em}{\large\bf}[.]
\titleformat{\subsubsection}[runin]{\normalfont}{\bf \thesubsubsection .}{.5em}{\bf}[.]
\titleformat*{\paragraph}{\itshape\mdseries}

\begin{document}
\tikzset{every node/.style={circle, minimum size=.1cm, inner sep = 2pt, scale=1}}

\title{\vspace{-5ex}\bfseries\scshape{Eigenvalue distribution of the Neural Tangent Kernel in the quadratic scaling}}
\author{L. \textsc{Benigni}\thanks{Supported in part by a Natural Sciences and Engineering Research Council of Canada (NSERC) RGPIN 2023-03882 \& DGECR 2023-00076.}\\\vspace{-0.15cm}\footnotesize{\it{Universit\'e de Montr\'eal}}\\\footnotesize{\it{lucas.benigni@umontreal.ca}}\and E. \textsc{Paquette}\thanks{Supported in part by a Natural Sciences and Engineering Research Council of Canada (NSERC) RGPIN 2025-04643.}\\\vspace{-0.15cm}\footnotesize{\it{McGill University}}\\\footnotesize{\it{elliot.paquette@mcgill.ca}}}
\date{}
\maketitle
\abstract{
  We compute the asymptotic eigenvalue distribution of the neural tangent kernel of a two-layer neural network under a specific scaling of dimension. Namely, if $X\in\Rbb^{n\times d}$ is an i.i.d random matrix, $W\in\Rbb^{d\times p}$ is an i.i.d $\mathcal{N}(0,1)$ matrix and $D\in\Rbb^{p\times p}$ is a diagonal matrix with i.i.d bounded entries, we consider the matrix 
  \[
  \mathrm{NTK} 
  = 
  \frac{1}{d}XX^\top 
  \odot 
  \frac{1}{p}
  \sigma'\left(
    \frac{1}{\sqrt{d}}XW
  \right)D^2
  \sigma'\left(
    \frac{1}{\sqrt{d}}XW
  \right)^\top
  \]
  where $\sigma'$ is a pseudo-Lipschitz function applied entrywise and under the scaling $\frac{n}{dp}\to \gamma_1$ and $\frac{p}{d}\to \gamma_2$. We describe the asymptotic distribution as the free multiplicative convolution of the Marchenko--Pastur distribution with a deterministic distribution depending on $\sigma$ and $D$.
}
\section{Introduction}

Understanding the theoretical foundations of neural network training remains one of the central challenges in machine learning. Despite the empirical success of deep learning \cites{imagenet, bert, diffusion}, the mechanisms underlying gradient-based optimization in high-dimensional parameter spaces are not well understood. Before delving into the theoretical framework, we first establish our notation for neural networks. Consider a dataset $X \in \mathbb{R}^{n \times d}$ where each row represents a $d$-dimensional input sample, with $n$ total samples. A $(L+1)$-layered neural network computes the function:
\[
\hat{\mathbf{y}}_L(X,\Theta) 
= 
 \sigma(X_{L-1} W^{(L)})\mathbf{a}\in\Rbb^{n}\quad \text{with} \quad X_\ell = \sigma(X_{\ell-1}W^{(\ell)})\in\Rbb^{n\times p_\ell}\quad\text{and}\quad X_0=X.
\]
where $W^{(\ell)} \in \mathbb{R}^{p_{\ell-1} \times p_\ell}$ (with $p_0=d$) represents the weights connecting the hidden layers, $\sigma: \mathbb{R} \to \mathbb{R}$ is a nonlinear activation function applied element-wise, and $\mathbf{a} \in \mathbb{R}^{p_{L}}$ is the output layer that collapses the hidden representations to produce outputs in $\mathbb{R}^n$. This is illustrated in Figure \ref{fig:nn}.

\begin{figure}[!ht]
  \centering 
  \begin{tikzpicture}
			\node[draw = ForestGreen, fill = ForestGreen!20!White, circle, minimum size = .5cm, ] (11) at (-.5,1) {};
			\node[draw = ForestGreen, fill = ForestGreen!20!White, circle, minimum size = .5cm, ] (12) at (-.5,2) {};
			\node[draw = ForestGreen, fill = ForestGreen!20!White, circle, minimum size = .5cm, ] (13) at (-.5,3) {};
			
			\node[draw=none] at (-.5, 0.25) {\textcolor{ForestGreen}{${X=X_0}$}};

			\node[draw = RoyalBlue, fill = RoyalBlue!20!White, circle, minimum size = .5cm, ] (21) at (1,0) {};
			\node[draw = RoyalBlue, fill = RoyalBlue!20!White, circle, minimum size = .5cm, ] (22) at (1,1) {};
			\node[draw = RoyalBlue, fill = RoyalBlue!20!White, circle, minimum size = .5cm, ] (23) at (1,2) {};
			\node[draw = RoyalBlue, fill = RoyalBlue!20!White, circle, minimum size = .5cm, ] (24) at (1,3) {};
			\node[draw = RoyalBlue, fill = RoyalBlue!20!White, circle, minimum size = .5cm, ] (25) at (1,4) {};

			\node[draw=none] at (1, -0.5) {\cob{${X_1}$}};
			\node[draw=none] at (0.15, 3.8) {\cob{$W_1$}};

			\node[draw = RoyalBlue, fill = RoyalBlue!20!White, circle, minimum size = .5cm, ] (31) at (2.5,0) {};
			\node[draw = RoyalBlue, fill = RoyalBlue!20!White, circle, minimum size = .5cm, ] (32) at (2.5,1) {};
			\node[draw = RoyalBlue, fill = RoyalBlue!20!White, circle, minimum size = .5cm, ] (33) at (2.5,2) {};
			\node[draw = RoyalBlue, fill = RoyalBlue!20!White, circle, minimum size = .5cm, ] (34) at (2.5,3) {};
			\node[draw = RoyalBlue, fill = RoyalBlue!20!White, circle, minimum size = .5cm, ] (35) at (2.5,4) {};

			\node[draw=none] at (1.75, 4.3) {\cob{$W_2$}};
			\node[draw=none] at (2.5, -0.5) {\cob{${X_2}$}};

			\node[draw= none] at (3.5,2) {\Large$\cdots$};

			\node[draw = RoyalBlue, fill = RoyalBlue!20!White, circle, minimum size = .5cm, ] (41) at (4.5,0) {};
			\node[draw = RoyalBlue, fill = RoyalBlue!20!White, circle, minimum size = .5cm, ] (42) at (4.5,1) {};
			\node[draw = RoyalBlue, fill = RoyalBlue!20!White, circle, minimum size = .5cm, ] (43) at (4.5,2) {};
			\node[draw = RoyalBlue, fill = RoyalBlue!20!White, circle, minimum size = .5cm, ] (44) at (4.5,3) {};
			\node[draw = RoyalBlue, fill = RoyalBlue!20!White, circle, minimum size = .5cm, ] (45) at (4.5,4) {};

			\node[draw=none] at (4.5, -0.5) {\cob{${X_L}$}};
			\node[draw=none] at (5.4, 3.2) {\cob{$\mathbf{a}$}};

			\node[draw = BrickRed, fill = BrickRed!20!White, circle, minimum size = .5cm, ] (51) at (6,2) {};

			\node[draw=none] at (6.75, 1.5) {\cor{$\hat{\mathbf{y}}_L(X,\Theta)$}};

			\draw[->] (11) edge (21); 
			\draw[->] (11) edge (22);
			\draw[->] (11) edge (23);
			\draw[->] (11) edge (24);
			\draw[->] (11) edge (25);

			\draw[->] (12) edge (21);
			\draw[->] (12) edge (22);
			\draw[->] (12) edge (23);
			\draw[->] (12) edge (24);
			\draw[->] (12) edge (25);

			\draw[->] (13) edge (21);
			\draw[->] (13) edge (22);
			\draw[->] (13) edge (23);
			\draw[->] (13) edge (24);
			\draw[->] (13) edge (25);

			\draw[->] (21) edge (31); 
			\draw[->] (21) edge (32);
			\draw[->] (21) edge (33);
			\draw[->] (21) edge (34);
			\draw[->] (21) edge (35);

			\draw[->] (22) edge (31);
			\draw[->] (22) edge (32);
			\draw[->] (22) edge (33);
			\draw[->] (22) edge (34);
			\draw[->] (22) edge (35);

			\draw[->] (23) edge (31);
			\draw[->] (23) edge (32);
			\draw[->] (23) edge (33);
			\draw[->] (23) edge (34);
			\draw[->] (23) edge (35);
			
			\draw[->] (24) edge (31); 
			\draw[->] (24) edge (32);
			\draw[->] (24) edge (33);
			\draw[->] (24) edge (34);
			\draw[->] (24) edge (35);

			\draw[->] (25) edge (31);
			\draw[->] (25) edge (32);
			\draw[->] (25) edge (33);
			\draw[->] (25) edge (34);
			\draw[->] (25) edge (35);

			\draw[->] (41) edge (51); 
			\draw[->] (42) edge (51);
			\draw[->] (43) edge (51);
			\draw[->] (44) edge (51); 
			\draw[->] (45) edge (51);
			
		\end{tikzpicture}
    \caption{Illustration of a multi-layered feed-forward neural network}
    \label{fig:nn}
  \end{figure}
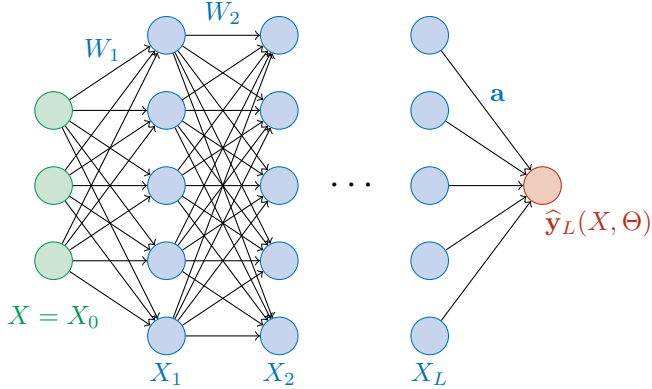

  The parameters $\theta = \{W^{(1)},\dots,W^{(L)}, \mathbf{a}\}$ are typically initialized randomly and updated through gradient-based optimization to minimize a loss function on the training data. Indeed, in supervised learning, the goal is to learn a mapping from inputs to target outputs. Given a training dataset $(X, \mathbf{y})$ where $\mathbf{y} \in \mathbb{R}^n$ contains the corresponding target values for each input sample, we seek to minimize a loss function $\mathcal{L}(\hat{\y}_L(X; \Theta), \y)$ that measures the discrepancy between the network's predictions and the true targets. Common choices include the $\ell^2$-loss $\mathcal{L}(\x, \y) = \frac{1}{2}\Vert \x - \y\Vert_2^2$ for regression tasks, but also its regularized versions, cross-entropy loss for classification tasks or plenty of others. Training proceeds by iteratively updating the parameters $\Theta$ using some optimization algorithm. This optimization process aims to find parameters that generalize well to unseen data, though understanding the dynamics of this high-dimensional, non-convex optimization problem has remained challenging.

  For simplicity, consider the $\ell^2$ training loss and optimize under gradient flow. Then we have 
  \[
  \frac{\d \Theta_t}{\d t} = -\nabla_\Theta \mathcal{L}(\hat{\y}(X,\Theta_t),\y) 
  =
  -\nabla_\Theta \hat{\y}_L(X,\Theta_t)^\top
  \left(
    \hat{\y}_L(X,\Theta_t)-\y
  \right)
  \] 
  where the Jacobian matrix of the predictions with respect to the parameters $\Theta$ is given by 
  \[
  \nabla_\Theta \hat{\y}_L(X,\Theta)^\top
  =
  \left(
    \nabla_\Theta \hat{\y}_L(\x_1,\Theta)\vert \dots\vert \nabla_\Theta\hat{\y}_L(\x_n,\Theta)
  \right)
  \]
with $\x_i$ the $i$-th row of $X$. In particular, when one considers the time evolution of the output of the networks, one gets 
\[
\frac{\d}{\d t}(\hat{\y}_L(X,\Theta_t)-\y)
=
\nabla_\Theta \hat{\y}_L(X,\Theta_t)\frac{\d \Theta_t}{\d t}
=
-\nabla_\Theta \hat{\y}_L(X,\Theta_t)\nabla_\Theta \hat{\y}_L(X,\Theta_t)^\top 
\left(
  \hat{\y}_L(X,\Theta_t)-\y
\right).
\]
Thus, we see that the dynamics of the output is completely described by the so-called Neural Tangent Kernel \cites{jacot2018, zhu2019,du2018gradient},
\[
K^{\mathrm{NTK}}_t = \nabla_\Theta \hat{\y}_L(X,\Theta_t)\nabla_\Theta \hat{\y}_L(X,\Theta_t)^\top.
\]
In particular, we see that if $(\lambda_{i,t},\u_{i,t})$ are the eigenvalue-eigenvector pairs of $K^{\mathrm{NTK}}_t$, we can observe that 
\[
\u_{i,t}^\top \frac{\d}{\d t}\left(\hat{\y}_L(X,\Theta_t)-\y\right) = -\lambda_{i,t}\u_{i,t}^\top\left(\hat{\y}_L(X,\Theta_t)-\y\right).
\]  
The eigenvalues,and eigenvectors of the NTK matrix thus play a critical role in determining the learning dynamics, as they directly influence the rate of convergence for different components of the function being learned. Components corresponding to larger eigenvalues are learned faster during training, while those associated with smaller eigenvalues require more iterations to converge. Understanding this spectral structure is therefore essential for predicting which functions can be efficiently learned by neural networks and for designing architectures that optimize learning efficiency. In the infinite-width limit, or with small learning rate, the NTK remains constant during training \cites{jacot2018,du2018gradient,du2019gradient,zhu2019,chizat2019lazy}, transforming the complex neural network optimization into a linear dynamical system governed by the kernel matrix. In particular, this means that in the NTK framework, for sufficiently wide networks, when trained with gradient descent, the output of a neural network behave equivalently to a kernel method with a fixed, deterministic kernel that depends only on the network architecture and initialization. Thus, if the parameters are initialized randomly, the kernel matrix becomes a fixed large random matrix which governs the full dynamics of the output.

In this paper, we focus on the case of a two-layer neural network, the simplest case of a nonlinear neural network, such that for $W\in \Rbb^{d\times p}$ and $\mathbf{a}\in\Rbb^p$ we have,
\[
\hat{\y}(X,W,\mathbf{a}) = \frac{1}{\sqrt{p}}\sigma\left(\frac{1}{\sqrt{d}}XW\right)\mathbf{a}\in\Rbb^n.
\]
In this case, the NTK matrix at initialization can be computed easily and one gets the sum of two terms, coming from differentiating with respect to $W$ or to $\mathbf{a},$
\begin{equation}\label{eq:ntk}
K^{\mathrm{NTK}} = \frac{1}{p}\sigma\left(\frac{1}{\sqrt{d}}XW\right)\sigma\left(\frac{1}{\sqrt{d}}XW\right)^\top
 +
  \frac{1}{d}XX^\top\odot \frac{1}{p}\sigma'\left(\frac{1}{\sqrt{d}}XW\right)\mathrm{Diag}(\mathbf{a})^2\sigma'\left(\frac{1}{\sqrt{d}}XW\right)^\top \in\Rbb^{n\times n}
\end{equation}
where $\odot$ corresponds to the Hadamard, or entrywise, product $(A\odot B)=A_{ij}B_{ij}$ and $\mathrm{Diag}(\mathbf{a})\in\Rbb^{p\times p}$ is the diagonal matrix whose diagonal is given by the vector $\mathbf{a}.$ In order to understand the eigenvalue distribution, the regime of dimension is greatly important. Indeed, if one consider the classical random matrix theory regime, or \emph{linear} regine $\frac{n}{d}\to \gamma_1$ and $\frac{p}{d}\to \gamma_2$ then one can see the the diagonal of the Hadamard product part of \eqref{eq:ntk} dominates and this part actually concentrates to a scalar matrix. Thus, the limiting eigenvalue distribution of this model goes to a deterministic shift of the eigenvalue distribution of the \emph{conjugate kernel} \cites{adlam2020neural, fan2020spectra}
\begin{equation}\label{eq:ck}
K^{\mathrm{CK}}=\frac{1}{p}\sigma\left(\frac{1}{\sqrt{d}}XW\right)\sigma\left(\frac{1}{\sqrt{d}}XW\right)^\top.
\end{equation}
However, the part involving the Hadamard product gets a nontrivial eigenvalue distribution in the \emph{quadratic} scaling $\frac{n}{pd}\to \gamma_1$ and $\frac{p}{d}\to\gamma_2$ which is the scaling chosen in this paper. An easier model involving the Hadamard product of two independent sample covariance matrices have been studied in this regime in \cite{abou2025eigenvalue}. The first part involving the conjugate kernel is a low-rank matrix, with a rank of order $\O{\sqrt{n}}$, so that it will not impact the bulk distribution. However, the \emph{mini-bulk} coming from the diverging number of outliers can be interesting in their own right but their study is outside of the scope of this article. 

This quadratic scaling coincides with the interpolation threshold where the number of parameters, here $p(d+1)$ is of the same order as the number of samples $n$. Around this threshold, the training error, coming from optimization with respect to the training data, goes to zero but an interesting phenomenon arises for the generalization error or test error. It exhibits a non-monotonic relationship with model complexity, also called a \emph{double descent} \cites{neyshabur2015, zhang2021, belkin2019} or \emph{multiple descent} \cites{liang2020multiple,canatar2021spectral}. As the number of parameters increases, test error initially decreases, then increases, coming from the classical bias-variance tradeoff, but then surprisingly decreases again as the model becomes heavily overparameterized. This second descent occurs around the point where the model has just enough capacity to perfectly fit the training data i.e the interpolation threshold. Beyond this threshold, further increasing model width continues to improve generalization, contradicting classical statistical learning theory. This phenomenon is now theoretically understood for several related models \cites{belkin2020two, hastie2022surprises,mei2022generalization,montanari2022interpolation,schroder2023deterministic, latourelle2023dyson,montanari2025} but still open for general multi-layered neural networks. The authors direct the interested reader to a recent review \cite{misiakiewicz2024six}.

This paper falls within the scope of nonlinear random matrix theory, which focuses on random matrix models where a nonlinearity is applied entrywise. This area has attracted considerable attention in recent years, particularly due to its applications in statistics and machine learning. The first eigenvalue distribution of such a model was studied in \cites{elkaroui, chengsinger, do2013spectrum} and was of the form $\sigma(XX^\top)$ with $X$ being i.i.d, $\sigma$ applied entrywise, and in a linear scaling of dimension. Further work includes the analysis of the largest eigenvalue \cites{fan2019spectral} or the study of polynomial scalings of dimension in \cites{misiakiewicz2022spectrum,lu2022equivalence,dubova2023universality, pandit2024universality, kogan2024extremal}. The conjugate kernel as in \eqref{eq:ck} has also seen substantial study both in linear scaling of dimension for the eigenvalue distribution \cites{louart2018random, pennington2017nonlinear,benigni2021eigenvalue,fan2020spectra, piccolo2021analysis,chouard2023deterministic,latourelle2023dyson,guionnet2025global}, analysis of the largest eigenvalue \cites{benigni2022largest,wang2024nonlinear} or in the polynomial scaling \cite{hu2024asymptotics} and overparametrized scaling \cite{wang2024deformed}. The eigenvalue distribution of the neural tangent kernel as in \eqref{eq:ntk} has been studied in the linear scaling of dimension in \cite{fan2020spectra} or in the overparametrized case in \cite{wang2024deformed}. While these three models, of somewhat increasing complexity, form the bulk of the literature, other interesting nonlinear models have been studied \cites{liao2021,guionnet2023spectral, feldman2025spectral}.

\section{Main model and results}

Since we only consider the Hadamard product part of \eqref{eq:ntk} which only involves $\sigma'=\phi$, for ease of notation, we consider the following model, 
\[
K \coloneqq \frac{1}{d}XX^\top \odot \frac{1}{p}Y_\phi D^2 Y_\phi^\top \quad\text{with}\quad Y_\phi\coloneqq \phi\left(\frac{1}{\sqrt{d}}XW\right)   
\]
where $X\in\Rbb^{n\times d}$, $W\in\Rbb^{d\times p}$, $\phi:\Rbb\to\Rbb$ is a function applied entrywise, $D=\mathrm{Diag}(\mathbf{a})$ where $\mathbf{a}\in\Rbb^p$ and the Hadamard product is defined by 
\[
(A\odot B)_{ij}=a_{ij}b_{ij}.    
\] 
We work in the setting of the following \emph{quadratic scaling}.
\begin{ass}[Quadratic scaling]
We suppose that there exists $\gamma_1,\gamma_2>0$ such that 
\[
\frac{n}{dp}\xrightarrow[n\to\infty]{}\gamma_1\quad \text{and} \quad\frac{p}{d}\xrightarrow[n\to\infty]{}\gamma_2.
\]
\end{ass}
We suppose that the data is unstructured and consider i.i.d.\,random variables for each independent samples. Deterministic data, or structured data, may be more realistic but is outside of the scope of this paper but note that we do not need the data to be Gaussian but more generally distributed.
\begin{ass}[Data assumptions]
The entries of $X$ are i.i.d.\,centered random variables of variance $1$\footnote{The result can be easily generalized to a general variance, for simplicity we develop the proof with unit variance} such that $\mu_{4,x}\coloneqq\E[x_{11}^4]$ exists and we suppose that for any $\varepsilon,\theta>0$,
\begin{equation}\label{eq:conchi2}
      \P\left(  
        \left\vert \frac{1}{d}\Vert \x_i\Vert_2^2-1\right\vert
        \geqslant \frac{d^\varepsilon}{\sqrt{d}}
        \right)
        \leqslant N^{-\theta}.
    \end{equation} 
\end{ass} 

We now give our assumption on our network weights. We note that we do need the weights of $W$ to be Gaussian.
\begin{ass}[Weight assumptions] 
  We suppose that $X$, $W$, and $D$ are independent random matrices such that the entries of $W$ are i.i.d.\ $\mathcal{N}(0,1)$ and the entries of $\mathbf{a}$ are i.i.d.\,centered, bounded random variables.  We denote by $\nu$ the law of $a_{11}^2$.
\end{ass}
Finally, we give the assumption on $\phi$ corresponding to the derivative of the activation function $\sigma$.
\begin{ass}[Activation function assumptions]
    We suppose that $\phi$ is a pseudo-Lipschitz function: there exists $\alpha>0$ and $L>0$ such that for any $x,y\in\Rbb$
    \[
      \left\vert
        \phi(x)-\phi(y)
      \right\vert
      \leqslant 
      L\vert x-y\vert \left(
        1+\vert x\vert^\alpha+\vert y\vert^\alpha
      \right).
    \]
\end{ass}

To formulate our main result, we need to introduce the notion of the Marchenko--Pastur map.
\begin{defn}
  For a probability measure $\nu$ on $[0,\infty)$ and a $\gamma >0$, the \emph{Marchenko--Pastur map $\mu_{\mathrm{MP}}^{\gamma}\boxtimes \nu$} is the free multiplicative convolution of the Marchenko--Pastur distribution of shape $\gamma$ with $\nu$.  Its Stieltjes transform is given for any $z\in\C_+$ by the unique solution in $\C_+$ to the equation
  \begin{equation}\label{eq:mp_map}
  s(z) = \int_\Rbb \frac{1}{t(1-\gamma(1+zs(z)))-z}\d \nu(t).
  \end{equation}
\end{defn}
\noindent The Marchenko--Pastur map is  (informally) the mapping from the spectrum of the population covariance matrix to the spectrum of the empirical covariance matrix of i.i.d samples, where the ratio of features to samples is $\gamma$.  More precisely:
\begin{thm}[\cite{baisilverstein}*{Theorem 4.3}]\label{theo:baisilverstein}
  Suppose that $A \succeq 0$ and $B \succeq$ are $a \times a$ and $b \times b$ matrices, and that $Y$ is an $a \times b$ random matrix with i.i.d standard variables having all moments.  Then if $a \to \infty$ and $b \to \infty$ with $a/b \to \gamma \in (0,\infty)$, and if the e.e.d of $A$ and of $B$ both converge to $\nu$, then the e.e.d of the sample covariance matrix\footnote{Here $b$ represents the number of samples, and $a$ represents the number of features.} $\frac{1}{b}A^{1/2} YY^{\top} A^{1/2}$ converges weakly almost surely to $\mu_{\mathrm{MP}}^{\gamma}\boxtimes \nu$.   The e.e.d of the Gram matrix $\frac{1}{a} Y B Y^{\top}$ converges weakly almost surely to $\mu_{\mathrm{MP}}^{1/\gamma}\boxtimes \nu$.
\end{thm}

We are now ready to give our main result giving a description of the eigenvalue distribution of $K$.
\begin{thm}\label{theo:main}
  There exists a deterministic measure $\mu_{\nu,\varphi}$ such that the empirical spectral distributions of $K$ and $K^{\mathrm{NTK}}$ converge weakly in probability to 
  \(
  \mu_{\mathrm{MP}}^{\gamma_1}\boxtimes \mu_{\nu,\varphi},
  \)
  the Marchenko--Pastur map with shape $\gamma_1$ applied to $\mu_{\nu,\varphi}$.
\end{thm}
\noindent We note that $\gamma_1$ is expressed as a limiting ratio of samples to features (hence opposite the convention in Theorem \ref{theo:baisilverstein}), but the NTK naturally takes the form of a Gram matrix, for which reason we arrive at $\mu_{\mathrm{MP}}^{\gamma_1} \boxtimes \nu$. 

The description of $\mu_{\nu,\varphi}$ in full generality can be found in Proposition \ref{prop:stieltjes}. However, there are specific cases where we can fully describe the measure in terms of elementary operations on measures. We give now a corollary describing the limiting measure in the case where $D^2=\mathrm{Id}_p$ or when $\varphi(x)$ is linear.



\begin{figure}[t!]
  \centering
  \begin{subfigure}{0.45\textwidth}
    \centering
    \includegraphics[width=0.9\textwidth]{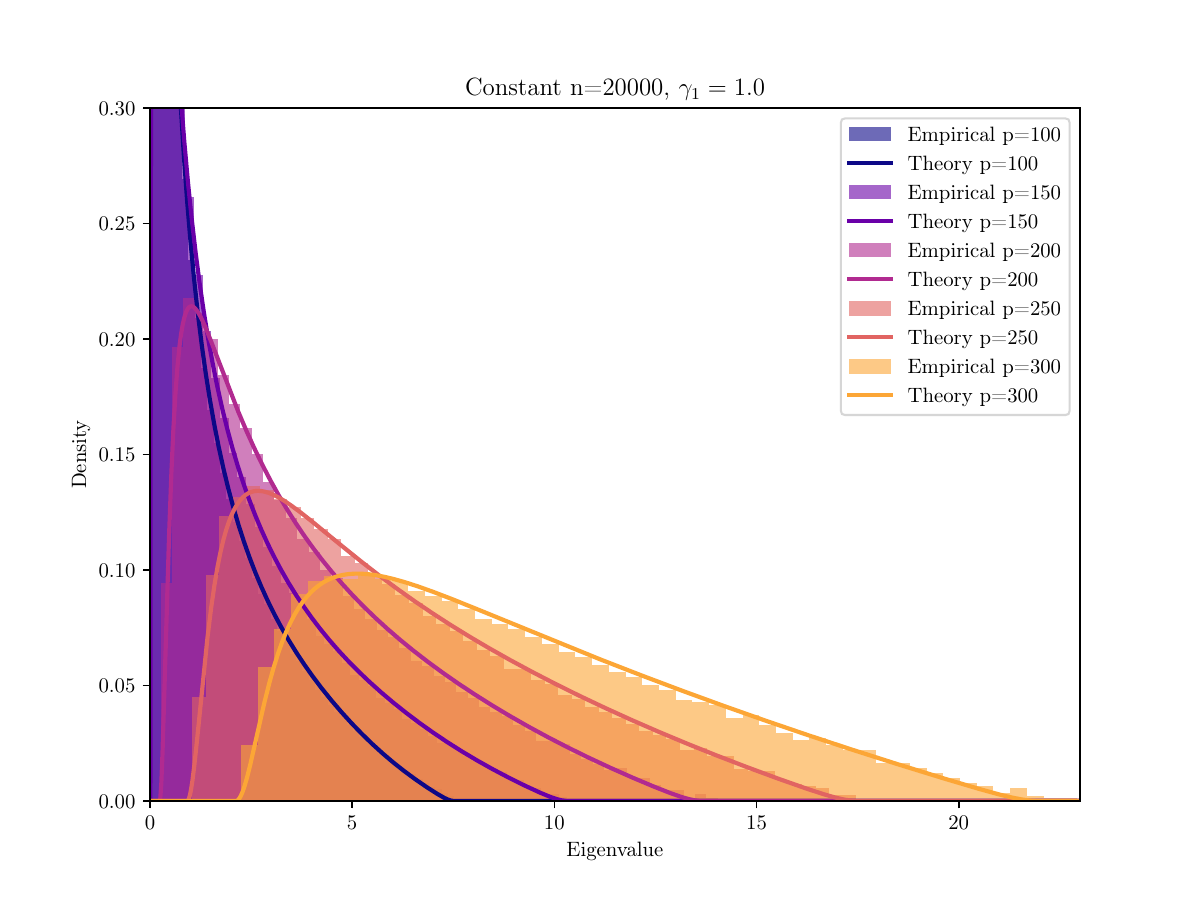}
    \caption{$\varphi(x)=x, \nu=\delta_1.$}
    \label{fig:ntk_gamma1_sweep}
  \end{subfigure}
  \begin{subfigure}{0.45\textwidth}
    \centering
    \includegraphics[width=0.9\textwidth]{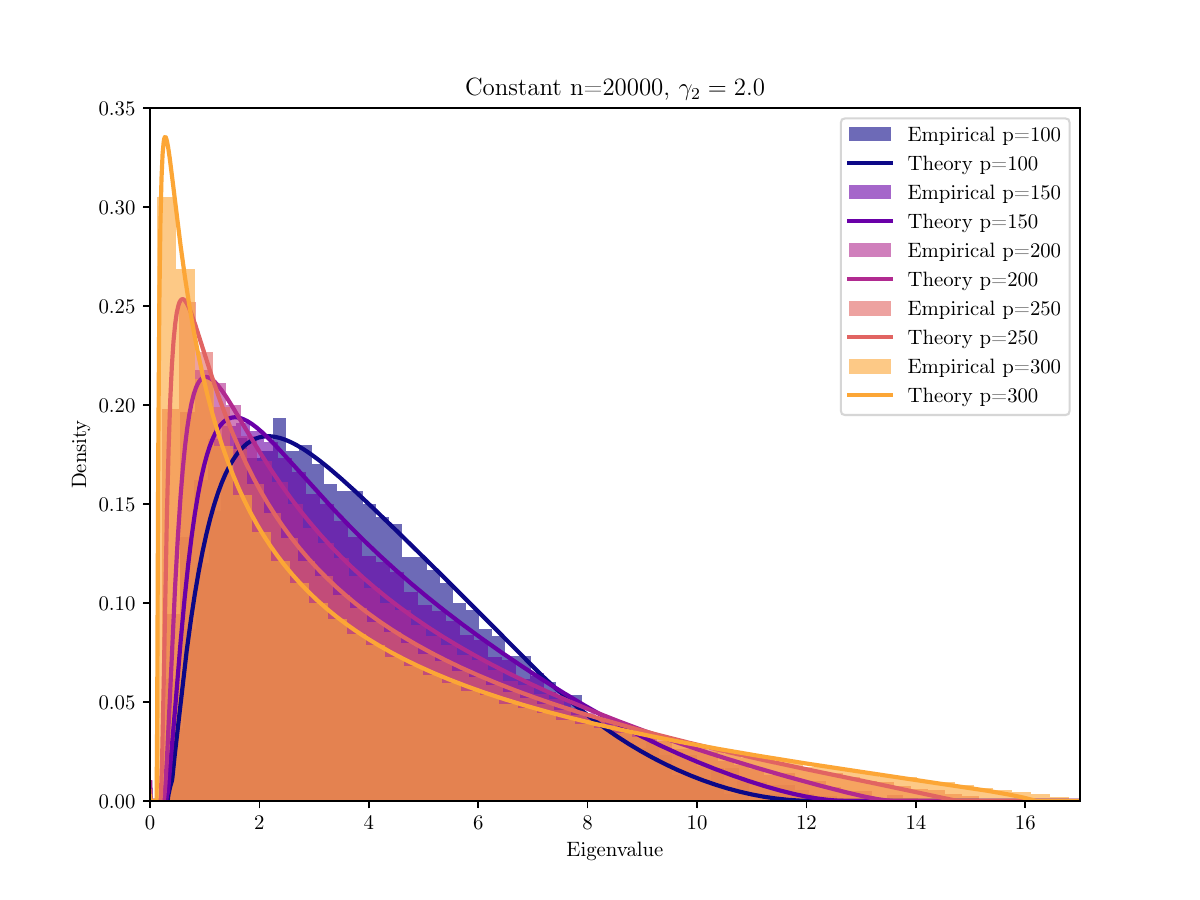}
    \caption{$\varphi(x)=x, \nu=\delta_1.$}
    \label{fig:ntk_gamma2_sweep}
  \end{subfigure}
  \begin{subfigure}{0.45\textwidth}
    \centering
    \includegraphics[width=0.9\textwidth]{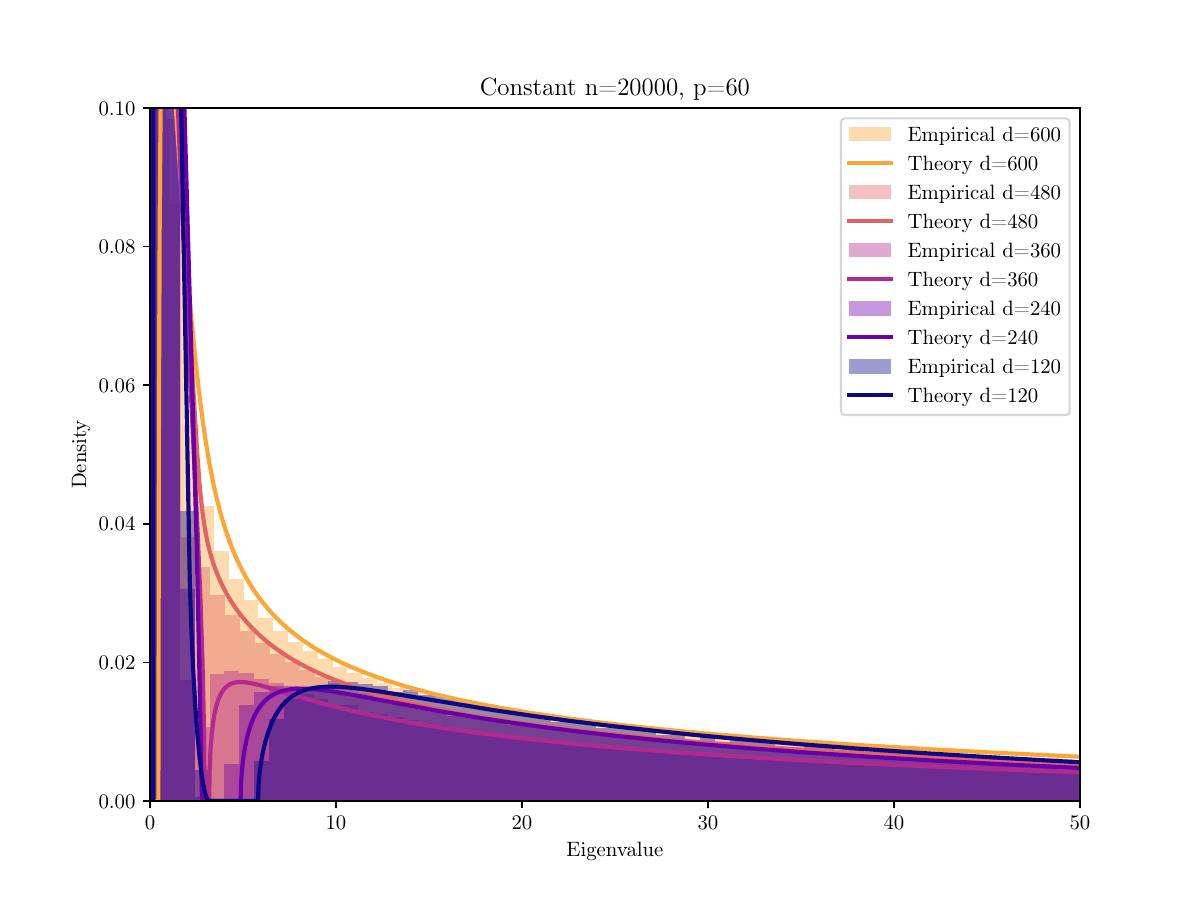}
    \caption{$\varphi(x)=x, \nu=\frac{1}{2}\delta_1 + \frac{1}{2}\delta_{30}.$}
    \label{fig:ntk_HH}
  \end{subfigure}
  \begin{subfigure}{0.45\textwidth}
    \centering
    \includegraphics[width=0.9\textwidth]{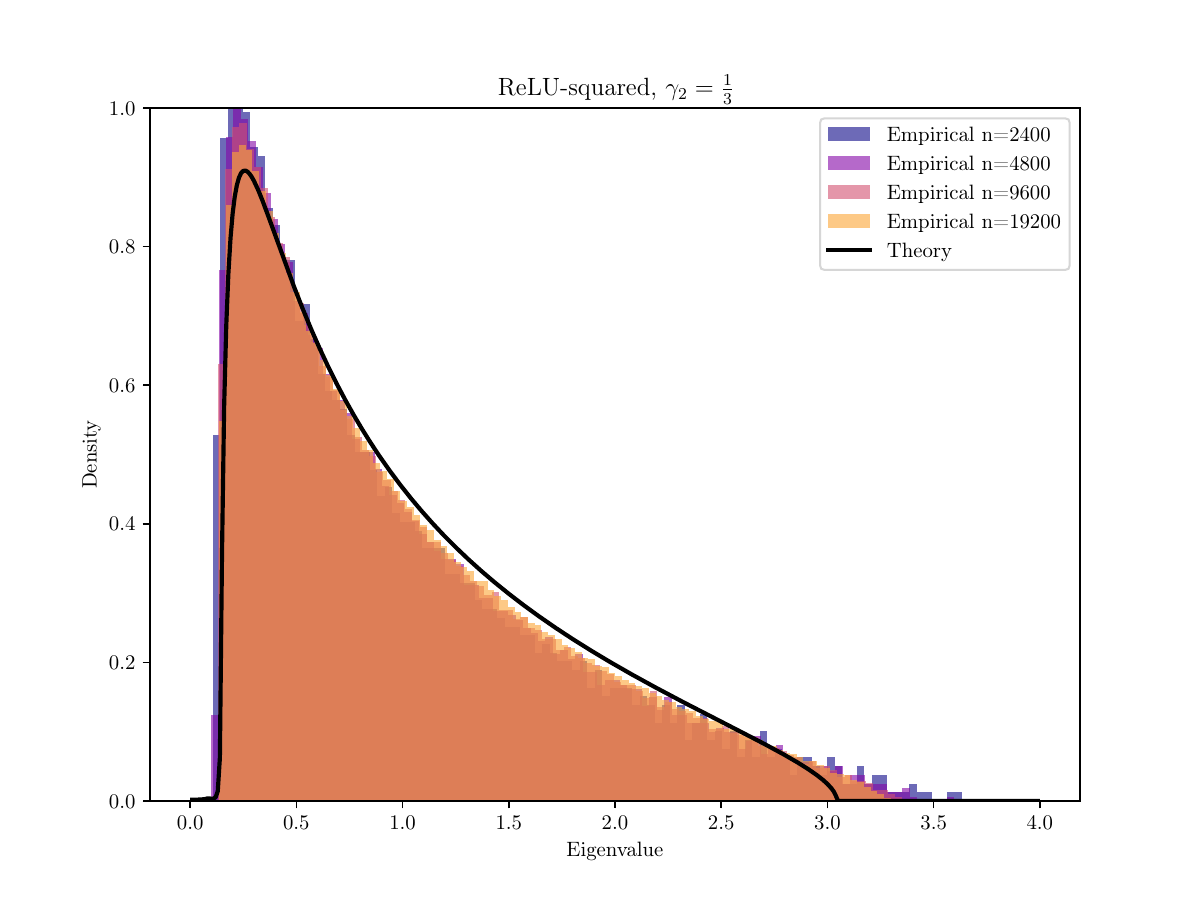}
    \caption{$\varphi(x)=\min\{x,0\}, \nu=\delta_1.$}
    \label{fig:ntk_relu_squared}
  \end{subfigure}
  \caption{Eigenvalue density comparisons for the model \eqref{eq:ntk} over different parameter sweeps.  The first two figures show varying aspect ratios for the quadratic activation $\sigma(x)=\frac{1}{2}x^2$ with $\nu=\delta_1$.  The third figure shows the effect of nontrivial $\nu$ -- the nontrivial disconnected support of the spectrum is inherited for $n/dp$ and $d/p$ sufficiently large.  The fourth shows the effect of nontrivial activation, here $\sigma(x)=\frac{1}{2}(\min\{x,0\})^2$ with $\nu=\delta_1$.}
  \label{fig:ntk_eigenvalue_density_subfigs}
\end{figure}

\begin{corollary}\label{cor:main}
 Suppose that $\nu=\delta_1$ or that $\varphi(x)= a x+b$ then we have that the limiting spectral distribution of $K$ is given by 
 \[
  \mu_{\mathrm{MP}}^{\gamma_1}
  \boxtimes
  \mathrm{Law}\left( \alpha_\phi^2 \chi + \beta_\psi^2\right)
  \quad\text{where}\quad
  \mathrm{Law}(\chi) = \frac{\gamma_2}{2} (\mu_{\mathrm{MP}}^{\gamma_2} \boxtimes \nu) * (\mu_{\mathrm{MP}}^{\gamma_2} \boxtimes \nu) + (1-\gamma_2) (\mu_{\mathrm{MP}}^{\gamma_2} \boxtimes \nu) + \frac{\gamma_2}{2} \delta_0.
 \]
 where $\ast$ denotes the classical convolution of measures, and we define 
 \[
 \alpha_\phi \coloneqq \E_{\mathcal{N}(0,1)}\left[z\varphi(z)\right],
 \quad\text{and}\quad
 \beta_\psi^2 \coloneqq \E_{\mathcal{N}(0,1)}\left[\psi(z)^2\right],
 \quad \text{where} \quad \psi(x) \coloneqq \varphi(x) - \alpha_\phi x - \mathbb{E}_{\mathcal{N}(0,1)}\left[\varphi(z)\right].
 \]
\end{corollary}
Note that while this compact form allows us to give one formula for the two distinct cases, for $\nu=\delta_1$, the law of $\chi$ simplifies to  
\[
  \mathrm{Law}(\chi) = \frac{\gamma_2}{2} (\mu_{\mathrm{MP}}^{\gamma_2} * \mu_{\mathrm{MP}}^{\gamma_2}) + (1-\gamma_2) \mu_{\mathrm{MP}}^{\gamma_2} + \frac{\gamma_2}{2} \delta_0,
\]
and in the case where $\varphi(x)=a x$ then we have $\alpha_\phi^2=a^2$ and $\beta_\psi^2=0$ and we can dispense with the shift, i.e.\ 
\[
  \mu_{\mathrm{MP}}^{\gamma_1}
  \boxtimes
  \mathrm{Law}\left( \alpha_\phi^2 \chi\right)
  \quad\text{where}\quad
  \mathrm{Law}(\chi) = \frac{\gamma_2}{2} (\mu_{\mathrm{MP}}^{\gamma_2} \boxtimes \nu) * (\mu_{\mathrm{MP}}^{\gamma_2} \boxtimes \nu) + (1-\gamma_2) (\mu_{\mathrm{MP}}^{\gamma_2} \boxtimes \nu) + \frac{\gamma_2}{2} \delta_0.
\]
In the general case, the shift noncommutative analogue of $a_\phi^2 \chi + \beta_\psi^2$ has a shift term that depends $D$ which no longer commutes with the term represented by $\chi$; this leads to a nontrivial coupling which distorts the free convolution $\boxtimes$, and so the formula is false in general. Nevertheless, we provide a free probabilistic characterization of the measure by computing its moments in Subsection \ref{subsec:freeproba}.

We note that $\psi$ contains the contributions of all the Gaussian chaoses beyond the first order, with the first order chaos of $\phi$ leading to the special convolution structure, and the zero-th order producing a low-rank `minibulk' term, much like the conjugate kernel.  Moreover, all contributions to the NTK from chaoses of higher than the first order produce the same effect, mirroring what is seen in the conjugate kernel and for kernel-inner product matrices \cites{hastie2022surprises,mei2022generalization,dubova2023universality}.  The convolutional structure that appears in the first-order term, however, does not appear in the conjugate kernel or kernel-inner product matrices.  It can be seen to arise from a fundamental algebraic structure built into the covariance of the NTK (see especially Lemma \ref{lem:spectrumQhat}).  It can also be seen to appear in other statistical problems, such as the matrix least squares problem setup of \cite{fan2024kronecker}.  We also note that the structure of the measure $\mu_{\nu,\varphi}$ in the general case bears strong similarities to the general form of the solution in \cite{fan2024kronecker} (see \eqref{eq:frepresentation}).

We now give a sketch of the proof of Theorem \ref{theo:main}.
\paragraph{Step 1: Handling the nonlinearity.} In virtually every model of nonlinear random matrix theory, at the level of global eigenvalue distribution, the nonlinearity should be considered as ``linear + noise''. To make this idea formal we start by decomposing the function as follows,
\[
\phi(x)  =  \mathbb{E}_{\mathcal{N}(0,1)}\left[\phi(z)\right] + \alpha_\phi x + \psi(x)
\quad\text{where}\quad 
\alpha_\phi \coloneqq 
\E_{\mathcal{N}(0,1)}\left[
  z\phi(z)
  \right]
\]
In particular we have that $\psi$ is orthogonal to the identity and constant map with respect to the Gaussian measure,
\[
  \E_{\mathcal{N}(0,1)}\left[\psi(z)\right]=\E_{\mathcal{N}(0,1)}\left[z\psi(z)\right] =0
  \quad\text{and}\quad
  \beta_\psi^2 \coloneqq \E_{\mathcal{N}(0,1)}\left[
    \psi(z)^2
    \right].
\]
Our original model can thus be rewritten as 
\[
K = \frac{1}{d}XX^\top \odot 
\frac{1}{p}
\left(
  \left(\frac{\alpha_\phi}{\sqrt{d}} XW + Y_\psi + c\right)
  D^2
  \left(\frac{\alpha_\phi}{\sqrt{d}} XW + Y_\psi + c\right)^{\top}
  \right).    
\]
We then show in Section \ref{sec:nonlinearity} that the eigenvalue distribution of $K$ is asymptotically the same as the eigenvalue distribution of 
\[
\widetilde{K} = \frac{1}{d}XX^\top \odot 
\frac{1}{p}
\left(
  \left(\frac{\alpha_\phi}{\sqrt{d}} XW + \psi(\widetilde{X})\right)
  D^2
  \left(\frac{\alpha_\phi}{\sqrt{d}} XW + \psi(\widetilde{X})\right)^{\top}
  \right)
\]
where $\widetilde{X}\in\Rbb^{n\times p}$ is an independent i.i.d $\mathcal{N}(0,1)$ matrix. The model thus reduces to a linear model with added independent noise. However, this transformation does not eliminate the dependence between the two matrices in the Hadamard product, as both involve the data matrix $X$.
\paragraph{Step 2: The NTK is a Gram matrix.}
We rewrite the entries of this Hadamard product as an inner product  on $\Rbb^d\otimes \Rbb^p \simeq \Rbb^{dp}$, indeed we have 
\[
\tilde{k}_{ij} = 
\frac{1}{pd}\langle \x_i,\x_j\rangle_{\Rbb^d} 
\langle D \tilde{\y}_{\psi,i},D\tilde{\y}_{\psi,j}\rangle_{\Rbb^p}  
\eqqcolon 
\frac{1}{pd}\langle \x_i \otimes D\tilde{y}_{\psi,i},\, \x_j\otimes D\tilde{y}_{\psi,j}\rangle_{\Rbb^d\otimes \Rbb^p}.     
\]
where $\x_i$ and $\tilde{\y}_{\psi,i}$ are respectively the $i$-th rows of $X$ and $\widetilde{Y}_\psi = \frac{\alpha_\phi}{\sqrt{d}}XW+\psi(\widetilde{X})$. Thus, if we denote 
\begin{equation}\label{eq:defomega}
\tilde{\omega}^i = \x_i \otimes D\tilde{\y}_{\psi,i} \in \Rbb^{d}\otimes \Rbb^p    
\end{equation}
our model becomes a Gram matrix of the family $(\tilde{\omega}^i)_{1\leqslant i\leqslant n}$ i.e. 
\[
\tilde{k}_{ij} = \frac{1}{pd}\langle \tilde{\omega}^i,\tilde{\omega}^j\rangle.    
\]
Using this formalization, we can apply the result of Bai--Zhou \cite{baizhou} on the limiting eigenvalue distribution of sample covariance matrices with independent rows. This yields that the limiting e.e.d of $\widetilde{K}$ (and thus of $K$) conditionally on $W$ and $D$ is given by the Marchenko--Pastur map $\mu_{\mathrm{MP}}^{\gamma_1}\boxtimes \pi$ where $\pi$ is the limiting eigenvalue distribution of the conditional covariance tensor 
\[
\mathrm{Cov}(\tilde{\omega}^i,\tilde{\omega}^j)_{qrst}
=
\E\left[
  \left(
    \tilde{\omega}^i_{qr}-\E\left[\tilde{\omega}^i_{qr}\middle\vert W,D\right]
  \right)
  \left(
    \tilde{\omega}^j_{st}-\E\left[\tilde{\omega}^j_{st}\middle\vert W,D\right]
  \right)
  \middle\vert W,D
\right].
\] 
To apply this theorem, we must verify several assumptions from the main result of \cite{baizhou}, primarily a concentration property for quadratic forms. This verification is carried out in Section \ref{sec:baizhou}.
\paragraph{Step 3: Analysis of the covariance tensor.} Having established the Marchenko--Pastur map, we now seek to understand the limiting measure $\pi$. In the special cases from Corollary \ref{cor:main}, the full spectrum and eigenvectors for finite $n$ can be computed in terms of the singular values of $WD$, which explains why we obtain the explicit limit given by the classical convolution of two Marchenko--Pastur measures. In the general case, this direct computation is no longer feasible, so we characterize the distribution through moment calculations. We show that the special cases involving classical convolution can be recovered from this moment formula, and we provide a free probabilistic characterization based on the moment computation. This analysis is presented in Section \ref{sec:analysisQ}.

\subsection*{Acknowledgements}

The authors would like to thank Yue Lu for helpful conversations related to the problem and to Jeffrey Pennington for suggesting the problem, which corresponds to the ``linear scaling regime'' in \cite{adlam2020neural}.

\section{Screening low-rank factors and decoupling the nonlinearity}\label{sec:nonlinearity}
Our first goal is to remove unwanted terms in $K$, and the higher order terms contributed from the nonlinearity from the data matrix $X$.
In short, we would like to reduce the problem to studying a new random matrix $\widetilde{K}$, defined as follows. Let $\widetilde{X}\in\Rbb^{n\times d}$ be a matrix with i.i.d $\mathcal{N}(0,1)$ entries independent of $X$, $D$, and $W$ and consider the model 
  \begin{equation}\label{eq:defwidetildeK}
        \widetilde{K} =
        \frac{1}{d}XX^\top \odot 
\frac{1}{p}
\left(
  \left(\frac{\alpha_\phi}{\sqrt{d}} XW + \psi(\widetilde{X})\right)
  D^2
  \left(\frac{\alpha_\phi}{\sqrt{d}} XW + \psi(\widetilde{X})\right)^{\top}
  \right)
  \end{equation}

\begin{prop}\label{prop:replace}
    Let $\widetilde{K}$ be as in \eqref{eq:defwidetildeK} then we have that the eigenvalue distribution of $K$ is asymptotically the same as the eigenvalue distribution of $\widetilde{K}$.
    Let $z\in\C_+$ then for any $\varepsilon>0$, we have 
    \[
    \P\left(
      \left\vert 
        \frac{1}{n}\tr (K-z)^{-1}-\frac{1}{n}\tr(\widetilde{K}-z)^{-1}
      \right\vert
      \geqslant \frac{d^\varepsilon}{\sqrt{d}}
      \right)
    \xrightarrow[n\to\infty]{}0.    
    \]
    The same holds for $K^{\mathrm{NTK}}$.
\end{prop}
We first consider the adjoint tensor $\widehat{\mathcal{L}}$ of $K$ defined by
\[
\widehat{\mathcal{L}} = \frac{1}{n} \sum_{i=1}^n \widehat{\omega}^i\otimes\widehat{\omega}^i
\quad\text{with}\quad 
\widehat{\omega}^i \coloneqq \x_i\otimes D\y_{\phi,i} = \x_i\otimes D\phi\left(\frac{1}{\sqrt{d}}W^\top \x_i\right).
\]
Now we have two steps to prove the proposition.  First, we should discard the contribution of the mean term of $\varphi$, which will leverage that it is low-rank; this is quick and simple.  Second, we employ a replacement strategy using the Sherman--Morrison formula to swap summands of a $\widehat{\mathcal{L}}$--like matrix one at a time; this constitutes the bulk of the work.

\subsection*{Discarding low rank factors}

For the low rank factors, we need start with two simple deterministic Stieltjes transform estimates.
\begin{lem}\label{lem:stieltjes}
  Suppose $A,B$ are symmetric matrices $n \times n$ matrices.  Then for any $z \in \C_+$ we have 
  \[
  \left|\frac{1}{n}\tr(A-z)^{-1} - \frac{1}{n}\tr(B-z)^{-1}\right| 
  \leq  \frac{\pi}{n \Im z} \operatorname{rank}(A-B).
  \]
  Similarly if $A,B$ are $p \times n$ matrices, then we have 
  \[
  \left|\frac{1}{p}\tr(AA^{\top}-z)^{-1} - \frac{1}{p}\tr(BB^{\top}-z)^{-1}\right| 
  \leq  \frac{\pi}{p \Im z} \operatorname{rank}(A-B).
  \]
\end{lem}
\begin{proof}
  See \cite{baisilverstein}*{Theorem A.43,44} which give bounds on the distribution functions $F^A$ and $F^B$ of the e.e.d of $A$ and $B$ in either case.  In particular (in the symmetric case), they show that 
  \[
  \| F^A - F^B \|_\infty \leq \frac{1}{n} \operatorname{rank}(A-B).
  \]
  Note that by Lebesgue-Stieltjes integration by parts, we can represent the Stieltjes transform of $A$ as 
  \[
  \frac{1}{n}\tr(A-z)^{-1} = \int_\R \frac{1}{x-z} \d F^A(x) = \int_\R \frac{1}{(x-z)^2} F^A(x) \d x.
  \]
  Hence taking the difference of the Stieltjes transforms and bounding produces the claimed result.
\end{proof}

Now, the difference of $K^{\mathrm{NTK}}$ and $K$ is given by $K^{\mathrm{CK}}$, the conjugate kernel \eqref{eq:ck}, which is rank $p$.  Hence we have that
\[
\left|\frac{1}{n}\tr(K^{\mathrm{NTK}}-z)^{-1} - \frac{1}{n}\tr(K-z)^{-1}\right| 
\leq  \frac{p \pi}{n \Im z}.
\]
Similarly, we can remove the contribution of the mean term of $\varphi$ by introducing 
\[
{\omega}^i \coloneqq \x_i \otimes D\left(\frac{\alpha_\phi}{\sqrt{d}}W^\top \x_i + \psi\left( \frac{1}{\sqrt{d}}W^\top \x_i\right)\right),
\quad\text{so that}\quad
\widehat{\omega}^i - {\omega}^i = c \x_i \otimes (D \mathbf{1}),
\]
with $c$ the mean of $\varphi$.
We introduce a matrix 
\[
\mathcal{L} = \frac{1}{n} \sum_{i=1}^n {\omega}^i\otimes{\omega}^i.
\]
Then $\widehat{\mathcal{L}}$ and $\mathcal{L}$ can be expressed as 
\[
\widehat{\mathcal{L}} = \widehat{\Omega}\widehat{\Omega}^\top, \quad \mathcal{L} = {\Omega}{\Omega}^\top,
\]
where ${\Omega}$ is the matrix with rows ${\omega}^i$ and $\widehat{\Omega}$ is the matrix with rows $\widehat{\omega}^i$.  The difference of ${\Omega}$ and ${\omega}^i$ has rank at most the rank of $X$, which is $d$. Thus from Lemma \ref{lem:stieltjes}, we have that 
\[
\left|\frac{1}{n}\tr(\widehat{\mathcal{L}}-z)^{-1} - \frac{1}{n}\tr(\mathcal{L}-z)^{-1}\right| 
\leq  \frac{d \pi}{n \Im z}.
\]
Hence, it suffices to show Proposition \ref{prop:replace} with $K$ and $\widetilde{K}$ replaced by $\mathcal{L}$ and $\widetilde{\mathcal{L}}$, where 
\[
\widetilde{\mathcal{L}} \coloneqq \frac{1}{n} \sum_{i=1}^n \widetilde{\omega}^i\otimes\widetilde{\omega}^i
\quad\text{with}\quad
\wt{\omega}^i \coloneqq \x_i \otimes D\wt{\y}_{\psi,i}
=
 \x_i \otimes D\left(
    \frac{\alpha_\phi}{\sqrt{d}}W^\top \x_i + \psi\left(\wt{\x}_i\right)
    \right).
\]

\subsection*{Performing the row replacement}

Using the Sherman--Morrison formula we obtain the following lemma.
\begin{lem}\label{lem:sherman}
  Let $i\in\unn{1}{n}$, and define the matrices
  \[
  \mathcal{L}^{i,0} = \mathcal{L} - \frac{1}{n}\omega^i\otimes\omega^i 
  \quad\text{and}\quad 
  \mathcal{L}^{i,\wt{\omega}} = \mathcal{L}^{i,0}+\frac{1}{n}\wt{\omega}^i\otimes\wt{\omega}^i
  \]
  then we have 
  \begin{align*}
  (\mathcal{L}-z)^{-1}& = (\mathcal{L}^{i,0}-z)^{-1} - \frac{1}{n} \frac{(\mathcal{L}^{i,0}-z)^{-1}\omega^i\otimes \omega^i (\mathcal{L}^{i,0}-z)^{-1}}{1+\frac{1}{n}\langle \omega^i\otimes \omega^i,(\mathcal{L}^{i,0}-z)^{-1}\rangle}  
  \end{align*}
  and thus 
  \begin{multline*}
    (\mathcal{L}-z)^{-1}-(\mathcal{L}^{i,\tilde{\omega}}-z)^{-1} \\
    =  \frac{1}{n}(\mathcal{L}^{i,0}-z)^{-1}
    \left( 
      \frac{\omega^i\otimes\omega^i}{1+\frac{1}{n}\langle \omega^i\otimes \omega^i,(\mathcal{L}^{i,0}-z)^{-1}\rangle}- \frac{\wt{\omega}^i\otimes \wt{\omega}^i}{1+\frac{1}{n}\langle \wt{\omega}^i\otimes \wt{\omega}^i,(\mathcal{L}^{i,0}-z)^{-1}\rangle} 
    \right)
    (\mathcal{L}^{i,0}-z)^{-1} 
  \end{multline*}
\end{lem}
  We see that we need to understand quadratic forms in $\omega^i \otimes \omega^i$ and $\wt{\omega}^i\otimes \wt{\omega}^i$. We can write the following lemma 
  \begin{lem}\label{lem:comparewwtilde}
    Let $i\in\unn{1}{n}$ and $\mathcal{A}$ be a $\Rbb^{d}\otimes \Rbb^p\otimes \Rbb^d\otimes  \Rbb^{p}$ tensor with bounded norm $\Vert \mathcal{A}\Vert_{\mathrm{op}}\lesssim 1$, for any $\varepsilon>0$,
    \[
    \P\left(
      \left\vert
        \frac{1}{n}\langle \omega^i\otimes \omega^i-\wt{\omega}^i\otimes \wt{\omega}^i,\mathcal{A}\rangle 
      \right\vert
      \geqslant \frac{d^\varepsilon}{\sqrt{d}}
    \right)  \xrightarrow[n\to\infty]{}0.
    \]
  \end{lem}
  To prove this lemma, we show concentration and convergence of the expectations. 

  \begin{lem}\label{lem:compareexpect}
    There exists $C>0$ such that for all $i\in\unn{1}{n}$ and $\varepsilon>0$, we have 
    \begin{equation}
    \left\Vert
      \E\left[\omega^i\otimes \omega^i\right]-\E\left[\wt{\omega}^i\otimes \wt{\omega}^i\right]
    \right\Vert_{\mathrm{F}}^2\leqslant Cp d^{\epsilon}.  
    \end{equation}
  \end{lem}
  \begin{proof}
    We can write the difference of the expectations as three terms,
    \begin{align*}
      \E\left[\omega^i\otimes \omega^i\right]-\E\left[\wt{\omega}^i\otimes \wt{\omega}^i\right] 
      &=
      \E\left[\x_i \otimes D\left(\y_{\psi,i}-\psi(\wt{\x}_{i})\right)\otimes \x_i\otimes \frac{\alpha_\phi}{\sqrt{d}}DW^\top \x_i\right]\\
      & + \E\left[\x_i \otimes \frac{\alpha_\phi}{\sqrt{d}}DW^\top \x_i\otimes \x_i\otimes  D\left(\y_{\psi,i}-\psi(\wt{\x}_{i})\right)\right]\\
      &+ \E\left[\x_i \otimes D\y_{\psi,i}\otimes \x_i\otimes D\y_{\psi,i}\right] -\E\left[\x_i \otimes D\psi(\wt{\x}_{i})\otimes \x_i\otimes D\psi(\wt{\x}_{i})\right]. 
    \end{align*}
    The first two terms are bounded identically, for the first term, if we denote $\z_i = \frac{1}{\sqrt{d}}W^\top \x_i$, then we want to bound 
    \[
    \alpha_\phi\E\left[\x_i\otimes D\psi(\z_i)\otimes \x_i\otimes D\z_i\right]  
    \]
    since we have
    \[
    \E\left[\x_i\otimes D\psi(\wt{\x}_{i})\otimes \x_i \otimes D\z_i\right]=0  
    \]
    coming from the fact that $W$ is centered and independent of $X$ and $\wt{X}$. Let $q,s\in\unn{1}{d}$ and $r,t\in\unn{1}{p}$, then considering the $(q,r,s,t)$ entry of the expected tensor, we get 
    \[
    \alpha_\phi\E\left[x_{i,q}D_{rr}\psi(z_{i,r})x_{i,s}D_{tt}z_{i,t}\right] = \alpha_\phi\delta_{rt} \E\left[x_{i,q}x_{i,s}D_{rr}^2z_{i,r}\psi(z_{i,r})\right].  
    \]
    Conditionally on $X$ and $D$, we see that $z_{i,r}\sim \mathcal{N}\left(0,\frac{\Vert \x_i\Vert_2^2}{d}\right)$ and by \eqref{eq:conchi2}, we know that with overwhelming probability 
    \[
    \frac{\Vert \x_i\Vert_2^2}{d}=1+\O{\frac{d^\varepsilon}{\sqrt{d}}}.  
    \]
    We can then write that 
    \[
    \alpha_\phi \E\left[x_{i,q}x_{i,s}D_{rr}^2 z_{i,r}\psi(z_{i,r})\right] 
    =
    \alpha_\phi\sigma_D^2 \E\left[x_{i,q}x_{i,s}\E\left[z_{i,r}\psi(z_{i,r})\middle\vert \x_i\right]\right] 
    \eqqcolon
    \alpha_\phi\sigma_D^2 \E\left[x_{i,q}x_{i,s}f\left(\frac{\Vert \x_i\Vert^2_2}{d}\right)\right],
    \]
    where we can write $f(x) = \E\left[\mathcal{N}(0,x)\psi(\mathcal{N}(0,x))\right]$. For $q = s$, we get since $f(1)=0$, and $f$ is pseudo-Lipschitz, 
    \[
    \alpha_\phi\sigma_D^2 \E\left[x_{i,q}^2z_{i,r}\psi(z_{i,r})\right]
    \leqslant 
    \alpha_\phi \sigma_D^2 \sqrt{\E\left[x_{i,q}^4\right]}\sqrt{\E\left[f\left(\frac{\Vert \x_i\Vert_2^2}{d}\right)^2\right]}\leqslant \frac{d^{2\varepsilon}}{\sqrt{d}}. 
    \] 
    For $q\neq s$, we need a better bound due to the $d^2$ number of such terms, but we can write using the fact that $f$ is pseudo-Lipschitz, 
    \[
      \E\left[
        x_{i,q}x_{i,s}D_{rr}^2 f\left(\frac{\Vert \x_i\Vert_2^2}{d}\right)
      \right] 
      = 
      \sigma_D^2 \E\left[
        x_{i,q}x_{i,s} f\left(\frac{1}{d}\sum_{j\neq q,s}x_{i,j}^2\right)
      \right]
      +
      \O{\frac{d^\varepsilon}{d}}
      =
    \O{\frac{d^\varepsilon}{d}}.
    \]

    For the last term, consider the $(q,r,s,t)$ entry of the tensor, 
    \begin{multline*}
      \E\left[x_{i,q}x_{i,s}D_{rr}D_{tt}\left(\psi\left(z_{i,r}\right)\psi(z_{i,t})-\psi\left(\wt{x}_{i,r}\right)\psi(\wt{x}_{i,t})\right)\right]
      =
      \delta_{rt} \sigma_D^2 \E\left[x_{i,q}x_{i,s}\left(\psi\left(z_{i,r}\right)^2-\psi\left(\wt{x}_{i,r}\right)^2\right)\right]\\
      =
      \delta_{rt}\sigma_D^2 \E\left[x_{i,q}x_{i,s}\psi\left(z_{i,r}\right)^2\right] - \delta_{rt}\delta_{qs}\sigma_D^2 \E\left[\psi\left(\wt{x}_{i,r}\right)^2\right]
    \end{multline*}
    where we used the independence of $D$ and the independence of $\wt{\x}_i$ from $\x_i$.To control the first term we first write,
    \[
    \E\left[x_{i,q}x_{i,s}\psi\left(z_{i,r}\right)^2\right] 
    =
    \E\left[
      x_{i,q}x_{i,s}\E\left[
        \psi\left(
          z_{i,r}
        \right)^2
        \middle\vert
        \x
      \right]
    \right]
    \eqqcolon
    \E\left[
      x_{i,q}x_{i,s}
      g\left(
        \frac{1}{d}\sum_{j=1}^n x_{i,j}^2
      \right)
    \right].
    \]
    where we defined $g(x)=\E\left[\psi(\mathcal{N}(0,x))^2\right],$ since $\psi^2$ is pseudo-Lipschitz, $g$ is too and thus can write 
    \[
    \E\left[x_{i,q}x_{i,s}\psi\left(z_{i,r}\right)^2\right] 
    =
    \E\left[
      x_{i,q}x_{i,s}g\left(
        \frac{1}{d}\sum_{j\neq q,s}x_{i,j}^2
      \right)
    \right]
    +
    \O{\frac{d^\varepsilon}{d}}
    =
    \delta_{qs}\E\left[{\psi}\left(\wt{x}_{i,r}\right)^2\right]+\O{\frac{d^\varepsilon}{d}} 
    \]
    where we used the concentration bound from Lemma \eqref{eq:conchi2} and thus, we get that this last term is a $\O{d^{-1+\varepsilon}}$. Putting everything together we get that 
    \begin{multline*}
      \left\Vert
        \E\left[\omega^i\otimes \omega^i\right]-\E\left[\wt{\omega}^i\otimes \wt{\omega}^i\right]
      \right\Vert_{\mathrm{F}}^2
      =
      \sum_{q,s=1}^d\sum_{r,t=1}^p 
      \left\vert
        \E\left[
        \omega^i_{qr}\omega^i_{st}-\wt{\omega}^i_{qr}\wt{\omega}^i_{st}
        \right]
      \right\vert^2  
      =
      \O{
      \sum_{q,s=1}^d
      \sum_{r,t=1}^p
      \delta_{qs}\delta_{rt}\frac{d^{4\varepsilon}}{d} 
      +
      \delta_{rt}\frac{d^{2\varepsilon}}{d^2}
      }
      \\
      =
      \O{d^{4\varepsilon}p}
    \end{multline*}
  \end{proof}
Now that we have a bound on the expected value of the Frobenius norm, we prove that we have concentration of the quadratic forms. To do so, we bound the variance these forms. We first prove a lemma which helps decorrelate a fourth moment with pseudo-Lipschitz function. 
  \begin{lem}\label{lem:x4pl}
    Suppose $h_1$ and $h_2$ are pseudo-Lipschitz.  Define 
    \[
        g\left( \frac{\|\x\|^2}{d}\right)
        \coloneqq \E \left[     
          h_1\left(\frac{(W^\top \x)_b}{\sqrt{d}}\right)
          h_2\left(\frac{(W^\top \x)_j}{\sqrt{d}}\right) 
          \middle\vert  \x
          \right].
    \]
    Then
    \[
      \E \left[ x_ax_cx_ix_k 
      h_1\left(\frac{(W^\top \x)_b}{\sqrt{d}}\right)
      h_2\left(\frac{(W^\top \x)_j}{\sqrt{d}}\right)
      \right]
      = 
      \E \left[ x_ax_cx_ix_k 
      \right] 
      \E \left[
        g\left( \frac{\|\x\|^2}{d}\right)
      \right] 
      +\frac{\mathcal{T}_{acik}}{d}\E \left[g' \left( \frac{\|\x\|^2}{d}\right)\right] + \O{\frac{1}{d^2}}.
    \]
    The tensor $\mathcal{T}$ has bounded entries, it vanishes if the cardinality of the set $\{a,c,i,k\}$ is $3$ or $4$.
  \end{lem}
  \begin{proof}
    We start by noting that we can write, if $z\sim\mathcal{N}(0,1)$,  
    \[
    h(x)\coloneqq g(x^2)=\E\left[
      h_1(xz)h_2(xz)
    \right].  
    \]
    In particular, since $h_1$ and $h_2$ are pseudo-Lipschitz, they are locally Lipschitz and thus differentiable Lebesgue-almost everywhere and we have for $x\neq 0$, 
    \[
    h'(x) = 2xg'(x^2) = \E\left[
      z\left(
      h_1'(xz)h_2(xz)+h_1(xz)h_2'(xz)
      \right)
    \right]   
    =
    \E\left[
      \frac{z^2-1}{x}h_1(xz)h_2(xz)
    \right]
    \]
    where we used Stein's lemma and thus 
    \[
    g'(x^2)=\frac{1}{2x^2}\E\left[
      (z^2-1)h_1(xz)h_2(xz)
    \right].  
    \]
    We note that the product of two pseudo-Lipschitz functions is pseudo-Lipschitz and thus there exists $K$ and $\alpha$ such that for $x,y> 0$,
    \begin{align*}
      \left\vert g'(x^2)-g'(y^2)\right\vert 
      &\leqslant 
      \frac{1}{2x^2}\E\left[
        \vert z^2-1\vert\left\vert
          h_1(xz)h_2(xz)-h_1(yz)h_2(yz)
        \right\vert
      \right]
      +
      \frac{1}{2}\E\left[
        \vert (z^2-1)h_1(yz)h_2(yz)\vert\left\vert
          \frac{1}{x^2}-\frac{1}{y^2}
        \right\vert
      \right]\\
      &\leqslant 
      K\frac{\vert x-y\vert}{2x^2}\E\left[
        \vert z^3-z\vert (1+\vert zx\vert^\alpha+\vert zy\vert^\alpha)
      \right]
      +
      \frac{\vert x^2-y^2\vert}{2x^2y^2}\E\left[
        \vert (z^2-1)h_1(xz)h_2(yz)\vert
      \right]\\
      &\leqslant 
      K'\frac{\vert x^2-y^2\vert }{2\min(x^2\vert x+y\vert,x^2y^2)}\left(
        1+\vert x\vert^{\alpha'}+\vert y\vert^{\alpha'}
      \right).
    \end{align*}
    Thus, we see that as long as $x$ and $y$ are away from 0, $g'$ is itself a pseudo-Lipschitz function. Now, we first write
    \[
    \frac{\Vert \x^{(a,c,i,k)}\Vert^2}{d}
    \coloneqq 
    \frac{\Vert \x\Vert^2}{d}  
    -
    \frac{x_a^2+x_c^2+x_i^2+x_k^2}{d}    \] 
    and then by a Taylor expansion
    \begin{align*}
    \E\left[
      x_ax_cx_ix_kg\left(
        \frac{\Vert \x\Vert^2}{d}
      \right)
    \right]  
    =&
    \E\left[
      x_ax_cx_ix_k
    \right]
    \E\left[
      g\left(
        \frac{\Vert \x^{(a,c,i,k)}\Vert^2}{d}
      \right)
    \right]\\
    &
    +\frac{1}{d}
    \E\left[
      x_ax_cx_ix_k 
      \left(
        x_a^2+x_c^2+x_i^2+x_k^2
      \right)
    \right]
    \E\left[
      g'\left(
        \frac{\Vert \x^{(a,c,i,k)}\Vert^2}{d}
      \right)
    \right]\\
    &
    +
    \O{
      \E\left[
        \left\vert
          \frac{d^2x_ax_cx_ix_k}{\Vert \x^{(a,c,i,k)}\Vert^4}
          \left(
            1+\left\vert \frac{\Vert \x^{(a,c,i,k)}\Vert^2}{d}\right\vert^{\alpha'}
          \right)
        \right\vert
        \left(
          \frac{x_a^2+x_c^2+x_i^2+x_k^2}{d}
        \right)^2
      \right]
    }.
    \end{align*}
    We see by independence of the $x$ entries and the fact that they are centered that the second term is zero if the cardinality of $\{a,c,i,k\}$ is 3 or 4 and the last term is actually $\O{\frac{1}{d^2}}.$ Now, if we want to reintegrate the missing term in the norm to get the final result we do the same process and since the terms will be multiplied by a tensor with zero entries when $\mathrm{Card}\{a,c,i,k\}\geqslant 3$, we obtain the result.
  \end{proof}
  We are now ready to prove the concentration of these quadratic forms in $\omega$.
  \begin{prop}
    For $i\in\unn{1}{n}$ and $\mathcal{A}\in \Rbb^d\otimes \Rbb^p\otimes \Rbb^d\otimes \Rbb^p$ such that $\Vert \mathcal{A}\Vert_{\mathrm{op}}\lesssim 1$, there exists a $C>0$, such that we have 
    \[
      \mathrm{Var}\left(\frac{1}{n}\langle \omega^i\otimes \omega^i,\mathcal{A}\rangle\right) \leqslant \frac{C}{d},
      \quad 
      \mathrm{Var}\left(
        \frac{1}{n}\langle 
          \wt{\omega}^i\otimes \wt{\omega}^i,\mathcal{A}
        \rangle
      \right)
      \leqslant \frac{C}{d}.
    \]
  \end{prop}
  \begin{proof}
    The proof for the two quantities are similar but it is harder for $\omega$ instead of $\wt{\omega}$ since there is more dependence between the entries of the tensor, we thus develop the proof only for this term. Since the $\omega^i$ are i.i.d variables for $i\in\unn{1}{n}$, we omit the $i$-dependence in this proof for ease of notation. We first see that 
    \[
      \Var \left(
        \frac{1}{n}\langle \omega\otimes \omega,\mathcal{A}\rangle
      \right)
      =
      \frac{1}{n^2}
      \E\left[
        \langle \omega^{\otimes 4},\mathcal{A}\otimes \mathcal{A}\rangle
      \right]
      -
      \frac{1}{n^2}
      \E\left[
        \langle\omega\otimes \omega,\mathcal{A}\rangle
      \right]^2.
    \]
  If we unfold the first term we get something of the form, remembering that $\omega = \x\otimes D\y_{\phi}$ (omitting the $i$-dependence), 
  \[
    \frac{1}{n^2}
    \E\left[
        \langle \omega^{\otimes 4},\mathcal{A}\otimes \mathcal{A}\rangle
      \right]
    =
    \frac{1}{n^2} \sum_{a,c,i,k=1}^d \sum_{b,d,j,\ell=1}^p 
    \E\left[
      x_aD_{bb}y_{\phi,b}x_cD_{dd}y_{\phi,d}x_iD_{jj}y_{\phi,j}x_kD_{\ell\ell}y_{\phi,\ell}
    \right]
    \mathcal{A}_{ijk\ell}\mathcal{A}_{abcd}.
  \]
  Using the fact that $D$ is a diagonal matrix with $i.i.d.$ centered random variables, we must match the entries and we can write this expectation as the sum of four terms,
  \begin{multline}\label{eq:tensordevelop}
    \frac{1}{n^2}
    \E\left[
      \langle \omega^{\otimes 4},\mathcal{A}\otimes \mathcal{A}\rangle
    \right]
    =
    \frac{\sigma_D^4}{n^2}
    \sum_{a,c,i,k=1}^d
    \sum_{b\neq j}^p
    \E\left[
      x_ax_cx_{i}x_ky_{\phi,b}^2y_{\phi,j}^2
    \right]
    \left(
      \mathcal{A}_{ijkj}\mathcal{A}_{abcb}+\mathcal{A}_{ijkb}\mathcal{A}_{ajcb}+\mathcal{A}_{ijkb}\mathcal{A}_{ab cj}
    \right)\\
    +
    \frac{\E\left[D_{11}^4\right]}{n^2}
    \sum_{a,c,i,k=1}^d
    \sum_{b=1}^p
    \E\left[
      x_ax_cx_ix_k y_{\phi,b}^4
    \right]\mathcal{A}_{ibkb}\mathcal{A}_{abcb}.
  \end{multline}
  We claim,
  using Lemma \ref{lem:x4pl}, 
  that
  \begin{multline}\label{eq:tensordevelop2}
    \frac{1}{n^2}
    \E\left[
      \langle \omega^{\otimes 4},\mathcal{A}\otimes \mathcal{A}\rangle
    \right]
    \\=
    \frac{\sigma_D^4}{n^2}
    \sum_{a,c,i,k=1}^d
    \sum_{b\neq j}^p
    \E\left[
      x_ax_cx_{i}x_k
      \right]
    \E\left[
      y_{\phi,b}^2y_{\phi,j}^2
    \right]
    \left(
      \mathcal{A}_{ijkj}\mathcal{A}_{abcb}+\mathcal{A}_{ijkb}\mathcal{A}_{ajcb}+\mathcal{A}_{ijkb}\mathcal{A}_{ab cj}
    \right)
    +\O{\frac{1}{d}}.
  \end{multline}
  The errors are controlled in two different ways.  First, for the error terms that multiply the tensor $\mathcal{T}$, the sum over $6$ indices is reduced to a sum over $4$ indices; each term in this is uniformly bounded by 
  \[ 
    \frac{1}{dn^2} 
    \|\mathcal{A}\|_{\mathrm{op}}^2 
    \cdot
    \max_{a,c,i,k} |\mathcal{T}_{acik}|
    \cdot
    \left|\E \left[g' \left( \frac{\|\x\|^2}{d}\right)\right]\right|
  \]
  and thus the whole contribution is $\O{\frac{1}{d}}$.  The remaining error terms carry a $\O{\frac{1}{d^2}}$ but require additional care in estimation. Indeed, we get an error of order
  \[
  \frac{1}{d^2n^2}\sum_{a,c,i,k=1}^d \sum_{b,j=1}^p \left\vert \mathcal{A}_{ijkj}\mathcal{A}_{abcb}+\mathcal{A}_{ijkb}\mathcal{A}_{ajcb}+\mathcal{A}_{ijkb}\mathcal{A}_{abcj}\right\vert.  
  \]
  For the first term we can bound it in the following way,
  \[
  \frac{1}{d^2n^2}\sum_{a,c,i,k=1}^d \sum_{b,j=1}^p \vert \mathcal{A}_{ijkj}\vert \cdot\vert \mathcal{A}_{abcb}\vert 
  =
  \frac{1}{d^2n^2}\left(
    \sum_{a,c=1}^d \sum_{b=1}^p \vert \mathcal{A}_{abcb}\vert
  \right)^2
  \leqslant \frac{pd^2}{d^2n^2}\Vert \mathcal{A}\Vert_\mathrm{F}^2
  \leqslant \frac{p^2d^3}{d^2n^2}\Vert \mathcal{A}\Vert_{\mathrm{op}} = \O{\frac{1}{d}}.
  \]
  The other two terms are bounded in the same way and are actually of lower order, 
  \begin{multline*}
  \frac{1}{d^2n^2}\sum_{a,c,i,k=1}^d\sum_{b,j=1}^p \vert \mathcal{A}_{ijkb}\vert \cdot\vert \mathcal{A}_{ajcb}\vert
  \leqslant
  \frac{1}{d^2n^2}\sum_{a,c,i,k=1}^d \sqrt{\sum_{b,j=1}^p \vert \mathcal{A}_{ijkb}\vert^2}\sqrt{\sum_{b,j=1}^p\vert \mathcal{A}_{ajcb}\vert^2}  
  \\=
  \frac{1}{d^2n^2}\left(
    \sum_{a,c=1}^d \sqrt{\sum_{b,j=1}^p \vert \mathcal{A}_{ajcb}\vert^2}
  \right)^2
  \leqslant \frac{1}{n^2}\Vert \mathcal{A}\Vert_{\mathrm{F}}^2 \leqslant \frac{1}{pd}\Vert \mathcal{A}\Vert_{\mathrm{op}}^2=\O{\frac{1}{pd}}.
  \end{multline*}
  For the term involving $\E[D_{11}^4]$, we can still apply Lemma \ref{lem:x4pl} (with $h_1=h_2$) and the term of higher order is of order
  \begin{multline*}
  \frac{1}{n^2}\sum_{a,c,i,k=1}^d \sum_{b=1}^p
  \E\left[
    x_ax_cx_ix_k
  \right]
  \E\left[
    y_{\phi,b}^4
  \right]\mathcal{A}_{ibkb}\mathcal{A}_{abcb}
  =
  \frac{\E\left[y_{\phi,1}^4\right]\sigma_X^4}{n^2}
  \sum_{a,i=1}^d \sum_{b=1}^p \left(
  \mathcal{A}_{ibib}\mathcal{A}_{abab}+\mathcal{A}_{ibab}^2+\mathcal{A}_{ibab}\mathcal{A}_{abib}
  \right)
  \\
  +
  \frac{\E\left[y_{\phi,1}^4\right]\E\left[x_{11}^4\right]}
  {n^2}\sum_{a=1}^d\sum_{b=1}^p \mathcal{A}_{abab}^2.
  \end{multline*}
  The first term in the sum can be bounded in the following way 
  \[
  \frac{1}{n^2}\sum_{a,i=1}^d \sum_{b=1}^p \mathcal{A}_{ibib}\mathcal{A}_{abab}
  =
  \frac{1}{n^2}\sum_{b=1}^p \left(
    \sum_{i=1}^d \mathcal{A}_{ibib}
  \right)^2
  \leqslant \frac{d}{n^2}\sum_{b=1}^p 
  \sum_{i=1}^d \mathcal{A}_{ibib}^2
  \leqslant \frac{pd^2}{n^2}\Vert \mathcal{A}\Vert_{\mathrm{op}}=\O{\frac{1}{d}}.
  \]
  The two other next terms are bounded in the same way by 
  \[
  \frac{1}{n^2}\sum_{a,i=1}^d \sum_{b=1}^p \mathcal{A}_{ibab}^2 \leqslant \frac{1}{n^2}\Vert \mathcal{A}\Vert_{\mathrm{F}}^2 \leqslant\frac{pd}{n^2}\Vert \mathcal{A}\Vert_{\mathrm{op}}=\O{\frac{1}{pd}}.   
  \]
  Finally, the very last term can also be crudely bounded by, 
  \[
  \frac{C}{n^2} \Vert \mathcal{A}\Vert_{\mathrm{F}}^2 \leqslant\frac{C}{pd}\Vert \mathcal{A}\Vert_{\mathrm{op}}=\O{\frac{1}{pd}}. 
  \]

  We now come back to \eqref{eq:tensordevelop2} and see that the first term from the sum will partially be cancelled by the centering the variance, indeed we have 
  \begin{multline}\label{eq:tensorexpect}
  \frac{1}{n^2}
  \left(
    \sum_{i,k=1}^d \sum_{j,\ell=1}^p 
    \E\left[
      x_{i}x_kD_{jj}y_{\phi,j}D_{\ell\ell}y_{\phi,\ell}
    \right]
    \mathcal{A}_{ijk\ell}
  \right)^2
  =
  \frac{\sigma_D^4}{n^2}
  \left(
    \sum_{i,k=1}^d 
    \sum_{j=1}^p
    \E\left[
      x_ix_ky_{\phi,j}^2
    \right]
    \mathcal{A}_{ijkj}
  \right)^2 
  \\
  =
  \frac{\sigma_D^4}{n^2}
  \sum_{a,c,i,k=1}^d
  \sum_{b,j,=1}^p
  \E\left[
    x_ix_ky_{\phi,j}^2
  \right]
  \E\left[
    x_ax_cy_{\phi,b}^2
  \right]
  \mathcal{A}_{ijkj}\mathcal{A}_{abcb}.
\end{multline}

We see that in \eqref{eq:tensordevelop2}, we have the four moment $\E[x_ax_ix_cx_k]$ and since entries of $x$ are independent and centered we obtain three terms from matching pairs of indices and one term from matching them all. Firstly, if we match $a\leftrightarrow c$ and $i\leftrightarrow k$ we get a contribution of order 
\[
  \frac{1}{n^2}\sum_{a,i=1}^d\sum_{b,j=1}^p\left(
    \mathcal{A}_{ijij}\mathcal{A}_{abab}+\mathcal{A}_{ijib}\mathcal{A}_{ajab}+\mathcal{A}_{ijib}\mathcal{A}_{abaj}
  \right).
\]
 The first term $\sum_{a,i}\sum_{b,j}\mathcal{A}_{ijij}\mathcal{A}_{abab}$ is exactly cancelled by the expectation from \eqref{eq:tensorexpect}. 
The two other terms are bounded similarly since we can write 
\[
\sum_{a,i=1}^d\sum_{b,j=1}^p \mathcal{A}_{ijib}\mathcal{A}_{ajab}
=
\sum_{b,j=1}^p \left(\sum_{i=1}^d \mathcal{A}_{ijib}\right)^2
=
\sum_{b,j=1}^p B_{jb}^2
=
\Vert B\Vert_\mathrm{F}^2
\leqslant pd^2\Vert \mathcal{A}\Vert_{\mathrm{op}}^2 
\]
and 
\[
\sum_{a,i=1}^d\sum_{b,j=1}^p \mathcal{A}_{ijib}\mathcal{A}_{abaj}
=
\sum_{b,j=1}^p B_{jb}B_{bj}
\leqslant \Vert B\Vert_\mathrm{F}^2
\leqslant pd^2\Vert \mathcal{A}\Vert_{\mathrm{op}}.
\]
Now, if we match $a\leftrightarrow i$ and $c\leftrightarrow k$, we obtain 
\[
  \sum_{i,k=1}^d\sum_{b,j=1}^p\left(
    \mathcal{A}_{ijkj}\mathcal{A}_{ibkb}+\mathcal{A}_{ijkb}\mathcal{A}_{ijkb}+\mathcal{A}_{ijkb}\mathcal{A}_{ibkj}
  \right).
\]
The first term can be bounded in exactly the same way as above and for the other two terms 
\[
\sum_{i,k=1}^d \sum_{b,j=1}^p \mathcal{A}_{ijkb}^2 + \mathcal{A}_{ijkb}\mathcal{A}_{ibkj} \leqslant 2\Vert \mathcal{A}\Vert_{\mathrm{F}}^2 \leqslant 2pd\Vert \mathcal{A}\Vert_{\mathrm{op}}.  
\]
Finally, for the matching $a\leftrightarrow k$ and $c\leftrightarrow i$, we have 
\[
  \sum_{i,k=1}^d\sum_{b,j=1}^p
  \left(
    \mathcal{A}_{ijkj}\mathcal{A}_{kbib}+\mathcal{A}_{ijkb}\mathcal{A}_{kjib}+\mathcal{A}_{ijkb}\mathcal{A}_{kbij}
  \right)
  \leqslant pd^2\Vert \mathcal{A}\Vert_{\mathrm{op}}
\]
since each term can be bounded similarly as above. Finally, we have the term when all indices are matched and the 4th moment of $x$ appear and similarly, 
\[
\frac{1}{n^2}\sum_{i=1}^d \sum_{b\neq j}^p\left(
  \mathcal{A}_{ijij}\mathcal{A}_{ibib}+\mathcal{A}_{ijib}^2+\mathcal{A}_{ijib}\mathcal{A}_{ibij}
\right)  
=
\O{\frac{1}{d}}.
\]
\end{proof}
We are now ready to prove Lemma \ref{lem:comparewwtilde}.
\begin{proof}[Proof of Lemma \ref{lem:comparewwtilde}]
  Denote for simplicity 
  \[
  \mathcal{K}=\E\left[
    \omega^i\otimes \omega^i
  \right],\quad \wt{\mathcal{K}}=\E\left[
    \wt{\omega}^i\otimes \wt{\omega}^i
  \right]
  \]
  Then we can write 
  \[
  \frac{1}{n}\left\vert 
          \langle \omega^i\otimes \omega^i - \wt{\omega}^i\otimes \wt{\omega}^i,\A\rangle  
  \right\vert
  \leqslant  
  \frac{1}{n}\left\vert
    \langle \omega^i\otimes \omega^i -\mathcal{K},\A\rangle
  \right\vert
  +
  \frac{1}{n}\left\vert
    \langle \wt{\omega}^i\otimes \wt{\omega}^i -\wt{\mathcal{K}},\A\rangle
  \right\vert 
  +
  \frac{1}{n}\left\vert
    \langle \wt{\mathcal{K}}-\mathcal{K},\A\rangle
  \right\vert.  
  \]
  For the third term, by Lemma \ref{lem:compareexpect}, we have that 
  \[
  \frac{1}{n}\langle \wt{\mathcal{K}}-\mathcal{K},\A\rangle 
  \leqslant 
  \frac{1}{n}\Vert \wt{\mathcal{K}}-\mathcal{K}\Vert_{\mathrm{F}}\Vert \A\Vert_{\mathrm{F}}
  \leqslant Cd^{\frac{\varepsilon}{2}}\frac{p\sqrt{d}}{n}\Vert \A\Vert_{\mathrm{op}}
  =
  \O{\frac{d^{\frac{\varepsilon}{2}}}{\sqrt{d}}}.  
  \]
  Using the Bienaym\'e--Chebyshev inequality, we thus have that 
    \[
    \P\left(
      \frac{1}{n}\left\vert 
        \langle
          \omega^i\otimes \omega^i-\mathcal{K},\A 
        \rangle
      \right\vert
      \geqslant \frac{d^\varepsilon}{\sqrt{d}}
    \right)  
    \leqslant d^{1-2\varepsilon}
    \mathrm{Var}\left(
      \frac{1}{n}\langle 
        \omega^i\otimes \omega^i,\A
      \rangle
    \right)
    \leqslant \frac{C}{d^{2\varepsilon}}
    \]
    and the same result holds for $\wt{\omega}^i$ and the result is proved.
\end{proof}
Finally we are ready to prove Proposition \ref{prop:replace}.
\begin{proof}[Proof of Proposition \ref{prop:replace}]
We start by using Lemma \ref{lem:sherman} to write 
\[
\frac{1}{n}\Tr(\mathcal{L}-z)^{-1}-\frac{1}{n}\Tr(\wt{\mathcal{L}}-z)^{-1}
=
\sum_{i=1}^n \frac{1}{n}\Tr (\mathcal{L}^{[i-1],\wt{\omega}}-z)^{-1}-\frac{1}{n}\Tr (\mathcal{L}^{[i],\wt{\omega}}-z)^{-1}  
\]
where we denoted 
\[
\mathcal{L}^{[i],\wt{\omega}}
=
\mathcal{L} - \frac{1}{n}\sum_{j=1}^i \omega^i\otimes \omega^i +\frac{1}{n}\sum_{j=1}^i \wt{\omega}^i\otimes \wt{\omega}^i
\quad\text{and}\quad 
\mathcal{L}^{[0],\wt{\omega}}
=
\mathcal{L}.  
\]
as
\begin{multline*}
  \frac{1}{n}\Tr(\mathcal{L}-z)^{-1}-\frac{1}{n}\Tr(\wt{\mathcal{L}}-z)^{-1}
  \\=
  \frac{1}{n}\sum_{i=1}^n \frac{1}{n}\Tr\left( 
  (\hat{\mathcal{L}}^{i,0}-z)^{-1}
    \left( 
      \frac{\omega^i\otimes\omega^i}{1+\frac{1}{n}\langle \omega^i\otimes \omega^i,(\hat{\mathcal{L}}^{i,0}-z)^{-1}\rangle}- \frac{\wt{\omega}^i\otimes \wt{\omega}^i}{1+\frac{1}{n}\langle \wt{\omega}^i\otimes \wt{\omega}^i,(\hat{\mathcal{L}}^{i,0}-z)^{-1}\rangle} 
    \right)
    (\hat{\mathcal{L}}^{i,0}-z)^{-1} 
    \right)
\end{multline*}
where we set 
\[
\hat{\mathcal{L}}^{i,0} = \mathcal{L}^{[i-1],\wt{\omega}}-\frac{1}{n}\omega^i\otimes \omega^i.  
\]
We can now use Lemma \ref{lem:comparewwtilde} to see that this quantity is of order, with high probability,  
\[
  \frac{1}{n}\left(
    1+\frac{d^\varepsilon}{\sqrt{d}}
  \right)
  \sum_{i=1}^n 
  \frac{1}{1+\frac{1}{n}\langle 
    \omega^i\otimes\omega^i,(\hat{\mathcal{L}}^{i,0}-z)^{-1}
  \rangle}
  \frac{1}{n}\Tr \left(
    (\hat{\mathcal{L}}^{i,0}-z)^{-1}
    \left(
      \omega^i\otimes \omega^i -\wt{\omega}^i\otimes \wt{\omega}^i
    \right)
    (\hat{\mathcal{L}}^{i,0}-z)^{-1}
  \right).
    \]
    We also note that we have 
    \begin{multline*}
      \frac{1}{n}\Tr \left(
    (\hat{\mathcal{L}}^{i,0}-z)^{-1}
    \left(
      \omega^i\otimes \omega^i
    \right)
    (\hat{\mathcal{L}}^{i,0}-z)^{-1}
    \right)
    =
    \frac{1}{n}\sum_{q,s=1}^d\sum_{r,t=1}^p 
    \omega_{qr}^i\omega_{st}^i 
    \sum_{a=1}^d\sum_{b=1}^p 
    (\hat{\mathcal{L}}^{i,0}-z)^{-1}_{abqr}
    (\hat{\mathcal{L}}^{i,0}-z)^{-1}_{stab}
    \\=
    \frac{1}{n}\langle 
      \omega^i\otimes \omega^i , \mathcal{B}
    \rangle
    \end{multline*}
    where we defined the tensor
     \[\mathcal{B}_{qrst} 
     =
      \sum_{a=1}^d\sum_{b=1}^p (\hat{\mathcal{L}}^{i,0}-z)^{-1}_{abqr}
    (\hat{\mathcal{L}}^{i,0}-z)^{-1}_{stab}.
    \]
    We note that since $\Vert (\hat{\mathcal{L}}^{i,0}-z)^{-1}\Vert_{\mathrm{op}}\lesssim 1$, we have that $\Vert\mathcal{B}\Vert_{\mathrm{op}}\lesssim 1$ and we can use Lemma \ref{lem:comparewwtilde} to see that overall, with high probability  
    \[
      \left\vert
        \frac{1}{n}\Tr(\mathcal{L}-z)^{-1}-\frac{1}{n}\Tr(\wt{\mathcal{L}}-z)^{-1}
      \right\vert
      \leqslant \frac{d^\varepsilon}{\sqrt{d}}
    \] 
\end{proof}
\section{Convergence to the Marchenko--Pastur map}\label{sec:baizhou}
In this section, we compute the equation followed by the expected resolvent of our new model $\wt{K}$ which lost its dependence structure within the nonlinearity $\phi$. We recall that we consider the model, with $\wt{X}$ an $i.i.d$ matrix of $\mathcal{N}(0,1)$, 
\[
  \wt{K}=\frac{1}{d}XX^\top \odot \frac{1}{p}\left(
    \frac{\alpha_\phi}{\sqrt{d}}XW + \psi(\wt{X})
  \right)D^2
  \left(
    \frac{\alpha_\phi}{\sqrt{d}}XW + \psi(\wt{X})
  \right)^\top 
\]
and we saw that we can see this matrix as a Gram matrix of elements in $\R^d\otimes \R^p$,
\[
\wt{k}_{ij} = \frac{1}{pd}\langle \wt{\omega}^i,\wt{\omega}^j\rangle
\quad\text{with}\quad 
\wt{\omega}^i=\x_i\otimes D\left(
  \frac{\alpha_\phi}{\sqrt{d}}W^\top \x_i + \psi(\wt{\x}_i)
\right).  
\]
We give in this section a self-consistent equation based upon the work \cite{baizhou}.
We start by computing the covariance structure of $\wt{\omega}$ conditionally on $W$ and $D$. Since in this section, we always work conditionally on these two matrices, to ease the notation we use the notation 
\[
\mathbf{E}[A] = \E\left[
  A\middle\vert W,D 
\right].  
\]
\begin{lem}\label{lem:covcomp}
Conditionally on $W$ and $D$ we have that for $i,j\in\unn{1}{n}$,
\begin{equation*}
\mathbf{E}\left[
  \left(
    \wt{\omega}^i-\mathbf{E}[\wt{\omega}^i]
  \right)\otimes 
  \left(
    \wt{\omega}^j-\mathbf{E}[\wt{\omega}^j]
  \right)
\right] \\
=
\delta_{ij}\mathcal{Q}
+
\delta_{ij}
\frac{\alpha_\phi^2}{d}\mu_{4,x}\mathrm{diag}_1\left(
  WD\otimes WD
  \right)
\end{equation*}
where we defined for $\B\in\R^{d,p}\otimes \R^{d,p}$,
\[
  \begin{gathered}
\mathrm{diag}_1(\B)_{qrst} = \delta_{qs}\B_{qrqt},\\
\Q = \frac{\alpha_\phi^2}{d}
\left(
  \tau_{23}\left(
    \Id_d\otimes (DW^\top WD)
  \right)
  +
  \tau_{24}\left(
    WD\otimes WD
  \right)
\right)
+
\beta_\psi^2\tau_{23}\left(
  \Id_d\otimes D^2
\right),
  \end{gathered}  
\]
\end{lem}
\begin{proof}
  We can write for $i,j\in\unn{1}{n},$ $q,s\in\unn{1}{d},$ and $r,t\in\unn{1}{p}$, 
  \[
  \mathbf{E}\left[
    \wt{\omega}^i_{qr}\otimes \wt{\omega}^j_{st}
  \right]  
  =
  \mathbf{E}\left[
    x_{iq}\left(
      \frac{\alpha_\phi}{\sqrt{d}}D_{rr}\sum_{k=1}^d w_{kr}x_{ik}+D_{rr}\psi(\wt{x}_{ir})
    \right)
    x_{js}
    \left(
      \frac{\alpha_\phi}{\sqrt{d}}D_{tt}\sum_{\ell=1}^d w_{\ell t}x_{j\ell}+D_{tt}\psi(\wt{x}_{jt})
    \right)
  \right].
  \]
  If we unfold this product, we obtain 4 terms. 
  However, since $\wt{X}$ is independent of $X$, two terms vanish and thus we can write 
  \[
    \begin{aligned}
  \mathbf{E}\left[
    \wt{\omega}^i_{qr}\otimes \wt{\omega}^j_{st}
  \right]
  &=
  \frac{\alpha_\phi^2}{d}D_{rr}D_{tt}
  \sum_{k,\ell=1}^dw_{kr}w_{\ell t}
  \E\left[
    x_{iq}x_{ik}x_{js}x_{j\ell}
  \right]
  +
  D_{rr}D_{tt}\E\left[
    x_{iq}x_{js}
  \right]
  \E\left[
    \psi(\wt{x}_{ir})\psi(\wt{x}_{jt})
  \right] \\
  &+
  \frac{\alpha_\phi}{\sqrt{d}}D_{rr}D_{tt}
  \sum_{k=1}^d
  \left(
  w_{kr} \E[x_{iq}x_{ik}x_{js}]\E[\psi(\wt{x}_{jt})]
  +w_{kt} \E[x_{iq}x_{js}x_{jk}]\E[\psi(\wt{x}_{ir})]
  \right).
    \end{aligned}
  \]
  Since $\E[\psi(\mathcal{N}(0,1))]=0$ by construction, we see that the second line vanishes and for $i\neq j$, the second term of the first line too. Also, the first term simply becomes 
  \[
    \mathbf{E}\left[
      \wt{\omega}^i_{qr}
      \otimes 
      \wt{\omega}^j_{st}
    \right]
    =
    \frac{\alpha_\phi^2}{d}D_{rr}D_{tt}w_{qr}w_{st}
    =
    \mathbf{E}\left[
      \wt{\omega}^i_{qr}
    \right]
    \mathbf{E}\left[
      \wt{\omega}^j_{st}
    \right]
  \]
  and thus this term is cancelled by the centering in the covariance computation. For $i= j$, due to the structure of $X$, we have three terms when matching in pairs the $x$'s and a term when matching all $4$. The term matching $q\leftrightarrow k$ and $s\leftrightarrow \ell$ is also cancelled by the centering and thus we obtain the terms 
  \begin{multline*}
    \mathbf{E}\left[
      \wt{\omega}^i_{qr}
      \otimes 
      \wt{\omega}^j_{st}
    \right]
    =
    \frac{\alpha_\phi^2}{d}\left(
      \delta_{qs}\left(
        DW^\top W D
      \right)_{rt}
      +
      (WD)_{qt}(WD)_{sr}
      +
      \delta_{qs}\E[x_{11}^4](WD)_{qr}(WD)_{qt}
    \right)
    \\+
    \delta_{qs}\delta_{rt}D_{rr}^2
    \E\left[
      \psi(\wt{x}_{ir})^2
    \right].
  \end{multline*}
\end{proof}

We now construct an operator whose entries are also given by $\wt{\omega}^i\otimes \wt{\omega}^j$. If we set the rank 3 tensor, which we see as an operator from $\Rbb^n$ to $\Rbb^d\otimes \Rbb^p\eqqcolon \R^{d,p}$,
\[
\mathcal{A} = \frac{1}{\sqrt{pd}}\sum_{i=1}^n \wt{\omega}^i \otimes \e_i \in \R^{d,p}\otimes \R^n  
\]
then its adjoint $\mathcal{A}^\top$ is an operator from $\R^{d,p}$ to $\R^n$ and we can write
\[
\mathcal{A}^\top \mathcal{A} = 
\frac{1}{pd}\sum_{i,j=1}^n \langle \wt{\omega}^i,\wt{\omega}^j\rangle \e_i\otimes \e_j \in \R^n\otimes \R^n
\]
or, in other words, $\mathcal{A}^\top\mathcal{A}$ is a $n\times n$ matrix whose entries are given by $\frac{1}{pd}\langle \wt{\omega}^i, \wt{\omega}^j\rangle$. We recall the fundamental result of \cite{baizhou} for matrices with independent columns which will be used to understand the eigenvalues of $\A^\top \A$. We denote 
\[
\underline{\A}_{qri} =  \wt{\omega}^i_{qr}-\mathbf{E}[\wt{\omega}^i_{qr}]
\]

\begin{thm}\label{thm:baizhou}
  For all $k\in\unn{1}{n}$, if we have that $\Eb\left[\underline{\A}_{qrk}\underline{\A}_{stk} \right]=\T_{qrst}$ and for any non-random tensor $\B\in \R^{d,p}\otimes \R^{d,p}$ with bounded norm,
  \[
    \frac{1}{n^2}\Eb\left[
      \left(
        \Au_{\cdot k}^\top \B\Au_{\cdot k} - \mathrm{Tr}(\B\T)
      \right)^2
    \right]  
    \xrightarrow[n\to\infty]{}0.
  \] If additionally the norm of the tensor $\T$ is uniformly bounded and the empirical eigenvalue distribution of $\T$ tends to a non-random probability distribution $\pi$ then the eigenvalue distribution of $\A^\top \A$ converges weakly almost surely to the probability distribution $\mu_{\mathrm{MP}}^{\gamma_1}\boxtimes \pi$ defined by Theorem \ref{theo:baisilverstein}.
\end{thm}
We have defined the tensor $\T$ in Lemma \ref{lem:covcomp} and we can write it as 
\[
\begin{gathered}
\T = \Q + \R^{(1)} 
\quad\text{ with }\quad
\R^{(1)} = \frac{\alpha_\phi^2}{d}\mu_{4,x}\mathrm{diag}_1\left(
  WD\otimes WD
  \right)
\end{gathered}
\]
\begin{lem}
  With very high probability, the tensor $\T$ is uniformly bounded.
\end{lem}
\begin{proof}
  We can bound each part of $\T$ separately. Firstly, we see that we have the first term in the definition of $\Q$, 
  \(
  \frac{\alpha_\phi^2}{d}\tau_{23}\left(
    \mathrm{Id}_d \otimes \left(
      DW^\top WD
    \right)
  \right).
  \)
  We can bound this part using the fact that $\mathrm{Id}_d$ has norm one and $\frac{1}{d}DW^\top WD$ has almost surely bounded spectral norm \cites{silverstein1995analysis, bai1998no}. The second part is slightly more complicated but note that we have for $\u\in\mathbb{R}^d$ and $\v\in\mathbb{R}^p$, 
  \[
  \tau_{24}\left(
    WD\otimes WD
  \right)(\u\otimes \v)
  =
  \left(
    WD\v
  \right)
  \otimes 
  \left(
    DW^\top \u
  \right).
  \]
  Using the fact that the map $\u\otimes \v \mapsto \v\otimes \u$ is an isometry from $\R^{d,p}$ to $\R^{p,d}$ and that $DW^\top WD$ has bounded spectral norm almost surely as before, we can bound independently of the dimension the norm of this part. For the part $\tau_{23}\left(\mathrm{Id}\otimes D^2\right)$ we can use the fact that the entries $D$ are compactly supported. Finally, we have the remainder term, if we consider a test matrix $A\in\mathbb{R}^{d\times p}$ of bounded norm then 
  \[
  (\R^{(1)}A)_{ab} = 
  \mu_{4,x}\frac{\alpha_\phi^2}{d}
  \sum_{r=1}^p
  (WD)_{ar}(WD)_{ab}A_{ar}
  =
  \mu_{4,x}\frac{\alpha_\phi^2}{d}
  (WD)_{ab}\left(ADW^\top\right)_{aa}
  \]
  which is uniformly bounded since $\frac{1}{\sqrt{d}}ADW^\top$ has bounded norm and $(WD)_{ab}$ is a single entry. 
\end{proof}
\begin{prop}\label{prop:baizhouquadratic}
  We have that for any $\varepsilon>0$, with very high probability and for all $i\in \{1,\dots,n\}$ and $\B\in \R^{d,p}\otimes \R^{d,p}$ of bounded norm, 
  \[
    \frac{1}{n^2}\Eb\left[
      \left(
        \Au_{\cdot i}^\top \B\Au_{\cdot i} - \mathrm{Tr}(\B\T)
      \right)^2
    \right]  
    =
    \O{\frac{d^{4\varepsilon}}{\sqrt{d}}}
  \] 
  where $(\Au_{\cdot i})_{qr} = \Au_{qri}$.
\end{prop}
Recall that we can write 
\[
\Au_{qri} = x_{iq}\left(
  \frac{\alpha_\phi}{\sqrt{d}}(DW^\top \x_i)_r + D_{rr}\psi(\wt{x}_{ir})
\right)
-
\frac{\alpha_\phi}{\sqrt{d}}(WD)_{qr}.
\]
For ease of notation, we introduce the following matrices 
\[
Z = \frac{\alpha_\phi}{\sqrt{d}}WD\quad\text{and}\quad M = Z^\top Z.
\]
In particular, we have that with very high probability and for any $\varepsilon>0$, $\Vert Z\Vert,\,\Vert M\Vert \leqslant d^\varepsilon$. We can also introduce the notation 
\[
\Au_{qri} = \Au^1_{qri}+\A^2_{qri}
\quad\text{with}\quad 
\Au^1_{qri}=x_{iq}(Z^\top\x_i)_r-Z_{qr} \quad\text{and}\quad 
\A^2_{qri} = x_{iq}D_{rr}\psi(\wt{x}_{ir}).
\]
If we compute the quadratic form we get that 
\[
\frac{1}{n^2}\sum_{q,s,q',s'=1}^d\sum_{r,t,r',t'=1}^p 
\B_{qrst}\B_{q'r's't'}
\left(
  \Eb\left[
    \Au_{qri}\Au_{sti}\Au_{q'r'i}\Au_{s't'i}-\T_{qrst}\T_{q'r's't'}
  \right]
\right)
\]

To evaluate this fourth moment, we consider conditioning on the random variables $\mathfrak{X} = \{ x_{i\ell}, \ell\in L\}$, where $L = \{q,s,q',s'\}$.  
This allows us to represent  
\begin{equation}\label{eq:Au-cond}
\Au_{qri} = (x_{iq}^2-1)Z_{qr} 
+ \sum_{\substack{\ell \in L\\ \ell\neq q}} x_{iq}x_{i\ell}Z_{\ell r} 
+ \sum_{\substack{\ell \notin L}} x_{iq}x_{i\ell}Z_{\ell r} 
+ x_{iq}D_{rr}\psi(\wt{x}_{ir}).
\end{equation}
We now record some conditional cumulants of $\Au_{qri}$ given $\mathfrak{X}$ (and still all $\{Z_{qr}\}$).
\begin{lem}
  The conditional cumulants satisfy, on an event in $\sigma( \{Z_{qr}\})$ of very high probability, the following estimates.
  \begin{enumerate}
    \item The first conditional cumulant:
    \begin{equation}\label{eq:first-cumulant}
    \kappa( \Au_{qri} \mid \mathfrak{X}) = (x_{iq}^2-1)Z_{qr} 
    + \sum_{\substack{\ell \in L\\ \ell\neq q}} x_{iq}x_{i\ell}Z_{\ell r}.
    \end{equation}
    \item The second conditional cumulant:
    \[
    \kappa( \Au_{qri},\Au_{q'r'i} \mid \mathfrak{X}) =
    x_{iq}x_{iq'}
    \left((
      Z^{\top} Z)_{rr'}
      + \omega^{(L)}_{rr'}
      + \delta_{r r'} D_{rr}^2\Eb[ \psi(\wt{x}_{ir})^2]
    \right),
    \]
    where $|\omega^{(L)}_{rr'}| \leq d^{\varepsilon-1}$ and $\|\omega^{(L)}\|^2 \leq d^{\varepsilon}$.
    \item The third conditional cumulant:
    \[
    \kappa( \Au_{qri},\Au_{q'r'i},\Au_{s't'i} \mid \mathfrak{X})
    = x_{iq}x_{iq'}x_{is'} \left( 
      \zeta_{rr't'} 
      + \omega^{(L)}_{rr't'}
      + \delta_{r r'}\delta_{r t'} D_{rr}^3 \kappa_3(\psi(\wt{x}_{ir}))
      \right),
    \]
    where $|\zeta_{rr't'}| \leq d^{\varepsilon-1}$ and $\|\omega^{(L)}\|^2 \leq d^{\varepsilon}$.
    \item The fourth conditional cumulant:
    \[
    \kappa( \Au_{qri},\Au_{q'r'i},\Au_{sti},\Au_{s't'i} \mid \mathfrak{X})
    = x_{iq}x_{iq'}x_{is}x_{is'} \left( \zeta_{r,r',t,t'} + \delta_{r r'}\delta_{r t} \delta_{r t'} D_{rr}^4 \kappa_4(\psi(\wt{x}_{ir}))\right),
    \]
    where $|\zeta_{rr'tt'}| \leq d^{\varepsilon-2} \times \left(
      p^{1/2} 
      + p(\delta_{r r'}\delta_{tt'} 
      + \delta_{r t}\delta_{r' t'} 
      + \delta_{r t'}\delta_{t' t})
       \right)$
    and $\|\omega^{(L)}\|^2 \leq d^{\varepsilon}$.
  \end{enumerate}
  \label{lem:conditional-cumulants}
\end{lem}
\begin{proof}
  These cumulant formulas follow directly from the multilinearity of the cumulant, \eqref{eq:Au-cond} and the fact that the entries of $x$ are independent and centered.
  The first cumulant is the conditional mean and is straightforward.
  
  \paragraph{Second cumulant:} For the second cumulant, we have that 
  \[
  \kappa( \Au_{qri},\Au_{q'r'i} \mid \mathfrak{X})
  =
  \kappa( \Au_{qri} - \Eb[\Au_{qri} \mid \mathfrak{X}],\Au_{q'r'i} - \Eb[\Au_{q'r'i} \mid \mathfrak{X}] \mid \mathfrak{X}),
  \]
  using that all higher conditional cumulants of $\mathfrak{X}$ are zero.  We can then expand this using independence of the entries of $x$ and the fact that $\Eb[x_{iq}^2] = 1$ to get 
  \[
  \begin{aligned}
  \kappa( \Au_{qri},\Au_{q'r'i} \mid \mathfrak{X})
  &=
  \sum_{\ell \notin L} \kappa( x_{iq}x_{i\ell}Z_{\ell r}, x_{iq'}x_{i\ell}Z_{\ell r'} \mid \mathfrak{X}) + \kappa(x_{iq}D_{rr}\psi(\wt{x}_{ir}), x_{iq'}D_{rr'}\psi(\wt{x}_{i'r'}) \mid \mathfrak{X}) \\
  &= 
  \sum_{\ell \notin L} \kappa( x_{iq}x_{i\ell}Z_{\ell r}, x_{iq'}x_{i\ell}Z_{\ell r'} \mid \mathfrak{X}) + \kappa(x_{iq}D_{rr}\psi(\wt{x}_{ir}), x_{iq'}D_{rr'}\psi(\wt{x}_{i'r'}) \mid \mathfrak{X}) \\
  &= 
  \sum_{\ell \notin L} x_{iq}x_{iq'} Z_{\ell r}Z_{\ell r'} + x_{iq}x_{iq'} \delta_{r r'} D_{rr}^2\Eb[ \psi(\wt{x}_{ir})^2].
  \end{aligned}
  \]
  Hence setting $\omega^{(L)}_{rr'}= -\sum_{\ell \in L} Z_{\ell r}Z_{\ell r'}$, we get the desired result.
  
  \paragraph{Third cumulant:} For the third cumulant, subtracting the mean, we have 
  \[
  \begin{aligned}
  &\kappa( \Au_{qri},\Au_{q'r'i},\Au_{s't'i} \mid \mathfrak{X}) \\
  &=
  \sum_{\ell \notin L} \kappa( x_{iq}x_{i\ell}Z_{\ell r}, x_{iq'}x_{i\ell}Z_{\ell r'}, x_{is'}Z_{\ell t'} \mid \mathfrak{X}) + \kappa(x_{iq}D_{rr}\psi(\wt{x}_{ir}), x_{iq'}D_{rr'}\psi(\wt{x}_{i'r'}), x_{is'}D_{tt'}\psi(\wt{x}_{is'}) \mid \mathfrak{X}) \\ 
  &= x_{iq}x_{iq'}x_{is'} \left( 
    \sum_{\ell \notin L} Z_{\ell r}Z_{\ell r'}  Z_{\ell t'} 
    +
  \delta_{r r'}\delta_{r t'} D_{rr}^3 \kappa_3(\psi(\wt{x}_{ir}))\right) \\
  &\eqqcolon x_{iq}x_{iq'}x_{is'} 
  \left( 
    \sum_{\ell} Z_{\ell r}Z_{\ell r'}  Z_{\ell t'} 
    + \omega^{(L)}_{rr't'}
    + \delta_{r r'}\delta_{r t'} D_{rr}^3 \kappa_3(\psi(\wt{x}_{ir}))
  \right).
  \end{aligned}
  \]
  This third order sum over $Z$ terms is always centered (taking expectation) and essentially uncorrelated.  As each term is $\O{d^{-3/2}}$ in standard deviation, and there are $p$ terms, we get that the sum is $\O{d^{\varepsilon-1}}$ with very high probability, uniformly in $L$.

  \paragraph{Fourth cumulant:} For the fourth cumulant, the procedure is similar. 
  \[
    \begin{aligned}
    &\kappa( \Au_{qri},\Au_{q'r'i},\Au_{sti},\Au_{s't'i} \mid \mathfrak{X}) \\
    &= x_{iq}x_{iq'}x_{is}x_{is'} \left( 
      \sum_{\ell \notin L} Z_{\ell r}Z_{\ell r'}   Z_{\ell t}Z_{\ell t'} 
      +\delta_{r r'}\delta_{r t} \delta_{r t'} D_{rr}^4 \kappa_4(\psi(\wt{x}_{ir}))
    \right) \\
    &\eqqcolon x_{iq}x_{iq'}x_{is}x_{is'} \left( 
      \sum_{\ell} Z_{\ell r}Z_{\ell r'}   Z_{\ell t} Z_{\ell t'}
      +\omega^{(L)}_{rr'tt'}
      +\delta_{r r'}\delta_{r t} \delta_{r t'} D_{rr}^4 \kappa_4(\psi(\wt{x}_{ir}))
    \right).
    \end{aligned}
  \]
  When all of $\{r,r',t,t'\}$ are paired (so the set is at most $2$ elements), we simply bound the entries of the sum.  Otherwise, we can bound the sum by $p^{1/2} \times d^{\varepsilon-2}$ with very high probability, which gives the desired result.

\end{proof}

From the moment-cumulant formulas, we have the representation 
\begin{equation}\label{eq:Au-moment-cumulant}
\Eb\left[
  \Au_{qri}\Au_{sti}\Au_{q'r'i}\Au_{s't'i} \mid \mathfrak{X}
\right]
=\sum_{\pi} \prod_{B \in \pi} \kappa( \{ \Au_{abi} : (a,b) \in B \} \mid \mathfrak{X}),
\end{equation}
with the sum over all set partitions of the multiset $\{ (q,r), (s,t), (q',r'), (s',t') \}$.
Taking expectation, we will have an expression that is dominated by the contributions of ``Wick terms''.  To describe these terms, we introduce a diagrammatic representation.  We arrange the $8$ vertices $\{q,s,q',s',r,t,r',t'\}$ into two columns, with the first column containing $\{q,s,q',s'\}$ and the second column containing $\{r,t,r',t'\}$.  We then draw the graph complement of the matching $\{ (q,r), (s,t), (q',r'), (s',t') \}$ (see Figure \ref{fig:matching}).
\begin{figure}[ht]
  \centering
  \begin{tikzpicture}[every node/.style={circle, fill=black, inner sep=1.5pt}]

    \node (q)  at (0, 1.5) {};
    \node (s)  at (0, 1) {};
    \node (qp) at (0, .5) {};
    \node (sp) at (0, 0) {};
  
    \node (r)  at (1, 1.5) {};
    \node (t)  at (1, 1) {};
    \node (rp)  at (1, .5) {};
    \node (tp)  at (1, 0) {};
  
    \node[draw=none, fill=none, left=10pt of q]  {$q$};
    \node[draw=none, fill=none, left=10pt of s]  {$s$};
    \node[draw=none, fill=none, left=8pt of qp] {$q'$};
    \node[draw=none, fill=none, left=8pt of sp] {$s'$};
  
    \node[draw=none, fill=none, right=8pt of r]  {$r$};
    \node[draw=none, fill=none, right=8pt of t]  {$t$};
    \node[draw=none, fill=none, right=8pt of rp]  {$r'$};
    \node[draw=none, fill=none, right=8pt of tp]  {$t'$};

    \draw[line width = .1em] (q) edge (t);
    \draw[line width = .1em] (q) edge (rp);
    \draw[line width = .1em] (q) edge (tp);

    \draw[line width = .1em] (s) edge (r);
    \draw[line width = .1em] (s) edge (rp);
    \draw[line width = .1em] (s) edge (tp);

    \draw[line width = .1em] (qp) edge (r);
    \draw[line width = .1em] (qp) edge (t);
    \draw[line width = .1em] (qp) edge (tp);

    \draw[line width = .1em] (sp) edge (r);
    \draw[line width = .1em] (sp) edge (t);
    \draw[line width = .1em] (sp) edge (rp);

    \draw[line width = .1em] (q) edge (s);
    \draw[line width = .1em] (s) edge (qp);
    \draw[line width = .1em] (qp) edge (sp);
    \draw[line width = .1em] (q) to[bend right=45] (sp);
    \draw[line width = .1em] (s) to[bend right=45] (sp);
    \draw[line width = .1em] (q) to[bend right=45] (qp);

    \draw[line width = .1em] (r) edge (t);
    \draw[line width = .1em] (t) edge (rp);
    \draw[line width = .1em] (rp) edge (tp);
    \draw[line width = .1em] (r) to[bend left=45] (tp);
    \draw[line width = .1em] (r) to[bend left=45] (rp);
    \draw[line width = .1em] (t) to[bend left=45] (tp);

  \end{tikzpicture}
  \caption{Diagrammatic representation of the matching $\{ (q,r), (s,t), (q',r'), (s',t') \}$.}
  \label{fig:matching}
\end{figure}
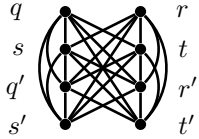  
For any matching on this graph, we weight edges between $\{q,s,q',s'\}$ by $\delta_{a b},$ edges between $\{r,t,r',t'\}$ by $(Z^{\top}Z)_{ab} + \delta_{a b} D_{aa}^2\kappa_2(\psi(\wt{x}))$ and crossing edges from $a$ to $b$ by $Z_{ab}$.  We define the weight of the matching as the product of the weights of the edges.  We then define the $8$-tensor 
\[
K_{qrstq'r's't'} \coloneqq \sum_{\mathfrak{m}} \operatorname{Weight}(\mathfrak{m})_{qrstq'r's't'}.
\]

Then we start by showing we can reduce attention to the $K$ terms.
\begin{lem}\label{lem:baizhou-0}
  We have that for any $\varepsilon >0$, with very high probability, uniformly in $i\in\{1,\dots,n\}$, 
  \[
  \frac{1}{n^2}\sum_{q,s,q',s'=1}^d\sum_{r,t,r',t'=1}^p \B_{qrst}\B_{q'r's't'}
  \left(
    \Eb\left[
      \Au_{qri}\Au_{sti}\Au_{q'r'i}\Au_{s't'i}
      \right]
      -
      K_{qrstq'r's't'}
  \right)
  =
  \O{d^{\varepsilon-1/2}}.
  \]
\end{lem}
\begin{proof}
  Using \eqref{eq:Au-moment-cumulant}, we can group the error terms into the following types:
  \[
  (4) \quad (3,1) \quad (2,1,1) \quad (1,1,1,1) \quad (2,2),
  \]
  which refers to the shape of the set partition.  For the first two types, there is no analogous term in the $K$ tensor; for the three types, we need to identify cancellation with the $K$ terms.  For all of these bounds, we will work on the event that 
  \begin{equation}\label{eq:Z-bound}
  |(Z^{\top}Z)_{\ell,u}| \leq d^{\varepsilon-(1-\delta_{\ell,u})/2}
  \quad \text{and} \quad
  |Z_{\ell,u}| \leq d^{\varepsilon-1/2}
  \quad \text{for all $\ell,u$},
  \end{equation}
  which holds with very high probability, as well as the events in Lemma \ref{lem:conditional-cumulants}.

  \paragraph{(4) terms:}  
  For these terms, we replace the contribution to $\Eb\left[
    \Au_{qri}\Au_{sti}\Au_{q'r'i}\Au_{s't'i}
    \right]$
  from the fourth cumulant contribution 
    \[
    \Eb\left[\kappa( \Au_{qri},\Au_{q'r'i},\Au_{sti},\Au_{s't'i} \mid \mathfrak{X})\right]
    = \Eb\left[x_{iq}x_{iq'}x_{is}x_{is'}\right] 
    \left( 
      \zeta_{rr'tt'} 
      + \omega^{(L)}_{rr'tt'}
      + \delta_{r r'}\delta_{r t} \delta_{r t'} D_{rr}^4 \kappa_4(\psi(\wt{x}_{ir}))
      \right),
    \]
  Using the moment-cumulant formula, we have that
  \[
  \Eb\left[ x_{iq}x_{iq'}x_{is}x_{is'} \right]
  = 
  \delta_{q q'}\delta_{s s'} + \delta_{q s}\delta_{q' s'} + \delta_{q s'}\delta_{q' s}
  + \kappa_4(x_{iq}) \delta_{q q'} \delta_{q s} \delta_{q s'}.  
  \]
  The $\kappa_4(x_{iq})$ terms will be only smaller, so we focus on the first terms, which we refer to as \emph{Wick terms}.
  These give contributions like
  \[
  \frac{1}{n^2}\sum_{q,q'=1}^d\sum_{r,t,r',t'=1}^p \B_{qrqt}\B_{q'r'q't'}
  \left( 
    \zeta_{rr'tt'} 
    +\omega^{(L)}_{rr'tt'}
  + \delta_{r r'}\delta_{r t} \delta_{r t'} D_{rr}^4 \kappa_4(\psi(\wt{x}_{ir}))
  \right). 
  \]
  \noindent (\textit{Bounding the $\kappa_4$ terms}): Bounding the entries of $\B$ by $\|\B\|_{\mathrm{op}}$ and $D$ by $d^{\varepsilon}$, we can bound the $\kappa_4$ terms by 
  \[
  \frac{1}{n^2}\sum_{r,t,r',t'=1}^p \left|\sum_{q,q'=1}^d\B_{qrqt}\B_{q'r'q't'}\left(\delta_{r r'}\delta_{r t} \delta_{r t'} D_{rr}^4 \kappa_4(\psi(\wt{x}_{ir}))\right)\right| \leq \frac{d^{2+4\varepsilon} p}{n^2}\|\B\|_{\mathrm{op}}^2 \kappa_4(\psi(\wt{x})) =  \O{\frac{d^{2+4\varepsilon}}{d}}.
  \]
  \noindent (\textit{Bounding the $\zeta$ terms}): We introduce the notation $(B^{(q,q')})_{rt} = \B_{qrq't}$.  Then we can express the $\zeta$ terms as  
  \[
    \frac{1}{n^2}\sum_{q,q'=1}^d\sum_{r,t,r',t'=1}^p \B_{qrqt}\B_{q'r'q't'} \zeta_{rr'tt'} 
    =
    \frac{1}{n^2}\sum_{q,q'=1}^d\sum_{r,t,r',t'=1}^p B^{(q,q)}_{rt}B^{(q',q')}_{r't'} \zeta_{rr'tt'} 
    = \frac{1}{n^2}
    \sum_{q,q'=1}^d
    \langle B^{(q,q)}, \zeta B^{(q',q')} \rangle,
  \]
  for an appropriate braid of the tensor $\zeta$ and where the multiplication is contraction over both free coordinates of $\B^{(q',q')}$. Then the vector norm of $\B^{(q,q')}$ is the Hilbert-Schmidt norm of the matrix $\B^{(q,q')}$, and hence
  \begin{equation}\label{eq:B-op-norm}
  \|B^{(q,q')}\|^2= \mathrm{Tr}\left(B^{(q,q')}B^{(q,q')\top}\right) \leq p \|B^{(q,q')}\|_{\mathrm{op}}^2 \leq p \|B\|_{\mathrm{op}}^2,
  \end{equation}
  where in the last line we have used that 
  \[
    \|B^{(q,q')}\|_{\mathrm{op}}
    = \sup_{\substack{\|v\|_2 = 1 \\ \|w\|_2 = 1}} \left| \langle B^{(q,q')}v, w \rangle \right|
    = \sup_{\substack{\|v\|_2 = 1 \\ \|w\|_2 = 1}} \left| \sum_{r,t} \B_{qrqt} v_r w_t \right|
    \leq \sup_{\substack{\|V\|_2 = 1 \\ \|W\|_2 = 1}} \left| \langle \B V, W \rangle \right|,
  \]
  where in the last step, $V$ and $W$ are in $\R^{d} \otimes \R^{p}$, and the inequality is by identification of $V$ with $\delta_q \otimes v$ and $W$ with $\delta_q' \otimes w$, which both have norm $1$.

  We then divide $\zeta = \zeta^{(0)} + \zeta^{(1)}$, where $\zeta^{(0)}_{rtr't'}$ is $0$ if $|\{r,t,r',t'\}|\leq 2$ and $\zeta^{(1)}_{rtr't'}$ is $0$ if $|\{r,t,r',t'\}|\geq 3$.  For the operator norm of $\zeta^{(0)}$, we can use the Schur test to get that 
  \[
  \| \zeta^{(0)} \|_{\mathrm{op}} \leq d^{\varepsilon-2} p^{1/2} pd. 
  \]
  Putting everything together, using \eqref{eq:B-op-norm}, we have that 
  \[
    \frac{1}{n^2}
    \sum_{q,q'=1}^d
    \left| \langle \B^{(q,q)}, \zeta^{(0)} \B^{(q',q')} \rangle \right|
    \leq C\frac{p^{5/2} d^{\epsilon+1}}{n^2} = \O{d^{\epsilon-1/2}}.
  \]
  As for the $\zeta^{(1)}$ terms, we bound the sum entrywise, i.e. 
  \[
    \frac{1}{n^2}\sum_{q,q'=1}^d\sum_{r,t,r',t'=1}^p \left|\B_{qrqt}\B_{q'r'q't'}\left( \zeta^{(1)}_{rr'tt'}\right)\right| \leq \frac{d^{2} p^2}{n^2}\|\B\|_{\mathrm{op}}^2 d^{\varepsilon-2}p = \O{{d^{\epsilon-1}}}.
  \]

  \noindent (\textit{Bounding the $\omega^{(L)}$ terms}): The $\omega^{(L)}$ terms can be bounded by
  \[
    \frac{1}{n^2}
    \sum_{q,q'=1}^d
    \left| \langle B^{(q,q)}, \omega^{(L)} B^{(q',q')} \rangle \right|
    \leq \frac{d^2}{n^2} p \|\B\|_{\mathrm{op}}^2 d^{\epsilon},
  \]
  where we have bounded the operator norm of $\omega^{(L)}$ by its Hilbert-Schmidt norm.x

  \paragraph{(3,1) terms:}
  For these terms, we replace the contribution to $\Eb\left[
    \Au_{qri}\Au_{sti}\Au_{q'r'i}\Au_{s't'i}
    \right]$
  from the third cumulant contribution and a first cumulant contribution
    \[
    \begin{aligned}
    \Eb\left[
      \kappa( \Au_{qri},\Au_{q'r'i},\Au_{sti} \mid \mathfrak{X})
      \kappa(\Au_{s't'i} \mid \right. &\left. \mathfrak{X})
    \right]
    = \Eb\left[x_{iq}x_{iq'}x_{is} (x_{is'}^2 - 1) \right] \left( 
      \zeta_{rr't}
      +\omega^{(L)}_{rr't}
      + \delta_{r r'}\delta_{r t} D_{rr}^3 \kappa_3(\psi(\wt{x}_{ir}))\right) Z_{s't'} \\
    &+ \Eb\biggl[x_{iq}x_{iq'}x_{is} x_{is'} \sum_{\substack{\ell \in L \\ \ell \neq s'}} x_{i\ell}Z_{\ell t'}\biggr] \left( 
      \zeta_{rr't'} 
      +\omega^{(L)}_{rr't'}
      + \delta_{r r'}\delta_{r t'} D_{rr}^3 \kappa_3(\psi(\wt{x}_{ir}))\right)\!. \\
    \end{aligned}
    \]
    For the first line, since the $x_{is'}^2-1$ is again centered, we still have the Wick terms and a single higher cumulant term when all $q=q'=s=s'$.  Hence this is the same as the fourth cumulant case.  Furthermore, $\zeta_{r,r',t} Z_{s't'}$ has the same entrywise bound as in the fourth cumulant case, i.e. $\O{d^{\epsilon-3/2}}$.  Hence the bound proceeds the same way.  For $\kappa_3(\psi(\wt{x}_{ir}))$ term, the cardinality of the terms over which we sum is $d^2p^2$; the summands can be bounded by $\|\B\|_{\mathrm{op}}^2 d^{\varepsilon-1/2}$, and the normalization is $n^{-2}$; hence these terms are $\O{d^{\epsilon-1/2}}$.

    For the second line, we first evaluate the expectation using the moment-cumulant formula; a representative term is 
    \[
    \Eb\biggl[x_{iq}x_{iq'}x_{is} x_{is'} x_{iq}\biggr]
    = 
    \kappa_2(x_{iq})\kappa_3(x_{iq'})\delta_{q's}\delta_{q's'}
    +\kappa_3(x_{iq})\left(\delta_{q q'}\delta_{s s'} + \delta_{q s}\delta_{q' s'} + \delta_{q s'}\delta_{q' s}\right) 
    +\kappa_5(x_{iq}) \delta_{q q'} \delta_{q s} \delta_{q s'}.     
    \]
    Up to the reweighting by the cumulants, all but the first term is the same as the first line up to reweighting by cumulants, and hence the bound is identical.  The first term must be treated separately.

    For these, we should bound 
    \[
      \frac{1}{n^2}\sum_{q,q'=1}^d\sum_{r,t,r',t'=1}^p \B_{qrq't}\B_{q'r'q't'}\left( \zeta_{rr't}Z_{q't'} + \delta_{r r'}\delta_{r t} D_{rr}^3 \kappa_3(\psi(\wt{x}_{ir}))Z_{q't'}\right). 
    \]
    The $\kappa_3$ terms can again be bounded by the cardinality of the sum and the max of the entries, which is unchanged from the earlier $\kappa_3$ cases.  For the $\zeta$ terms, we follow a similar argument as in the fourth cumulant case. We set for parallelism with that case, $\zeta^{(0),q'}_{rtr't'} = \zeta_{r,r',t}Z_{q',t'}$, in terms of which we have 
    \[
    \| \zeta^{(0),q'} \|_{\mathrm{op}} \leq d^{2\varepsilon-3/2} pd,
    \]
    which is the same order as in the fourth cumulant case up to factors of $d^{\varepsilon}$.  Hence the $\zeta^{(0)}$ terms take the form 
    \[
      \frac{1}{n^2}\sum_{q,q'=1}^d\sum_{r,t,r',t'=1}^p \B_{qrq't}\B_{q'r'q't'}\zeta^{(0),q'}_{rtr't'}
      = 
      \frac{1}{n^2}
      \sum_{q,q'=1}^d
      \langle \B^{(q,q')}, \zeta^{(0),q'} \B^{(q',q')} \rangle.
    \]
    This is now bounded the same way as the fourth cumulant case, i.e.\ using \eqref{eq:B-op-norm},
    \[
      \frac{1}{n^2}
      \sum_{q,q'=1}^d
      \left|\langle B^{(q,q')}, \zeta^{(0),q'} B^{(q',q')} \rangle\right|
      \leq C\frac{p^{2} d^{2\varepsilon+3/2}}{n^2} = \O{d^{2\varepsilon-1/2}}.
    \]

    \paragraph{(2,1,1) terms:}
    For these terms, we replace the contribution to $\Eb\left[
      \Au_{qri}\Au_{q'r'i}\Au_{sti}\Au_{s't'i}
      \right]$
    with second cumulant contribution and two first cumulant contributions.  This gives 
    \[
      \begin{aligned}
      &\Eb\left[
        \kappa( \Au_{qri},\Au_{q'r'i} \mid \mathfrak{X})
        \kappa(\Au_{sti} \mid \mathfrak{X})
        \kappa(\Au_{s't'i} \mid \mathfrak{X})
      \right] \\ 
      &= \Eb\left[x_{iq}x_{iq'}(x_{is}^2-1)(x_{is'}^2 - 1) \right] \left( (Z^{\top}Z)_{rr'} + \omega^{(L)}_{rr'} + \delta_{r r'} D_{rr}^2 \kappa_2(\psi(\wt{x}_{ir}))\right)Z_{st} Z_{s't'} \\
      &+ \Eb\biggl[
        x_{iq}x_{iq'}x_{is} x_{is'} 
        \sum_{\substack{\ell \in L \\ \ell \neq s}} x_{i\ell}Z_{\ell t}
        \sum_{\substack{\ell' \in L \\ \ell' \neq s'}} x_{i\ell'}Z_{\ell' t'}
      \biggr] \left( (Z^{\top}Z)_{rr'} + \omega^{(L)}_{rr'} + \delta_{r r'} D_{rr}^2 \kappa_2(\psi(\wt{x}_{ir}))\right) Z_{st} Z_{s't'} \\
      &+ \text{(cross terms)}.
      \end{aligned}
    \]
    where for the cross terms, we mix $(x_{is}^2-1)Z_{st}$ and $\sum_{\substack{\ell \in L \\ \ell \neq s'}} x_{i\ell}Z_{\ell t'}$ terms.

    We note that in the first line, evaluating the expectation, we again have Wick terms and a single higher cumulant term when all $q=q'=s=s'$, as in the fourth cumulant case.  For the $(r,r',t,t')$ part of the tensor, we set 
    \[
    \zeta^{(0,s,s')}_{rr'tt'} = (1-\delta_{rr'})\left((Z^{\top}Z)_{rr'} + \omega^{(L)}_{rr'}\right) Z_{st} Z_{s't'}
    \, \text{and} \,
    \zeta^{(1,s,s')}_{rr'tt'} = \delta_{r r'}\left((Z^{\top}Z)_{rr'} + \omega^{(L)}_{rr'} + D_{rr}^2 \kappa_2(\psi(\wt{x}_{ir}))\right) Z_{st} Z_{s't'}.
    \]
    Then we have using \eqref{eq:Z-bound} that 
    \[
    \| \zeta^{(0,s,s')}\|_{\mathrm{op}} \leq d^{3\varepsilon-3/2} pd
    \quad \text{and} \quad
    \max_{r,r',t,t'} | \zeta^{(1,s,s')}_{rr'tt'} | \leq d^{3\varepsilon-1}.
    \]
    The argument for the $\zeta^{(0,s,s')}$ terms is now similar to the previous cases; there are $d^2$ choices of $(q,q',s,s')$ which give nonzero contribution to the expectation.  The sum over $(r,r',t,t')$ we bound either using the operator norm of $\zeta^{(0,s,s')}$ or doing an entrywise bound on $\zeta^{(1,s,s')}$.

    There are new terms that arise in the second line which are matching terms for which there are two crossing edges and thus are canceled by the definition of $K$. Indeed, if we take as a representative term $\ell = q$ and $\ell' = q'$ we obtain the diagram term 
    \[
    \begin{tikzpicture}[every node/.style={circle, fill=black, inner sep=1.5pt}]

  \node (q)  at (0, 1.5) {};
  \node (s)  at (0, 1) {};
  \node (qp) at (0, .5) {};
  \node (sp) at (0, 0) {};

  \node (r)  at (1, 1.5) {};
  \node (t)  at (1, 1) {};
  \node (rp)  at (1, .5) {};
  \node (tp)  at (1, 0) {};

  \node[draw=none, fill=none, left=4pt of q]  {$q$};
  \node[draw=none, fill=none, left=4pt of s]  {$s$};
  \node[draw=none, fill=none, left=4pt of qp] {$q'$};
  \node[draw=none, fill=none, left=4pt of sp] {$s'$};

  \node[draw=none, fill=none, right=4pt of r]  {$r$};
  \node[draw=none, fill=none, right=4pt of t]  {$t$};
  \node[draw=none, fill=none, right=4pt of rp]  {$r'$};
  \node[draw=none, fill=none, right=4pt of tp]  {$t'$};

  \draw[dashed] (-0.5,.75) -- (1.5,.75);

  \draw[line width = .1em] (r) edge[bend left] (rp);
  \draw[line width = .1em] (q) edge (t);
  \draw[line width = .1em] (s) edge[bend right] (sp);
  \draw[line width = .1em] (qp) edge (tp);
\end{tikzpicture}
    \]
    The higher order terms are all higher order in the number of $\delta$-functions, and hence lead to tensors which are supported on at most $d^2$ choices of $(q,q',s,s')$.

      \noindent (\textit{Bounding the $\omega^{(L)}$ terms}): we still however need to bound the term involving $\omega^{(L)}$.
      Define the family of matrices for $\ell\in\{1,\dots,d\}$ and $t\in\{1,\dots,p\}$
      \[
      A^{(\ell,t)}_{qr} = Z_{\ell r}Z_{qt} = \left(\textbf{z}_\ell \otimes (\mathbf{z}^\top)_{t}\right)_{qr}
      \]
      whose norms are uniformly bounded with high probability by $d^{2\varepsilon}$. 
      The sum over $\ell$ takes the values $\ell=q,q',$or $s$. For $\ell=q$ or $q'$ the bounds are the same and we can write 
      \[
      \frac{1}{n^2}\sum_{q,q',s=1}^d
      \sum_{r,t,r't'=1}^p 
      \B_{qrst}\B_{q'r'st'}
      Z_{qt}Z_{q't'}Z_{qr}Z_{qr'}.
      \]
      We define, similarly as before, $A^{(q,t')}_{q'r'}=Z_{q't'}Z_{qr'}$ then the term becomes 
      \[
      \frac{1}{n^2}\sum_{q,q',s=1}^d\sum_{r,t,r',t'=1}^p 
      \B_{qrst}\B_{q'r'st'}Z_{qt}Z_{qr}A^{(q,t')}_{q'r'}
      =
      \frac{1}{n^2}\sum_{q,s=1}^d \sum_{r,t,t'=1}^p 
      \B_{qrst}Z_{qt}Z_{qr}
      \left(
        A^{(q,t')}\B
      \right)_{st'}.
      \]
      We now consider $B^{(q,s)}_{rt}=\B_{qrst}$ and the term now can be written as, with $\mathbf{z}_q$ the $q$-th row of $Z$ whose $\ell^2$-norm is bounded by $d^{\varepsilon}$ with very high probability, 
      \[
      \left\vert
      \frac{1}{n^2}\sum_{q,s=1}^d 
      \left(
        \mathbf{z}_q B^{(q,s)}\mathbf{z}_q^\top
      \right)
      \sum_{t'=1}^p 
      \left(
        A^{(q,t')}\B
      \right)_{st'}
      \right\vert
      \leqslant
      \frac{d^{2\varepsilon}p}{n^2}
      \sum_{q=1}^d 
      \sum_{t'=1}^p
      \left\| 
        \left(
          A^{{(q,t')}}\B
        \right)_{\cdot t'}
      \right\|_2^2
      =
      \O{
        \frac{d^{1+4\varepsilon}p^2}{n^2}
      }
      \]  
    For the case $\ell =s$, the term becomes 
    \[
    \begin{aligned}
    \frac{1}{n^2}
    \sum_{q,q',s=1}^d\sum_{r,t,r',t'=1}^p \B_{qrst}\B_{q'r'st'}Z_{qt}Z_{q't'}Z_{sr}Z_{sr'}
    =
    \frac{1}{n^2}
    \sum_{s=1}^d\left(
      \sum_{t=1}^p \left(
        A^{(s,t)}\B
      \right)_{st}
    \right)^2
    & \leqslant
    \frac{d^{1+2\varepsilon}p^2}{n^2}
    \end{aligned}
    \]
    using the fact that for uniformly in $s$ and $t$, the matrix $A^{(s,t)}\B$ has an operator norm bounded by $d^{\varepsilon}$.
    \paragraph{(1,1,1,1) terms:}
    For these terms, we replace the contribution to $\Eb\left[
      \Au_{qri}\Au_{sti}\Au_{q'r'i}\Au_{s't'i}
      \right]$
    with four first cumulant contributions.  This gives 
    \[
      \begin{aligned}
      &\Eb\left[
        \kappa(\Au_{qri} \mid \mathfrak{X})
        \kappa(\Au_{q'r'i} \mid \mathfrak{X})
        \kappa(\Au_{sti} \mid \mathfrak{X})
        \kappa(\Au_{s't'i} \mid \mathfrak{X})
      \right].
      \end{aligned}
    \]
    In comparison to the previous cases, we gain a new term which is not sparse in $(q,q',s,s')$.  In the case that $|\{q,q',s,s'\}|=4$, we have the term only involving the sums over $L$ when unfolding the product as per the computation of the first cumulant \eqref{eq:first-cumulant}
    \[
      \begin{aligned}
        \sum_{\substack{\ell_1 \in L \\ \ell_1 \neq q}}
        \sum_{\substack{\ell_2 \in L \\ \ell_2 \neq q'}}
        \sum_{\substack{\ell_3 \in L \\ \ell_3 \neq s}}
        \sum_{\substack{\ell_4 \in L \\ \ell_4 \neq s'}}
        \Eb\left[
          x_{iq}x_{iq'}x_{is}x_{is'}x_{i\ell_1}x_{i\ell_2}x_{i\ell_3}x_{i\ell_4}
        \right]
        Z_{\ell_1 r}Z_{\ell_2 r'}Z_{\ell_3 t}Z_{\ell_4 t'}
      \end{aligned}
    \]
    where in the sum we also have the restriction $|\{\ell_1,\ell_2,\ell_3,\ell_4\}|=4.$
    This is a matching term in which all edges cross from left to right and is thus canceled by the definition of $K$. It involves diagrams such as
    \[
    \begin{tikzpicture}[every node/.style={circle, fill=black, inner sep=1.5pt}]

  \node (q)  at (0, 1.5) {};
  \node (s)  at (0, 1) {};
  \node (qp) at (0, .5) {};
  \node (sp) at (0, 0) {};

  \node (r)  at (1, 1.5) {};
  \node (t)  at (1, 1) {};
  \node (rp)  at (1, .5) {};
  \node (tp)  at (1, 0) {};

  \node[draw=none, fill=none, left=4pt of q]  {$q$};
  \node[draw=none, fill=none, left=4pt of s]  {$s$};
  \node[draw=none, fill=none, left=4pt of qp] {$q'$};
  \node[draw=none, fill=none, left=4pt of sp] {$s'$};

  \node[draw=none, fill=none, right=4pt of r]  {$r$};
  \node[draw=none, fill=none, right=4pt of t]  {$t$};
  \node[draw=none, fill=none, right=4pt of rp]  {$r'$};
  \node[draw=none, fill=none, right=4pt of tp]  {$t'$};

  \draw[dashed] (-0.5,.75) -- (1.5,.75);

  \draw[line width = .1em] (q) edge (t);
  \draw[line width = .1em] (s) edge (tp);
  \draw[line width = .1em] (qp) edge (r);
  \draw[line width = .1em] (sp) edge (rp);
\end{tikzpicture}
\] 
If $\vert \{q,q',s,s'\}\neq 4$, then there are terms involving the two parts of the first cumulant in \eqref{eq:first-cumulant} but they are smaller due do the additional identification. For instance, if one consider $q=q'$, then one can get a term 
\[
\frac{1}{n^2}
  \sum_{q,s,s'=1}^d\sum_{r,t,r',t'=1}^p
  \B_{qrst}\B_{qr's't'}
  \sum_{\substack{\ell_1 \in L \\ \ell_1 \neq q}}
  \sum_{\substack{\ell_3 \in L \\ \ell_3 \neq s}}
  \sum_{\substack{\ell_4 \in L \\ \ell_4 \neq s'}}
  \Eb\left[
    (x_{iq}^2-1)x_{iq}x_{is}x_{is'}x_{i\ell_1}x_{i\ell_2}x_{i\ell_3}
  \right]
  Z_{qr}Z_{\ell_1 r'}Z_{\ell_2t}Z_{\ell_3t'}
\]
and due to the centering of the entries of $X$ there must be some identification between the $\ell$ and $\{s,s'\}$. For instance, we can  have the term, with $C_x = \Eb[x_{i_q}^3-x_{iq}]$, 
\[
\begin{aligned}
\frac{C_x}{n^2}
\sum_{q,s,s'=1}^d \sum_{r,t,r',t'=1}^p 
\B_{qrst}\B_{qr's't'}
Z_{qr}Z_{sr'}Z_{s't}Z_{qt'}
&=
\frac{C_x}{n^2}
\sum_{q=1}^d \sum_{r=1}^p \B_{qrst}Z_{qr}\sum_{s=1}^d 
\left(
  \left(
    \B A^{(q,t)}
  \right)Z^\top
\right)_{qs}
\\& = 
C_x\frac{d^{1+2\varepsilon}}{n^2}
\sum_{q=1}^d 
\sqrt{\sum_{r=1}^p \B_{qrst}^2}
\sqrt{\sum_{r=1}^p Z_{qr}^2}
\\&=
\O{
  \frac{d^{2+3\varepsilon}}{n^2}
}.
\end{aligned}
\]
Other terms can be bounded similarly.
    \paragraph{(2,2) terms:}
    For these terms, we replace the contribution to $\Eb\left[
      \Au_{qri}\Au_{sti}\Au_{q'r'i}\Au_{s't'i}
      \right]$
    with two second cumulant contributions.  A representative term is 
    \[
    \begin{aligned}
    \Eb\left[
    \kappa(\Au_{qri},\Au_{q'r'i} \mid \mathfrak{X})
    \kappa(\Au_{sti},\Au_{s't'i} \mid \mathfrak{X})
    \right]
    =
    \Eb\left[
      x_{iq}x_{iq'}x_{is}x_{is'}
      \right]
      &\left(
        (Z^{\top}Z)_{rr'} + \omega^{(L)}_{rr'} + \delta_{r r'} D_{rr}^2 \kappa_2(\psi(\wt{x}_{ir}))
      \right)\\
    \times&\left(
        (Z^{\top}Z)_{tt'} + \omega^{(L)}_{tt'} + \delta_{t t'} D_{tt}^2 \kappa_2(\psi(\wt{x}_{it}))
      \right).
    \end{aligned}
    \]
    These contribute matching terms in which all edges on the right side are matched.  Higher order terms are bounded in a the same fashion as the other cases. It remains to bound the terms involving $\omega^{(L)}$ but this can done directly since for general matrices $M^{(L)}$ and $N^{(L)}$ of operator norm bounded by $d^{\varepsilon}$ we can obtain the bound 
    \begin{multline*}
      \frac{1}{n^2}\sum_{q,s,q',s'=1}^d \sum_{r,t,r',t'=1}^p 
      \Eb\left[
        x_{iq}x_{iq'}x_{is}x_{is'}
      \right]
      \B_{qrst}\B_{q'r's't'}
      M^{(L)}_{rr'}N^{(L)}_{tt'}
      \\=
      \frac{1}{n^2}
      \sum_{q,s,q',s'=1}^d 
      \Eb\left[
        x_{iq}x_{iq'}x_{is}x_{is'} 
      \right]
      \mathrm{Tr}\left(
        B^{(q,s)}N^{(L)}{B^{(q',s')}}^\top M^{(L)}
      \right).
    \end{multline*}
Now, by independence of the entries of $X$, we must have at least one (and thus two) identifications in the set $L=\{q,s,q',s'\}$ which then gives a bound of order $\O{\frac{d^{2\varepsilon+3}}{n^2}}$.
\end{proof}
Now that we have removed all terms besides the $(2,2)$ terms, we show that these converge to the correct limit.
\begin{lem}\label{lem:baizhou-1}
  We have that for any $\varepsilon >0$, with very high probability, uniformly in $i\in\{1,\dots,n\}$, 
  \[
  \frac{1}{n^2}\sum_{q,s,q',s'=1}^d\sum_{r,t,r',t'=1}^p \B_{qrst}\B_{q'r's't'}
  \left(
      K_{qrstq'r's't'}
      -
      \Q_{qrst}\Q_{q'r's't'}
  \right)
  =
  \O{d^{4\varepsilon-\frac{1}{2}}}.
  \]
\end{lem}
\begin{proof}
Using the diagrammatic representation from Figure \ref{fig:matching}, we need to bound the different types of perfect matchings. We note also that due to the centering of $\Q_{qrst}\Q_{q'r's't'}$ there must be at least an edge (and thus at least two) crossing between the top $\{q,r,s,t\}$ and bottom indices $\{q',r',s',t'\}$. This is illustrated by adding a dashed line between these indices. Many such perfect matchings are bounded similarly and we now give some typical examples of matchings whose bounding techniques can be applied to all. We now consider matchings where only two edges cross the line from top to bottom indices, we consider first the three following types
 \[
 \mathfrak{T}_{LL}:\quad\vcenter{\hbox{\begin{tikzpicture}[every node/.style={circle, fill=black, inner sep=1.5pt}]

  \node (q)  at (0, 1.5) {};
  \node (s)  at (0, 1) {};
  \node (qp) at (0, .5) {};
  \node (sp) at (0, 0) {};

  \node (r)  at (1, 1.5) {};
  \node (t)  at (1, 1) {};
  \node (rp)  at (1, .5) {};
  \node (tp)  at (1, 0) {};

  \node[draw=none, fill=none, left=4pt of q]  {$q$};
  \node[draw=none, fill=none, left=4pt of s]  {$s$};
  \node[draw=none, fill=none, left=4pt of qp] {$q'$};
  \node[draw=none, fill=none, left=4pt of sp] {$s'$};

  \node[draw=none, fill=none, right=4pt of r]  {$r$};
  \node[draw=none, fill=none, right=4pt of t]  {$t$};
  \node[draw=none, fill=none, right=4pt of rp]  {$r'$};
  \node[draw=none, fill=none, right=4pt of tp]  {$t'$};

  \draw[dashed] (-0.5,.75) -- (1.5,.75);

  \draw[line width = .1em] (q) edge[bend left] (qp);
  \draw[line width = .1em] (s) edge[bend right] (sp);
  \draw[line width = .1em] (r) edge (t);
  \draw[line width = .1em] (rp) edge (tp);
\end{tikzpicture}}},
\quad 
\mathfrak{T}_{RR}:\quad 
\vcenter{\hbox{
  \begin{tikzpicture}[every node/.style={circle, fill=black, inner sep=1.5pt}]

  \node (q)  at (0, 1.5) {};
  \node (s)  at (0, 1) {};
  \node (qp) at (0, .5) {};
  \node (sp) at (0, 0) {};

  \node (r)  at (1, 1.5) {};
  \node (t)  at (1, 1) {};
  \node (rp)  at (1, .5) {};
  \node (tp)  at (1, 0) {};

  \node[draw=none, fill=none, left=4pt of q]  {$q$};
  \node[draw=none, fill=none, left=4pt of s]  {$s$};
  \node[draw=none, fill=none, left=4pt of qp] {$q'$};
  \node[draw=none, fill=none, left=4pt of sp] {$s'$};

  \node[draw=none, fill=none, right=4pt of r]  {$r$};
  \node[draw=none, fill=none, right=4pt of t]  {$t$};
  \node[draw=none, fill=none, right=4pt of rp]  {$r'$};
  \node[draw=none, fill=none, right=4pt of tp]  {$t'$};

  \draw[dashed] (-0.5,.75) -- (1.5,.75);
\draw[line width = .1em] (r) edge[bend left] (tp);
  \draw[line width = .1em] (sp) edge (qp);
  \draw[line width = .1em] (q) edge (s);
  \draw[line width = .1em] (rp) edge (t);
\end{tikzpicture}
}},
\quad \mathfrak{T}_{LR} :\quad
\vcenter{\hbox{
  \begin{tikzpicture}[every node/.style={circle, fill=black, inner sep=1.5pt}]

  \node (q)  at (0, 1.5) {};
  \node (s)  at (0, 1) {};
  \node (qp) at (0, .5) {};
  \node (sp) at (0, 0) {};

  \node (r)  at (1, 1.5) {};
  \node (t)  at (1, 1) {};
  \node (rp)  at (1, .5) {};
  \node (tp)  at (1, 0) {};

  \node[draw=none, fill=none, left=4pt of q]  {$q$};
  \node[draw=none, fill=none, left=4pt of s]  {$s$};
  \node[draw=none, fill=none, left=4pt of qp] {$q'$};
  \node[draw=none, fill=none, left=4pt of sp] {$s'$};

  \node[draw=none, fill=none, right=4pt of r]  {$r$};
  \node[draw=none, fill=none, right=4pt of t]  {$t$};
  \node[draw=none, fill=none, right=4pt of rp]  {$r'$};
  \node[draw=none, fill=none, right=4pt of tp]  {$t'$};

  \draw[dashed] (-0.5,.75) -- (1.5,.75);

  \draw[line width = .1em] (r) edge[bend left] (rp);
  \draw[line width = .1em] (q) edge (t);
  \draw[line width = .1em] (s) edge[bend right] (sp);
  \draw[line width = .1em] (qp) edge (tp);
\end{tikzpicture}
}} 
\]
We recall that the weight given to a edge between $\{q,s,q',s'\}$ gives a $\delta_{ab}$, between $\{r,t,r',t'\}$ gives a $(Z^\top Z)_{ab}+\delta_{ab}D_{aa}^2\kappa_2(\psi(\tilde{x}))\eqqcolon M_{ab}$, and crossing between the two gives a $Z_{ab}$. Thus, bound the matchings of type $\mathfrak{T}_{LL}$ as above gives,
\[
\begin{aligned}
\frac{1}{n^2}\sum_{q,s=1}^d\sum_{r,t,r't'=1}^p\B_{qrst}\B_{qr'st'}M_{rt}M_{r't'}
=
\frac{1}{n^2}\sum_{q,s=1}^d \left(
  \sum_{r,t=1}^p \B_{qrst}M_{rt}
\right)^2
&\leqslant
\frac{1}{n^2}\sum_{q,s=1}^d 
\left(
  \sum_{r,t=1}^p \B_{qrst}^2
\right)
\left(
  \sum_{r,t=1}^ p M_{rt}^2
\right)
\\
& 
=
\frac{1}{n^2}\Vert M\Vert_{\mathrm{F}}^2\Vert \B\Vert_{\mathrm{F}}^2 \\
&\leqslant
\frac{p^2d}{n^2}\Vert M\Vert^2_{\mathrm{op}}\Vert \B\Vert^2_{\mathrm{op}}
=
\O{\frac{d^{\varepsilon}}{d}} 
\end{aligned}
\]
with very high probability. A matching of type $\mathfrak{T}_{RR}$ as above can be bounded in the following way, 
\[
\frac{1}{n^2}\sum_{q,q'=1}^d\sum_{r,t,r',t'=1}^p \B_{qrqt}\B_{q'r'q't'}M_{rt'}M_{tr'} =
\frac{1}{n^2}\mathrm{Tr}\left(
  \mathrm{Tr}_1(\B)M\mathrm{Tr}_1(\B)M
\right)
\]
where we introduced the notation
\(
\mathrm{Tr}_1(\B)_{rt} =\sum_{q=1}^d \B_{qrqt}.
\)
We can then bound this term as 
\[
  \frac{1}{n^2}\mathrm{Tr}\left(
    \mathrm{Tr}_1(\B)M\mathrm{Tr}_1(\B)M
  \right)
  \leqslant
  \frac{p}{n^2}\Vert \mathrm{Tr}_1(\B)\Vert_{\mathrm{op}}^2\Vert M\Vert_{\mathrm{op}}^2
  =
  \frac{pd^2}{n^2}\Vert\B\Vert_{\mathrm{op}}^2\Vert M\Vert_{\mathrm{op}}^2
  =
  \O{\frac{d^\varepsilon}{d}}
\]
with very high probability. Finally, for graphs of the type $\mathfrak{T}_{LR}$, we obtain the term 
\[
\frac{1}{n^2}\sum_{q,s,q'=1}^d \sum_{r,t,r',t'=1}^p \B_{qrst}\B_{qrst}\B_{q'r'st'}Z_{qt}Z_{q't'}M_{rr'}
\] 
    Define the matrix
    \[
    V_{rr'} = \sum_{q,q',s=1}^d\sum_{t,t'=1}^p \B_{qrst}\B_{q'r'st'}Z_{q t}Z_{q' t'}.
    \]
    We will bound the operator norm of this matrix. Then 
    \[
    \begin{aligned}
    \|V\|_{\mathrm{op}} 
    = \sup_{\|u\|=1} \left|\langle u, Vu\rangle\right|  
    &= \sup_{\|u\|=1}
    \left| \sum_{q,q',s=1}^d\sum_{r,t,r',t'=1}^p \B_{qrst}\B_{q'r'st'}Z_{q t}Z_{q' t'}u_r u_{r'}\right|. \\ 
    &= \sup_{\|u\|=1}
    \left| \sum_{s=1}^d 
      \left(
        \sum_{q=1}^d\sum_{r,t=1}^p \B_{qrst}Z_{q t}u_r
      \right)^2
    \right|. \\ 
    &\leq \sup_{\|u\|=1}
      \sum_{s=1}^d 
      \sum_{t=1}^p
      p
      \left(
        \sum_{q=1}^d\sum_{r=1}^p \B_{qrst}Z_{q t}u_r
      \right)^2. \\ 
    &\leq \sup_{\|u\|=1} p^2 \max_{t}\| \B^T(Z^{(t)} \otimes u)\|^2,
    \end{aligned}
    \]
    where $Z^{(t)}_q = Z_{qt}$.  The norm of the vector $Z^{(t)}\otimes u$ is bounded by $d^{\varepsilon}$ by \eqref{eq:Z-bound}, and hence we conclude $\|M\|_{\mathrm{op}} \leq p^2 d^{\varepsilon}$.  Hence, we conclude 
    \begin{equation}\label{eq:boundingLR}
    \begin{aligned}
      \frac{1}{n^2}\sum_{q,q',s=1}^d\sum_{r,t,r',t'=1}^p \B_{qrst}\B_{q'r'st'}Z_{q t}Z_{q' t'}
      M_{rr'}
      &= \frac{1}{n^2}
      \mathrm{Tr}(VM)
      = \O{d^{\epsilon-1}}.
    \end{aligned}
    \end{equation}

We have two distinct types where the dashed line is being crossed by two crossing from left to right.
\begin{equation}\label{eq:3graphs}
  \mathfrak{T}_{XX_1}:\quad
  \vcenter{\hbox{
 \begin{tikzpicture}[every node/.style={circle, fill=black, inner sep=1.5pt}]

  \node (q)  at (0, 1.5) {};
  \node (s)  at (0, 1) {};
  \node (qp) at (0, .5) {};
  \node (sp) at (0, 0) {};

  \node (r)  at (1, 1.5) {};
  \node (t)  at (1, 1) {};
  \node (rp)  at (1, .5) {};
  \node (tp)  at (1, 0) {};

  \node[draw=none, fill=none, left=4pt of q]  {$q$};
  \node[draw=none, fill=none, left=4pt of s]  {$s$};
  \node[draw=none, fill=none, left=4pt of qp] {$q'$};
  \node[draw=none, fill=none, left=4pt of sp] {$s'$};

  \node[draw=none, fill=none, right=4pt of r]  {$r$};
  \node[draw=none, fill=none, right=4pt of t]  {$t$};
  \node[draw=none, fill=none, right=4pt of rp]  {$r'$};
  \node[draw=none, fill=none, right=4pt of tp]  {$t'$};

  \draw[dashed] (-0.5,.75) -- (1.5,.75);

  \draw[line width = .1em] (q) edge (rp);
  \draw[line width = .1em] (s) edge (tp);
  \draw[line width = .1em] (qp) edge (sp);
  \draw[line width = .1em] (r) edge (t);
\end{tikzpicture}
  }},
  \quad 
  \mathfrak{T}_{XX_2}: \quad 
\vcenter{\hbox{
\begin{tikzpicture}[every node/.style={circle, fill=black, inner sep=1.5pt}]

  \node (q)  at (0, 1.5) {};
  \node (s)  at (0, 1) {};
  \node (qp) at (0, .5) {};
  \node (sp) at (0, 0) {};

  \node (r)  at (1, 1.5) {};
  \node (t)  at (1, 1) {};
  \node (rp)  at (1, .5) {};
  \node (tp)  at (1, 0) {};

  \node[draw=none, fill=none, left=4pt of q]  {$q$};
  \node[draw=none, fill=none, left=4pt of s]  {$s$};
  \node[draw=none, fill=none, left=4pt of qp] {$q'$};
  \node[draw=none, fill=none, left=4pt of sp] {$s'$};

  \node[draw=none, fill=none, right=4pt of r]  {$r$};
  \node[draw=none, fill=none, right=4pt of t]  {$t$};
  \node[draw=none, fill=none, right=4pt of rp]  {$r'$};
  \node[draw=none, fill=none, right=4pt of tp]  {$t'$};

  \draw[dashed] (-0.5,.75) -- (1.5,.75);

  \draw[line width = .1em] (q) edge (t);
  \draw[line width = .1em] (s) edge (tp);
  \draw[line width = .1em] (qp) edge (r);
  \draw[line width = .1em] (sp) edge (rp);
\end{tikzpicture}}}
\end{equation}
For the type $\mathfrak{T}_{XX_1}$, we have the term 
\[
\begin{aligned}
\frac{1}{n^2}\sum_{q,s,q'=1}^d\sum_{r,t,r',t'}
\B_{qrst}\B_{q'r'q't'}
M_{rt}Z_{qr'}Z_{st'}
&=
\frac{1}{n^2}
\sum_{q,s=1}^d \left(
  \sum_{r,t=1}^pB^{(q,s)}_{rt}M_{rt}
\right)
\left(
  Z\mathrm{Tr}_1(\B)Z^\top
\right)_{qs}
\\&=
\frac{1}{n^2}\sum_{q,s=1}^d \mathrm{Tr}\left(
  B^{(q,s)}M^\top
\right)
\left(
  Z\mathrm{Tr}_1(\B)Z^\top
\right)_{qs}
\\&\leqslant
\frac{1}{n^2}
\sqrt{
  \sum_{q,s=1}^d \mathrm{Tr}\left(
    B^{(q,s)}M^\top
  \right)^2
}
\sqrt{
  \sum_{q,s=1}^d 
  (Z\mathrm{Tr}_1(\B)Z^\top)^2_{qs}
}
\end{aligned}
\]
with the notation $B^{(q,s)}_{rt} = \B_{qsrt}$ the contraction of the tensor. We can bound with high probability,
\[
\sum_{q,s=1}^d \mathrm{Tr}\left(B^{(q,s)}M^\top\right)^2 
\leqslant
d^{2+2\varepsilon }p^2
\]
and 
\[
\sum_{q,s=1}^d \left(
  Z\mathrm{Tr}_1(\B)Z^\top
\right)^2_{qs}
=
\Vert Z\mathrm{Tr}_1(\B)Z^\top \Vert_\mathrm{F}^2 
\leqslant
d\Vert Z\mathrm{Tr}_1(\B)Z^\top \Vert_{\mathrm{op}} 
\leqslant d^{2+2\varepsilon}.
\]
Finally, this gives us the very high probability bound 
\[
\frac{1}{n^2}\sum_{q,s,q'=1}^d \sum_{r,t,r',t'} \B_{qrst}\B_{q'r'q't'}M_{rt}Z_{qr'}Z_{st'}
\leqslant
\frac{d^{1+2\varepsilon}p}{n^2}
=
\O{
  \frac{d^{2\varepsilon}}{d}
}.
\]
The type $\mathfrak{T}_{XX_2}$ can be written as 
\[
\begin{aligned}
\frac{1}{n^2}\sum_{q,s,q',s'=1}^d \sum_{r,t,r',t'=1}^p 
\B_{qrst}
\B_{q'r's't'}
Z_{qt}Z_{st'}Z_{q'r}Z_{s'r'}
&=
\frac{1}{n^2}\sum_{s,q'=1}^d \sum_{t,r'=1}^p
\left(
  A^{(t,q')}\B
\right)_{st}
\left(
  A^{(r',s)}\B
\right)_{q'r'}
\\&\leqslant
\frac{1}{n^2}
\sqrt{
  \sum_{s,q'=1}^d\sum_{t,r'=1}^p 
  \left(
    A^{(t,q')}\B
  \right)_{st}^2
}
\sqrt{
  \sum_{s,q=1}^d\sum_{t,r'=1}^p 
  \left(
    \B A^{(r',s)}
  \right)_{q'r'}^2
}
\\&=
\O{
  \frac{d^{1+2\varepsilon}p^2}{n^2}   
}
\end{aligned}
\]
with very high probability, using the notation $A^{(t,q')}_{qr}=Z_{qt}Z_{q'r}$ and the fact that its operator norm is bounded with very high probability by $d^{\varepsilon}$. 
We now consider the final types of matchings coming from only two crossing of the dashed line, the type where there is one on the left or right and one accross left and right 
\[
  \mathfrak{T}_{LX}:\quad 
  \vcenter{\hbox{\begin{tikzpicture}[every node/.style={circle, fill=black, inner sep=1.5pt}]

    \node (q)  at (0, 1.5) {};
    \node (s)  at (0, 1) {};
    \node (qp) at (0, .5) {};
    \node (sp) at (0, 0) {};
  
    \node (r)  at (1, 1.5) {};
    \node (t)  at (1, 1) {};
    \node (rp)  at (1, .5) {};
    \node (tp)  at (1, 0) {};
  
    \node[draw=none, fill=none, left=4pt of q]  {$q$};
    \node[draw=none, fill=none, left=4pt of s]  {$s$};
    \node[draw=none, fill=none, left=4pt of qp] {$q'$};
    \node[draw=none, fill=none, left=4pt of sp] {$s'$};
  
    \node[draw=none, fill=none, right=4pt of r]  {$r$};
  \node[draw=none, fill=none, right=4pt of t]  {$t$};
  \node[draw=none, fill=none, right=4pt of rp]  {$r'$};
  \node[draw=none, fill=none, right=4pt of tp]  {$t'$};
  
    \draw[dashed] (-0.5,.75) -- (1.5,.75);
  
    \draw[line width = .1em] (q) edge[bend right] (sp);
    \draw[line width = .1em] (s) edge (r);
    \draw[line width = .1em] (qp) edge (t);
    \draw[line width = .1em] (rp) edge (tp);
  \end{tikzpicture}}},
  \quad 
  \mathfrak{T}_{XR}:\quad 
  \vcenter{\hbox{
    \begin{tikzpicture}[every node/.style={circle, fill=black, inner sep=1.5pt}]

    \node (q)  at (0, 1.5) {};
    \node (s)  at (0, 1) {};
    \node (qp) at (0, .5) {};
    \node (sp) at (0, 0) {};
  
    \node (r)  at (1, 1.5) {};
    \node (t)  at (1, 1) {};
    \node (rp)  at (1, .5) {};
    \node (tp)  at (1, 0) {};
  
    \node[draw=none, fill=none, left=4pt of q]  {$q$};
    \node[draw=none, fill=none, left=4pt of s]  {$s$};
    \node[draw=none, fill=none, left=4pt of qp] {$q'$};
    \node[draw=none, fill=none, left=4pt of sp] {$s'$};
  
    \node[draw=none, fill=none, right=4pt of r]  {$r$};
  \node[draw=none, fill=none, right=4pt of t]  {$t$};
  \node[draw=none, fill=none, right=4pt of rp]  {$r'$};
  \node[draw=none, fill=none, right=4pt of tp]  {$t'$};
  
    \draw[dashed] (-0.5,.75) -- (1.5,.75);
  
    \draw[line width = .1em] (q) edge (s);
    \draw[line width = .1em] (qp) edge (tp);
    \draw[line width = .1em] (sp) edge (r);
    \draw[line width = .1em] (rp) edge (t);
  \end{tikzpicture}
  }}
\]
For the type $\mathfrak{T}_{LX}$, we obtain the term, 
\[
\begin{aligned}
\frac{1}{n^2}\sum_{q,s,q'=1}^d \sum_{r,t,r',t'=1}^p 
\B_{qrst}\B_{q'r'qt'}M_{r't'}Z_{sr}Z_{q't}
&=
\frac{1}{n^2}\sum_{q,s,q'=1}^d \left(
  ZB^{(q,s)}Z^\top
\right)_{sq'}\mathrm{Tr}\left(
  B^{(q',q)}M^\top
\right)
\\& \leqslant
\frac{1}{n^2}\sqrt{
  \sum_{q,s,q'=1}^d \left(
    ZB^{(q,s)}Z^\top
  \right)_{sq'}^2
}
\sqrt{
  \sum_{q,s,q'=1}^d 
  \mathrm{Tr}\left(B^{(q',q)}M^\top\right)^2
}
\end{aligned}
\]
Now we can see that 
\[
\sum_{q,s=1}^d \sum_{q'=1}^d \left(
  ZB^{(q,s)}Z^\top
\right)_{sq'}^2 
\leqslant
d^{2+2\varepsilon}
\]
with very high probability and we have that, with very high probability, 
\[
\sum_{q,s,q'=1}^d \mathrm{Tr}\left(
  B^{(q',q)}M^\top
\right)^2
\leqslant {p^2d^{3+2\varepsilon}}.
\]
Altogether, this gives the very high probability bound 
\[
  \frac{1}{n^2}\sum_{q,s,q'=1}^d \sum_{r,t,r',t'=1}^p 
  \B_{qrst}\B_{q'r'qt'}M_{r't'}Z_{sr}Z_{q't}
  \leqslant \frac{d^{\frac{5}{2}+2\varepsilon}p}{n^2}
  =
  \O{
    \frac{d^{2\varepsilon}}{\sqrt{d}}
  }.
\]
Considering the type $\mathfrak{T}_{XR}$, we need to bound the term 
\[
\begin{aligned}
\frac{1}{n^2}
\sum_{q,q',s'=1}^d\sum_{r,t,r',t'=1}^p
\B_{qrqt}\B_{q'r's't'}M_{r't}Z_{q't'}Z_{s'r}
&=
\frac{1}{n^2}
\sum_{q',s'=1}^d\sum_{r',t'=1}^p
\left(
  Z\mathrm{Tr}_1(\B)M
\right)_{s'r'}
\B_{q'r's't'}\\
&=
\frac{\sqrt{pd}}{n^2}\Vert \B\Vert_{\mathrm{F}}\Vert Z\mathrm{Tr}_1(\B)M\Vert_{\mathrm{F}}
\leqslant \frac{p\sqrt{p}d^{2+2\varepsilon}}{n^2}=
\O{
  \frac{d^{2\varepsilon}}{\sqrt{d}}
}
\end{aligned}
\]
where we used the fact that $\Vert \mathrm{Tr}_1(\B)\Vert =\O{d}$ and that the operator of $Z$ and $M$ are bounded by $\d^{\varepsilon}$ with very high probability.

We now consider matchings where all four edges cross the dashed line. They can either cross the line vertically, horizontally, or a mix between the two. We thus show the different bounds on three graphs typical of these behaviors.
\begin{equation}\label{eq:3graphs-2}
 \mathfrak{T}_{LLRR}\, : \quad  
 \vcenter{\hbox{
 \begin{tikzpicture}[every node/.style={circle, fill=black, inner sep=1.5pt}]
 
   \node (q)  at (0, 1.5) {};
   \node (s)  at (0, 1) {};
   \node (qp) at (0, .5) {};
   \node (sp) at (0, 0) {};
 
   \node (r)  at (1, 1.5) {};
   \node (t)  at (1, 1) {};
   \node (rp)  at (1, .5) {};
   \node (tp)  at (1, 0) {};
 
   \node[draw=none, fill=none, left=4pt of q]  {$q$};
   \node[draw=none, fill=none, left=4pt of s]  {$s$};
   \node[draw=none, fill=none, left=4pt of qp] {$q'$};
   \node[draw=none, fill=none, left=4pt of sp] {$s'$};
 
   \node[draw=none, fill=none, right=4pt of r]  {$r$};
  \node[draw=none, fill=none, right=4pt of t]  {$t$};
  \node[draw=none, fill=none, right=4pt of rp]  {$r'$};
  \node[draw=none, fill=none, right=4pt of tp]  {$t'$};
 
   \draw[dashed] (-0.5,.75) -- (1.5,.75);
 
   \draw[line width = .1em] (q) edge[bend left] (qp);
   \draw[line width = .1em] (s) edge[bend right] (sp);
   \draw[line width = .1em] (r) edge[bend right] (tp);
   \draw[line width = .1em] (t) edge (rp);
 \end{tikzpicture}
 }}
 \qquad
 \mathfrak{T}_{XXXX}\,:\quad 
 \vcenter{\hbox{
  \begin{tikzpicture}[every node/.style={circle, fill=black, inner sep=1.5pt}]
 
   \node (q)  at (0, 1.5) {};
   \node (s)  at (0, 1) {};
   \node (qp) at (0, .5) {};
   \node (sp) at (0, 0) {};
 
   \node (r)  at (1, 1.5) {};
   \node (t)  at (1, 1) {};
   \node (rp)  at (1, .5) {};
   \node (tp)  at (1, 0) {};
 
   \node[draw=none, fill=none, left=4pt of q]  {$q$};
   \node[draw=none, fill=none, left=4pt of s]  {$s$};
   \node[draw=none, fill=none, left=4pt of qp] {$q'$};
   \node[draw=none, fill=none, left=4pt of sp] {$s'$};
 
   \node[draw=none, fill=none, right=4pt of r]  {$r$};
  \node[draw=none, fill=none, right=4pt of t]  {$t$};
  \node[draw=none, fill=none, right=4pt of rp]  {$r'$};
  \node[draw=none, fill=none, right=4pt of tp]  {$t'$};
 
   \draw[dashed] (-0.5,.75) -- (1.5,.75);
 
   \draw[line width = .1em] (q) edge (rp);
   \draw[line width = .1em] (s) edge (tp);
   \draw[line width = .1em] (qp) edge (r);
   \draw[line width = .1em] (sp) edge (t);
 \end{tikzpicture}
 }}
 \qquad
 \mathfrak{T}_{LXXR}\,:\quad 
 \vcenter{\hbox{
 \begin{tikzpicture}[every node/.style={circle, fill=black, inner sep=1.5pt}]
 
   \node (q)  at (0, 1.5) {};
   \node (s)  at (0, 1) {};
   \node (qp) at (0, .5) {};
   \node (sp) at (0, 0) {};
 
   \node (r)  at (1, 1.5) {};
   \node (t)  at (1, 1) {};
   \node (rp)  at (1, .5) {};
   \node (tp)  at (1, 0) {};
 
   \node[draw=none, fill=none, left=4pt of q]  {$q$};
   \node[draw=none, fill=none, left=4pt of s]  {$s$};
   \node[draw=none, fill=none, left=4pt of qp] {$q'$};
   \node[draw=none, fill=none, left=4pt of sp] {$s'$};
 
   \node[draw=none, fill=none, right=4pt of r]  {$r$};
  \node[draw=none, fill=none, right=4pt of t]  {$t$};
  \node[draw=none, fill=none, right=4pt of rp]  {$r'$};
  \node[draw=none, fill=none, right=4pt of tp]  {$t'$};
 
   \draw[dashed] (-0.5,.75) -- (1.5,.75);
 
   \draw[line width = .1em] (q) edge[bend right] (sp);
   \draw[line width = .1em] (s) edge (rp);
   \draw[line width = .1em] (qp) edge (r);
   \draw[line width = .1em] (t) edge[bend left] (tp);
 \end{tikzpicture}
 }}
 \end{equation}
The type $\mathfrak{T}_{LLRR}$ from \eqref{eq:3graphs-2} gives the term 
\[
\frac{1}{n^2}\sum_{q,s=1}^d\sum_{r,t,r',t'=1}^p 
\B_{qrst}\B_{qr'st'}M_{rt'}M_{tr'}
=
\frac{1}{n^2}\sum_{q,s=1}^d 
\mathrm{Tr}\left(
  B^{(q,s)}MB^{(q,s)}M 
\right)
=
\O{
  \frac{d^{2+\varepsilon}p}{n^2}
}
=
\O{\frac{d^{\varepsilon}}{d}}
\]
with very high probability. The type $\mathfrak{T}_{XXXX}$ in \eqref{eq:3graphs-2} gives 
\[
\begin{aligned}
\frac{1}{n^2}\sum_{q,s,q',s'=1}^d\sum_{r,t,r',t'=1}^p
\B_{qrst}\B_{q'r's't'}Z_{qr'}Z_{st'}Z_{q'r}Z_{s't}
&=
\frac{1}{n^2}\sum_{q,s,q',s'=1}^d 
\left(
  ZB^{(q',s')}Z^\top
\right)_{qs}
\left(
  ZB^{(q,s)}Z^\top
\right)_{q's'}\\
&\leqslant
\frac{1}{n^2}\sum_{q,s,q',s'=1}^d \left(
  ZB^{(q's')}Z^\top
\right)_{qs}^2\\
&=
\frac{1}{n^2}\sum_{q's'=1}^d \Vert ZB^{(q',s')}Z^\top\Vert_{\mathrm{F}}^2 
=
\O{
  \frac{d^{3+\varepsilon}}{n^2}
}
\end{aligned}
\]
with very high probability. The final type $\mathfrak{T}_{LXXR}$ from \eqref{eq:3graphs-2} gives the term 
\[
\begin{aligned}
\frac{1}{n^2}\sum_{q,s,q'=1}^d \sum_{r,t,r',t'=1}^p 
\B_{qrst}\B_{q'r'qt'}Z_{sr'}Z_{q'r}M_{tt'}
&=
\frac{1}{n^2}\sum_{q,s,q'=1}^d 
\sum_{t'=1}^p
\left(
  ZB^{(q,s)}M
\right)_{q't'}
\left(
  Z{B^{(q',q)}}
\right)_{st'}
\\&\leqslant
\frac{1}{n^2}\sum_{q,s,q'=1}^d 
\sqrt{
  \sum_{t'=1}^p 
  \left(
    ZB^{(q,s)}M
  \right)_{q't'}^2
}
\sqrt{
  \sum_{t'=1}^p 
  \left(
    ZB^{(q',q)}
  \right)_{st'}^2
}
\\& = \O{
  \frac{d^{3+\varepsilon}}{n^2}
}
\end{aligned}
\]
with very high probability.
\end{proof}
\begin{proof}[Proof of Proposition \ref{prop:baizhouquadratic}] It suffices now to combine Lemmas \ref{lem:baizhou-0}, \ref{lem:baizhou-1} to get the final result.
  
\end{proof}
\section{Analysis of the tensor \texorpdfstring{$\Q$}{Q} and proof of the main result}\label{sec:analysisQ} 
From the previous section, we know that the limiting eigenvalue distribution of $\wt{K}$ is given by $\mu_{\mathrm{MP}}^{\gamma_1}\boxtimes \pi$ where $\pi$ is the asymptotic e.e.d of the tensor $\Q$ which is the leading part of the tensor $\T$. In this section we analyse the tensor $\Q$ to compute this measure $\pi$. We first start by writing,
\[
\Q = \frac{\alpha_\phi^2}{d}\Q_1 + \beta_\psi^2\Q_2,  
\]
with 
\[
\Q_1 = \tau_{23}\left(
  \Id_d\otimes DW^\top WD
\right)
+
\tau_{24}(WD\otimes WD)
\quad\text{and}\quad  
\Q_2 =  \tau_{23}\left(
  \Id_d\otimes D^2
\right).
\]
We also define the tensor 
\[
\D = \tau_{23}(\Id_d\otimes D).  
\]
We first notice that we can factorize the tensor $\Q$.
\begin{lem}\label{lem:factor}
  We can write 
  \[
  \Q = \D\left(
    \frac{\alpha_\phi^2}{d} \hat{\Q}_1  + \beta_\psi^2  \hat{\Q}_2
  \right)\D 
  \eqqcolon 
  \D\hat{\Q}\D
  \]
  with 
  \[
  \hat{\Q}_1 = \tau_{23}\left(
    \Id_d\otimes W^\top W
  \right)
  +
  \tau_{24}(W\otimes W),
  \quad\text{and}\quad 
  \hat{\Q}_2 = \tau_{23}(\Id_d\otimes \Id_p)=\mathcal{I}_{d,p}.  
  \]
\end{lem}
\begin{proof}
  The one contraction which is not direct comes from the braid operator $\tau_{24}$, but indeed we can compute 
  \begin{align*}
  \left( 
    \D\tau_{24}(W\otimes W)\D 
  \right)_{ijk\ell}
  &=
  \sum_{q,s=1}^d\sum_{r,t=1}^p
  \tau_{23}(\Id_d\otimes D)_{ijqr}
  \tau_{24}(W\otimes W)_{qrst}
  \tau_{23}(\Id_d\otimes D)_{stk\ell}
  \\
  &=
  \sum_{q,s=1}^d \sum_{r,t=1}^p \delta_{iq}\delta_{jr}D_{jj}W_{qt}W_{sr}\delta_{sk}\delta_{t\ell}D_{\ell\ell}
  \\
  &=
  W_{i\ell}D_{\ell}W_{kj}D_{jj}=(WD)_{i\ell}(WD)_{kj}=\tau_{24}(WD\otimes WD)_{ijk\ell}.
  \end{align*}
\end{proof}
We start by computing the spectrum of $\hat{\Q}$ still working conditionally on $W$ and $D$. We first recall the properties of the singular value decomposition of $W$ being from a rectangular Ginibre ensemble.

\begin{thm}\label{thm:svd}
  If we write the singular value decomposition of $W$ as 
  \[
  W = U\Sigma V^\top  
  \]
  then $U\sim\mathrm{Haar}(\mathrm{O}_d(\R))$, $V\sim \mathrm{Haar}(\mathrm{O}_p(\R))$, and $U$ and $V$ are independent of $\Sigma$. Beside, if we consider the empirical distribution of the squared singular values of $W$ they converge vaguely almost surely to the Marchenko--Pastur distribution of shape $\gamma_2$.  
\end{thm} 
A notable feature of this tensor is that his spectrum is exactly computable in terms of the squared singular values of $W$. 
\begin{lem}\label{lem:spectrumQhat}
  The spectrum of $\hat{\Q}$ is given by
  \[
  \mathrm{Spec}(\hat{\Q})
  =
  \{\beta_\psi^2\}\cup \left\{
    \frac{\alpha_\phi^2}{d}\left(
      \sigma_i^2+\sigma_j^2
    \right)
    +
    \beta_\psi^2
  \right\}_{i \leqslant j \leqslant \min\{d,p\}}
  \cup
  \left\{
    \frac{\alpha_\phi^2}{d}
    \sigma_j^2
    +\beta_\psi^2
  \right\}
  _{j \leqslant p < i }
  \]
  Besides, an eigenvector of $\frac{\alpha_\phi^2}{d}(\sigma_i^2+\sigma_j^2)+\beta_\psi^2$ is given by
  \[
    \frac{1}{\sqrt{(1+\delta_{ij})(\sigma_i^2+\sigma_j^2)}}\left(\sigma_j\u^i\otimes \v^j + \sigma_i\u^j\otimes \v^i  \right)
  \]
  and the eigenspace of $\beta_\psi^2$ is given by
  \[
  \mathrm{Ker} \left(
    \hat{\Q}-\beta_\psi^2 \mathcal{I}_{d,p}
  \right)
  =
  \mathrm{Span} \left(
    \left\{
    \frac{1}{\sqrt{\sigma_i^2+\sigma_j^2}}\left(
      \sigma_i\u^i\otimes \v^j - \sigma_j\u^j\otimes \v^i
    \right)
    \right\}_{i < j\leqslant \min\{d,p\}}
    \cup
    \left\{
      \u^i\otimes \v^j
    \right\}_{ i \leqslant d < j}
  \right)
  \]  
  where $\u^i$ and $\v^j$ are singular vectors of $W$ as in Theorem \ref{thm:svd}.
\end{lem}
\begin{proof}
  We start by computing 
  \[
  \left(
    \tau_{23}\left(\Id_d\otimes W^\top W\right)\u^i\otimes \v^j
  \right)_{k\ell}
  =
  \sum_{q=1}^d\sum_{r=1}^p \tau_{23}(\Id_d\otimes W^\top W)_{k\ell qr}\u^i_{q}\v^j_r
  = 
  \sum_{r=1}^p (W^\top W)_{\ell q}\u^i_k \v^j_q
  =
  \sigma_j^2(\u^i\otimes \v^j)_{k\ell}.
  \]
  For the second term, we have 
  \[
  \left(
    \tau_{24}\left(W\otimes W\right)\u^i\otimes \v^j
  \right)_{k\ell}
  =
  \sum_{q=1}^d \sum_{r=1}^p
  W_{kr}W_{q\ell}\u^i_q \v^j_r
  =
  \sigma_i\sigma_j (\u^j\otimes \v^i)_{k\ell}.
  \]
  We thus see that on $\mathrm{Span}\left( \u^i\otimes \v^j,\u^j\otimes \v^i\right)$, for $i\neq j$ and $i,j \leq \min\{p,d\}$, $\hat{\Q}$ is stable and the matrix in this basis can be written as 
  \[
  \frac{\alpha_\phi^2}{d}\begin{bmatrix}
    \sigma_j^2 & \sigma_i\sigma_j \\
    \sigma_i\sigma_j & \sigma_i^2
  \end{bmatrix}  
  +
  \beta_\psi^2 \Id_2
  \]
  which has for eigenvalues $\frac{\alpha_\psi^2}{d}(\sigma_i^2+\sigma_j^2)+\beta_\psi^2$ and $\beta_\psi^2$ and we have 
  \[
  \begin{bmatrix}
    \sigma_j^2 & \sigma_i\sigma_j\\
    \sigma_i\sigma_j & \sigma_i^2
  \end{bmatrix}
  \begin{bmatrix}
    \sigma_j \\
    \sigma_i
  \end{bmatrix}
  =
  (\sigma_i^2+\sigma_j^2)\begin{bmatrix}
    \sigma_j\\
    \sigma_i
  \end{bmatrix},
  \quad 
  \begin{bmatrix}
    \sigma_j^2 & \sigma_i\sigma_j\\
    \sigma_i\sigma_j & \sigma_i^2
  \end{bmatrix} 
  \begin{bmatrix}
    \sigma_i \\
    -\sigma_j
  \end{bmatrix}
  =
  0.
  \]
  As for other pairs, for $i \leq \min\{p,d\}$,
  $\u^i\otimes \v^i$ is an eigenvector of $\hat{\Q}-\beta_\psi^2 \mathcal{I}_{d,p}$ of eigenvalue $2\sigma_i^2$.  We also have that $\u^i\otimes \v^j \in  \mathrm{Ker} \left(
    \hat{\Q}-\beta_\psi^2 \mathcal{I}_{d,p}
  \right)$ whenever $j > \min\{d,p\}$.
  
\end{proof} 
We are now ready to give the asymptotic e.e.d of the tensor $\hat{\Q}$.
\begin{lem}
  The asymptotic eigenvalue distribution of $\hat{\Q}$ is given by $\mathrm{Law}(\alpha_\phi^2 \chi +\beta_\psi^2)$ where 
  \[
  \mathrm{Law}(\chi) 
  = \frac{\gamma_2}{2}\mu_{\mathrm{MP}}^{\gamma_2}\ast \mu_{\mathrm{MP}}^{\gamma_2}
  +
  (1-\gamma_2)\mu_{\mathrm{MP}}^{\gamma_2}
  +
  \frac{\gamma_2}{2}\delta_0.
  \]
\end{lem}
\begin{proof}
  By Lemma \ref{lem:spectrumQhat}, if we remove the scaling and the shift we see that the eigenvalues are given in terms of the reduced singular values (nonzero) of $W$. 
  
  First, suppose that $d>p$, in this case we have $\frac{p(p+1)}{2}$ eigenvalues of the form $\sigma_i^2+\sigma_j^2$. Combined with Theorem \ref{thm:svd}, we obtain a proportion of $\frac{p(p+1)}{2pd}\to \frac{\gamma_2}{2}$ converging toward $\mu_{\mathrm{MP}}^{\gamma_2}\ast \mu_{\mathrm{MP}}^{\gamma_2}$ and we have $p(d-p)$ eigenvalues of the form $\sigma_i^2$ for a proportion of $\frac{p(d-p)}{dp}\to 1-\gamma_2$ converging toward $\mu_{\mathrm{MP}}^{\gamma_2}$. This also gives a multiplicity of $\frac{p(p-1)}{2}$ for the eigenvalue 0 for a proportion of $\frac{p(p-1)}{2pd}\to\frac{\gamma_2}{2}$. Finally, we do obtain a scaling and shift of the measure 
  \[
   \frac{\gamma_2}{2}\mu_{\mathrm{MP}}^{\gamma_2}\ast \mu_{\mathrm{MP}}^{\gamma_2}
  +
  (1-\gamma_2)\mu_{\mathrm{MP}}^{\gamma_2}
  +
  \frac{\gamma_2}{2}\delta_0.
  \]
  
  Now, suppose that $d\leqslant p$, then the Marchenko--pastur distribution $\mu_{\mathrm{MP}}^{\gamma_2}$ has a point mass at 0 of size $1-\frac{1}{\gamma_2}$ and $\mu_{\mathrm{MP}}^{\gamma_2}-\left(1-\frac{1}{\gamma_2}\right)\delta_0$ is a sub-probability measure of total mass $\frac{1}{\gamma_2}$ and thus we define the reduced Marchenko--Pastur distribution, now a probability measure, 
  \begin{equation}\label{eq:defmubar}
  \mathrm{Law(\chi)}=\bar{\mu}_{\mathrm{MP}}^{\gamma_2} 
  \coloneqq
  \gamma_2\left( 
    \mu_{\mathrm{MP}}^{\gamma_2}-
    \left(
      1-\frac{1}{\gamma_2}
    \right)
      \delta_0
  \right).
  \end{equation}
  Similarly as the case $d>p$, we have $\frac{d(d+1)}{2}$ eigenvalues of the form $\sigma_i^2+\sigma_j^2$ for a proportion of $\frac{d(d+1)}{2pd}\to \frac{1}{2\gamma_2}$ converging towards $\bar{\mu}_{\mathrm{MP}}^{\gamma_2}\ast \bar{\mu}_{\mathrm{MP}}^{\gamma_2}$ and we have a multiplicity of $\frac{d(2p-(d+1))}{2}$ for the eigenvalue 0 for a proportion going towards $1-\frac{1}{2\gamma_2}$ and we thus gets a scaling and shift of the measure 
  \[
  \frac{1}{2\gamma_2}\bar{\mu}_{\mathrm{MP}}^{\gamma_2}\ast \bar{\mu}_{\mathrm{MP}}^{\gamma_2}
  +
  \left(
    1-\frac{1}{2\gamma_2}
  \right)\delta_0.
  \] 
  We can now use the definition of $\bar{\mu}_{\mathrm{MP}}^{\gamma_2}$ and unfold to write
  \begin{align*}
  \mathrm{Law}(\chi) &= \frac{\gamma_2}{2}\left(
    \mu_{\mathrm{MP}}^{\gamma_2}\ast \mu_{\mathrm{MP}}^{\gamma_2}
    -2\left(1-\frac{1}{\gamma_2}\right)\mu_{\mathrm{MP}}^{\gamma_2}
    +\left(1-\frac{1}{\gamma_2}\right)^2\delta_0
  \right)
  +
  \left(1-\frac{1}{2\gamma_2}\right)\delta_0
  \\&= 
  \frac{\gamma_2}{2}
  \mu_{\mathrm{MP}}^{\gamma_2}\ast \mu_{\mathrm{MP}}^{\gamma_2}
  +
  (1-\gamma_2)\mu_{\mathrm{MP}}^{\gamma_2}
  +
  \frac{\gamma_2}{2}\delta_0.
  \end{align*}
\end{proof}
This does not give the asymptotic eigenvalue distribution of $\Q$. However, in the two cases from Corollary \ref{cor:main}, we can obtain the limiting measure.
\begin{lem}
  If $\nu=\delta_1$ or if $\beta_\psi^2=0$ i.e $\varphi(x)=ax+b$ then the asymptotic eigenvalue distribution of $\Q$ is given by 
  \[
  \mathrm{Law}\left( \alpha_\phi^2 \chi + \beta_\psi^2\right)
  \quad\text{where}\quad
  \mathrm{Law}(\chi) = \frac{\gamma_2}{2} (\mu_{\mathrm{MP}}^{\gamma_2} \boxtimes \nu) * (\mu_{\mathrm{MP}}^{\gamma_2} \boxtimes \nu) + (1-\gamma_2) (\mu_{\mathrm{MP}}^{\gamma_2} \boxtimes \nu) + \frac{\gamma_2}{2} \delta_0.
  \]
\end{lem}
\begin{proof}
  We can do the two cases separately. If $\nu=\delta_1$ then $D=\mathrm{Id}_p$ and the factorization from Lemma \ref{lem:factor} is irrelevant and the asymptotic eigenvalue distribution of $\Q$ is the same as of $\hat{\Q}$ thus giving the result. In the case where $\beta_\psi^2=0$, the factorization is not needed in Lemma \ref{lem:spectrumQhat} and we can deal directly with $\Q$ by changing the singular values and vectors of $W$ by the singular values and vectors of $WD$. We thus obtain the result using Theorem \ref{theo:baisilverstein}.
\end{proof}
\subsection{Moment computation}
While the factorization from Lemma \ref{lem:factor} helped us compute the exact spectrum of $\hat{\Q}$, it does not always help us to compute the asymptotic eigenvalue of $\Q$ as there is no good relationship between the spectrum of $\hat{\Q}$ and the one of $\D \hat{\Q}\D$. In order to give information on the asymptotic eigenvalue distribution, we compute its moments. We first define a projection on the first coordinates in our tensor space.
\begin{defn}
Let $\mathbf{e}_1,\dots, \mathbf{e}_d$ be an orthonormal basis of $\Rbb^d$, we define the projection $\Pi^{i}:\Rbb^{d,p}\to\Rbb^{d,p}$, 
\[
\Pi^i (\v\otimes\mathbf{w}) = \langle\v,\ebf_i\rangle \ebf_i\otimes \mathbf{w}.
\] 
\end{defn}
We can then write the projection of the tensor $\widehat{\Q}$.
\begin{lem}\label{lem:ProjectionQhat}
We have for $i,j\in\unn{1}{d}$,
\[
  \left(\Pi^{i}\widehat{\Q}\Pi^{j}\right)(\v\otimes \mathbf{w}) = 
  v_j\ebf_i\otimes \left(
    \beta_\psi^2\delta_{ij}\w+\frac{\alpha_\phi^2}{d}\left(
      \delta_{ij} W^\top W\w+(W\w)_iW^\top\ebf_j
    \right)
  \right)
\]
\end{lem}
\begin{proof}
  We compute, by denoting $v_i=\langle \v,\mathbf{e}_i\rangle$ 
  \[
    (\widehat{\Q}\Pi^j)(\v\otimes \mathbf{w})
    =
    \widehat{\Q}(v_j\ebf_j\otimes \w) 
    =
    \frac{\alpha_\phi^2}{d}
    \left(
      v_j\ebf_j\otimes W^\top W \w + W\w\otimes v_j W^\top \ebf_j 
    \right)
    +
    \beta^2_\psi v_j\ebf_j\otimes \w.
  \]
  Now applying $\Pi^i$ to this, we obtain 
  \[
    (\Pi^i\widehat{\Q}\Pi^j)(\v\otimes \w)
    =
    \frac{\alpha_\phi^2}{d}\left(
      v_j\delta_{ij}\ebf_j\otimes W^\top W \w + v_j(W\w)_i\ebf_i\otimes W^\top \ebf_j
    \right)
    +
    \beta_\psi^2 v_j \delta_{ij}\ebf_j\otimes \w
  \]
  which gives the result.
\end{proof}
We then use the decomposition on the orthonormal basis to compute the tracial moments of $\Q$.
\begin{lem}\label{lem:TrQrepresentation}
  We have
  \[
  \frac{1}{pd}
    \Tr(\Q^k)
  =
  \frac{1}{pd}\sum_{\substack{i_1,\dots,i_{k+1}=1\\i_1=i_{k+1}}}^d \Tr\left(
    \prod_{\ell=1}^k \D \Pi^{i_\ell}\widehat{\Q}\Pi^{i_{\ell+1}}\D
  \right).
  \]
\end{lem}
\begin{proof}
  We start by noting that we can write  
  \[
  \Q = \D\widehat{Q}\D = \sum_{i=1}^d \Pi^i \D\widehat{\Q}\D.
  \]
  Besides, $\D$ and $\Pi^i$ commute since
  \[
  \Pi^i\D (\v\otimes \w) = \Pi^i(\v\otimes D\w) = v_i\ebf_i\otimes D\w = \D (v_i\ebf_i\otimes \w) = \D\Pi^i(\v\otimes \w)
  \]
  and since $\Pi^i$ is a projection, we have $\Pi^i = (\Pi^i)^2$, thus we can finally write 
  \begin{align*}
  \Q^ k = \left(\sum_{i=1}^d \Pi^i\D\widehat{\Q}\D \right)^k
  &=
  \sum_{i_1,\dots,i_k=1}^d
  \Pi^{i_1}\D\widehat{\Q}\D \Pi^{i_2}\D\widehat{\Q}\D \cdots \Pi^{i_k}\D\widehat{\Q}\D\\
  & = 
  \sum_{i_1,\dots,i_k=1}^d 
  (\Pi^{i_1})^2 \D\widehat{\Q}\D(\Pi^{i_2})^2\D\widehat{\Q}\D \cdots (\Pi^{i_k})^2\D\widehat{\Q}D\\
  & =
  \sum_{i_1,\dots,i_k=1}^d \Pi^{i_1}\D\Pi^{i_1}\widehat{\Q}\Pi^{i_2}\D\Pi^{i_2}\widehat{Q}\Pi^{i_3}\D \cdots \D\Pi^{i_k}\widehat{Q}\D.
  \end{align*}
  Finally, taking the trace and using the cyclic property to put the first $\Pi^{i_1}$ at the end and using that $\D$ and $\Pi_{i_1}$ commute we get the final result.  
\end{proof}

The trace $\Tr$ we can represent as two separate traces $\Tr = \Tr_1\otimes \Tr_2$ where $\Tr_1$ is the trace over the first $d$ coordinates and $\Tr_2$ is the trace over the last $p$ coordinates, so that for $w\in \Rbb^p$ we have 
\[
  \Tr_1(\Q)(w) = \sum_{i=1}^d (\Q \ebf_i\otimes \w)_i.
\]
Thus $\Tr_1(\Q)$ is a matrix, and we have the following representation of this matrix:
\begin{lem}\label{lem:TrQrepresentation2}
  For any $i_1,\dots,i_k\in\unn{1}{d}$, and with $i_{k+1}=i_1$, we let $1 \leq \rho_1 < \rho_2 < \cdots < \rho_{\ell+1}=k+1$ and $u_1, \ldots, u_\ell$ be such that 
  \[
  \begin{aligned}
  i_1 = i_2 = &\dots =i_{\rho_1} = u_1 \\
  i_{\rho_1+1} = i_{\rho_1+2} = &\dots = i_{\rho_2} = u_2 \\
  &\vdots \\
  i_{\rho_\ell+1} = i_{\rho_\ell+2} = &\dots = i_{\rho_{\ell+1}} = u_{1},
  \end{aligned}
  \]
  with $u_1 \neq u_2 \neq \dots \neq u_{\ell} \neq u_1$.  Then, for $\ell \geq 2$ we have, with $n_j \coloneqq \rho_j - \rho_{j-1} \geq 1$ for all $j\in\unn{1}{\ell+1}$ (setting $\rho_0=0$, $u_{\ell+1}=u_1$, $u_0=u_{\ell}$),
  \[
  \Tr \left(
    \prod_{j=1}^k \D \Pi^{i_j}\widehat{\Q}\Pi^{i_{j+1}}\D
  \right)
  =
  \left(       \frac{\alpha_\phi^2}{d}
  W D H(u_1)^{n_1+n_{\ell+1}-2} DW^\top \right)_{u_0, u_2}
  \prod_{j=2}^{\ell}
  \left(       \frac{\alpha_\phi^2}{d}
  W D H(u_j)^{n_j-1} DW^\top \right)_{u_{j-1}, u_{j+1}},
  \]
  where we set $H(u)$ to be 
  \[
  H(u) \coloneqq \beta_\psi^2D^2 + \frac{\alpha_\phi^2}{d}\left( DW^\top WD + DW^\top (\ebf_{u} \otimes \ebf_{u}) WD \right).
  \]
\end{lem}
\begin{proof}
  Given the definition of $\rho_1$ and $u_1$ we have using Lemma \ref{lem:ProjectionQhat} that 
  \[
  \begin{aligned}
    &\left( \prod_{\ell=1}^{\rho_1} \D \Pi^{i_{\ell}}\widehat{\Q}\Pi^{i_{\ell+1}}\D \right) ( \ebf_{u_2}\otimes \w) \\
    &=
    \ebf_{u_1} \otimes 
    \left(
      \beta_\psi^2D^2 + \frac{\alpha_\phi^2}{d}\left(
        DW^\top WD  + DW^\top (\ebf_{u_1} \otimes \ebf_{u_1}) WD 
      \right)
    \right)^{\rho_1 - 1} \left(
      \frac{\alpha_\phi^2}{d}
      DW^\top (\ebf_{u_2} \otimes \ebf_{u_1}) WD 
    \right)\w \\
    &=
    \ebf_{u_1} \otimes 
      H(u_1)^{\rho_1 - 1} 
  \left(
      \frac{\alpha_\phi^2}{d}
      DW^\top (\ebf_{u_2} \otimes \ebf_{u_1}) WD 
    \right)\w.
    \end{aligned}
  \]
  Continuing in this way (and using that $u_{\ell+1}=u_1$), we have 
  \[
    \begin{aligned}
    &\left( \prod_{\ell=1}^k \D \Pi^{i_\ell}\widehat{\Q}\Pi^{i_{\ell+1}}\D \right) ( \ebf_{u_{\ell+1}}\otimes \w) \\
    &=
    \ebf_{u_1}
    \otimes
    \left(
    \prod_{j=1}^{\ell} 
    \left(
       H(u_j)^{n_j-1}
      \frac{\alpha_\phi^2}{d}
      DW^\top (\ebf_{u_{j+1}} \otimes \ebf_{u_{j}}) WD 
    \right)
    \right)
    H(u_1)^{n_{\ell+1}-1}
    \w.
   \end{aligned}
  \]
  Thus we conclude using cyclicity of the trace and that $u_{\ell+1}=u_1$ and introducing $u_{0}=u_{\ell}$
  \[
    \begin{aligned}
      \Tr \left( \prod_{\ell=1}^k \D \Pi^{i_\ell}\widehat{\Q}\Pi^{i_{\ell+1}}\D \right)
      &=
      \Tr\left(
        H(u_1)^{n_{\ell+1}-1}
        \prod_{j=1}^{\ell} 
        \left(
       H(u_j)^{n_j-1}
      \frac{\alpha_\phi^2}{d}
      DW^\top (\ebf_{u_{j+1}} \otimes \ebf_{u_{j}}) WD 
    \right)
    \right) \\
    &=  
    \left(       \frac{\alpha_\phi^2}{d}
    W D H(u_1)^{n_1+n_{\ell+1}-2} DW^\top \right)_{u_0, u_2}
    \prod_{j=2}^{\ell}
    \left(       \frac{\alpha_\phi^2}{d}
    W D H(u_j)^{n_j-1} DW^\top \right)_{u_{j-1}, u_{j+1}}.
    \end{aligned}
  \]
\end{proof}

\begin{prop}\label{prop:TrQk}
  Define for $n \geq 1$,
  \[
  M(n) \coloneqq \left(
    \frac{\alpha_\phi^2}{d} W D H^{n-1} DW^\top
  \right),
  \quad \text{where} \quad
  H \coloneqq \beta_\psi^2D^2 + \frac{\alpha_\phi^2}{d}\left( DW^\top WD\right).
  \]
  Then we have for $k \geq 1$
  \[
  \frac{1}{pd}\Tr(\Q^k) 
  =
  \frac{1}{p}\Tr(H^k)+
  \frac{1}{dp} \sum n_1 \Tr[ M(n_1)M(n_2)\cdots M(n_\ell)]\Tr[M(n_1')M(n_2')\cdots M(n_\ell')]
  +o_{\P}(1),
  \]
  where $\ell \in \N\{1,\dots,k\}$ and $n_j \geq 1$ for $j \in \unn{1}{\ell}$ and $n_j' \geq 1$ for $j \in \unn{1}{\ell}$ are such that $n_1+\cdots+n_\ell+ n_1'+\cdots+n_\ell'=k$.
\end{prop}
\begin{proof}
  We start by combining Lemma \ref{lem:TrQrepresentation} and Lemma \ref{lem:TrQrepresentation2} to get 
  \begin{equation}\label{eq:trqk-1}
  \begin{aligned}
    &\frac{1}{pd}\Tr(\Q^k) 
    \\&
    =\frac{1}{p}\Tr(H^k)+
    \frac{1}{pd}
    \hspace{-.6em}\sum_{\substack{\tilde{\ell}, u_1, \dots, u_{\tilde{\ell}} \\ \tilde{n}_1+\cdots+\tilde{n}_{\tilde{\ell}+1}=k+1}}\hspace{-.6em}
    \left(       \frac{\alpha_\phi^2}{d}
    W D H(u_1)^{\tilde{n}_1+\tilde{n}_{\tilde{\ell}+1}-2} DW^\top \right)_{u_0, u_2}
    \prod_{j=2}^{\tilde{\ell}}
    \left(       \frac{\alpha_\phi^2}{d}
    W D H(u_j)^{\tilde{n}_j-1} DW^\top \right)_{u_{j-1}, u_{j+1}}
  \end{aligned}
\end{equation}
  where $\tilde{\ell} \in \{1,\dots,k\}$ and $\tilde{n}_j \geq 1$ for $j \in \unn{1}{\tilde{\ell}}$ are such that $\tilde{n}_1+\cdots+\tilde{n}_{\tilde{\ell}+1}=k+1$, and where $u_i \in \unn{1}{d}$ for $i \in \unn{1}{\tilde{\ell}}$ satisfy $u_{i+1} \neq u_i$ for $i \in \unn{1}{\tilde{\ell}}$.
  We now proceed in steps.
  \paragraph{Step 1: Removing the rank-1 perturbations from $H$.} We start by introducing the following notation
  \[
  Z = \frac{\alpha_\phi}{\sqrt{d}}WD,\quad
    M(u,m) = \frac{\alpha_\phi^2}{d}WDH(u)^{m-1}DW^\top = ZH(u)^{m-1}Z^\top 
  \quad\text{and}\quad
  \Delta(u,m) = M(u,m)-M(m).
  \]
  We expand $\Delta$ and see that only linear terms in $\e_u\otimes \e_u$ are actually important. Indeed, one can write 
  \begin{align*}
  &\Delta(u,m) = 
  \sum_{a_1=0}^{m-1}\sum_{b_1=1}^m\sum_{\substack{a_2,\dots,a_p\geqslant 0\\b_2,\dots,b_p\geqslant 0\\\sum a_i+b_i=m}}
  ZH^{a_1}(Z^\top \e_u\otimes \e_uZ)^{b_1}\dots H^{a_p}(Z\e_u\otimes Z)^{b_p}Z^\top\\
  &=
  \sum_{a_1=0}^{m-1}\sum_{b_1=1}^m\sum_{\substack{a_2,\dots,a_p\geqslant 0\\b_2,\dots,b_p\geqslant 0\\\sum a_i+b_i=m}}
  \left(
    \left(
      M(1)_{uu}
    \right)^{\sum_{i=1}^p(b_i-1)}
    \prod_{i=2}^{p-1}M(a_i)_{uu}
  \right)
  M(a_1)\e_u\otimes \e_u Z
  \left(
    M(a_p)_{uu}\mathrm{Id}_p\mathds{1}_{b_p\neq 0}+H^{a_p}\delta_{b_p,0}
  \right) Z^\top.
  \end{align*}
  Thus we see that each term in the expansion of $\Delta$ is a product of diagonal entries of $M$, which is $O(1)$ bounded with very high probability (see Theorem \ref{a:main}), and a term of the form $M(a)\e_u\otimes \e_uM(b)$, with $b=a_p+1$ or $1$. Remember that $m$ and thus $a_i$'s or $b_i$'s do not depend on the growing dimensions $n$, $d$, or $p$ so that we are considering $\Delta$ as being a single term of the sum without loss of generality i.e 
  \[
  \Delta(u,a,b) = C(u)M(a)\e_u\otimes \e_u M(b)\quad\text{in particular}\quad 
  \Delta(u,a,b)_{ij} = C(u)M(a)_{iu}M(b)_{uj}
  \quad\text{with}\quad 
  C(u)=\O{d^\varepsilon}
  \]  
  with very high probability.
  We recall that the goal in this step is to compare the sum in \eqref{eq:trqk-1} with $H(u)$ replaced with $H$. This term can be written, forgetting the sum over the exponents $\tilde{n}_j$ and over $\tilde{\ell}$, also dimension independent, as 
  \[
    \frac{1}{pd}
    \hspace{-.2em}\sum_{\substack{u_1, \dots, u_{\tilde{\ell}}}=1}^d
    M(u_1,\tilde{n}_1)_{u_0,u_2}
    \prod_{j=2}^{\tilde{\ell}}
    M(u_j,\tilde{n}_j)_{u_{j-1},u_{j+1}}
    \prod_{i=1}^{\tilde{\ell}}(1-\delta_{u_{i},u_{i+1}})
  \]
  By writing $M(u,m) = M(m)+\Delta(u,m)$ and with the consideration above on the form of $\Delta$, we want to bound terms which consist in products of entries of $\Delta(u_i,a_i,b_i)$, $M(n_i)$, and $\delta_{u_i,u_i+1}$ with at least one entry of $\Delta$ in the product. We note that by adding the adjacency condition $u_i\neq u_{i+1}$ in the sum as $1-\delta_{u_i,u_{i+1}}$ the sum is now free from restrictions. We start with an example to show the process of bounding these terms, considering a term of the form, for $\tilde{\ell}=6$, 
  \[
  \frac{1}{pd}\sum_{u_1,\dots u_6=1}^d\Delta(u_1,a_1,b_1)_{u_6u_2}M(\tilde{n}_2)_{u_1u_3}M(\tilde{n}_3)_{u_2u_4}M(\nt_4)_{u_3u_5}M(\nt_5)_{u_4u_6}M(\nt_6)_{u_5u_1}\delta_{u_2u_3}\delta_{u_5u_6}.
  \]
The first step consists in using the specific form of $\Delta$ which we showed above and using the $\delta_{u_iu_j}$ to write the term as 
\[
\frac{1}{pd}\sum_{u_1,u_2,u_4,u_5=1}^dC(u_1)M(a_1)_{u_5u_1}M(b_1)_{u_1u_2}M(\tilde{n}_2)_{u_1u_2}M(\tilde{n}_3)_{u_2u_4}M(\nt_4)_{u_2u_5}M(\nt_5)_{u_4u_5}M(\nt_6)_{u_5u_1}.
\]
We now can sum over the indices that appear only twice, since the sum is free from restrictions, this can be written as a matrix product. We note that $C(u_1)$ depending on $u_1$ is not a problem here since the index appearing inside $\Delta$ never appears only twice and thus can never be contracted as a matrix product. In this case, $u_4$ appears only twice in this term and we can write it as 
\[
\frac{1}{pd}\sum_{u_1,u_2,u_5=1}^dM(a_1)_{u_5u_1}M(b_1)_{u_1u_2}M(\tilde{n}_2)_{u_1u_2}
\left(
  M(\tilde{n}_3)M(\nt_5)
\right)_{u_2u_5}
M(\nt_4)_{u_2u_5}M(\nt_6)_{u_5u_1}.
\]
If every term is off-diagonal i.e.\ there is no equality between the $u_i$'s in the sum, then we can bound with very high probability each entry by $d^{-\frac{1}{2}+\varepsilon}$ (again using Theorem \ref{a:main}) and $C(u)$ by $d^\varepsilon$ so that the contribution from this restricted sum is of order $\O{\frac{d^{7\varepsilon}}{pd}}$. If there is additional identifications between the indices in the sum, we note that each identification transforms two off-diagonal edges into a diagonal edge and thus loses a power of $d$ while gaining a power of $d$ from summing one less index giving the same order of magnitude for the term.

We now give a general graphical representation of this procedure as in Figure \ref{fig:procedure}. We consider the vertices $u_1$,\dots,$u_{\tilde{\ell}}$ and we draw an edge for each entry $u_{i-1,i+1}$. We color each edge corresponding to $\Delta$ terms. On top of this we show identifications between adjacent vertices as arrows between them.

\begin{figure}[!ht]
  \centering
  \[
  \vcenter{\hbox{
    \begin{tikzpicture}
	  \node[fill=Black, label=$u_{1}$, inner sep = 0pt, minimum size=.2cm] (1) at (360/6+30: 1cm) {};
	  \node[fill=Black, label=135:$u_{6}$, inner sep = 0pt, minimum size=.2cm] (2) at (2*360/6+30: 1cm) {};
	  \node[fill=Black, label=225:$u_5$, inner sep = 0pt, minimum size=.2cm] (3) at (3*360/6+30: 1cm)  {};
	  \node[fill=Black, label=below:$u_{4}$, inner sep = 0pt, minimum size=.2cm] (4) at (4*360/6+30: 1cm) {};
	  \node[fill=Black, label=315:$u_{3}$, inner sep = 0pt, minimum size=.2cm] (5) at (5*360/6+30: 1cm) {};
	  \node[fill=Black, label=45:$u_{2}$, inner sep = 0pt, minimum size=.2cm] (6) at (6*360/6+30: 1cm) {};
	
	\draw[-,RoyalBlue] (1) edge[bend left=60, line width=.12em] (3);
	\draw[-,RoyalBlue] (2) edge[bend left=60, line width=.12em] (4);
	\draw[-,RoyalBlue] (3) edge[bend left=60, line width=.12em] (5);
	\draw[-,RoyalBlue] (4) edge[bend left=60, line width=.12em] (6);
	\draw[-,RoyalBlue] (5) edge[bend left=60, line width=.12em] (1);
	\draw[-,BrickRed] (6) edge[bend left=60, line width=.12em] (2);

  \draw[<->,ForestGreen] (2) edge[bend right, line width=.1em] (3);
  \draw[<->,ForestGreen] (5) edge[bend right, line width=.1em] (6);

	\end{tikzpicture}	
  }}
  \longrightarrow
  \vcenter{\hbox{
\begin{tikzpicture}
	  \node[fill=Black, label=$u_{1}$, inner sep = 0pt, minimum size=.2cm] (1) at (360/4: 1cm) {};
	  \node[fill=Black, label=135:$u_{6}{=}u_{5}$, inner sep = 0pt, minimum size=.2cm] (2) at (2*360/4: 1cm) {};
	  \node[fill=Black, label=270:$u_{4}$, inner sep = 0pt, minimum size=.2cm] (3) at (3*360/4: 1cm)  {};
	  \node[fill=Black, label=-45:$u_{3}{=}u_{2}$, inner sep = 0pt, minimum size=.2cm] (4) at (4*360/4: 1cm) {};

	\draw[-,RoyalBlue] (1) edge[bend left=40, line width=.12em] (2);
  \draw[-,RoyalBlue] (1) edge[bend right=40, line width=.12em] (2);
	\draw[-,RoyalBlue] (1) edge[bend left=40, line width=.12em] (4);
  \draw[-,RoyalBlue] (1) edge[bend right=40, line width=.12em] (4);
  \draw[-,RoyalBlue] (2) edge[bend right=40, line width=.12em] (4);
  \draw[-,RoyalBlue] (3) edge[bend left=40, line width=.12em] (2);
	\draw[-,RoyalBlue] (3) edge[bend right=40, line width=.12em] (4);

    \node[draw=ForestGreen, dotted, circle, fit=(3), inner sep=2pt, line width=.15 em] {};
	\end{tikzpicture}	
  }}
  \longrightarrow 
  \vcenter{\hbox{
  \begin{tikzpicture}
	  \node[fill=Black, label=225:$u_{5}$, inner sep = 0pt, minimum size=.2cm] (1) at (360/3+90: 1cm) {};
	  \node[fill=Black, label=-45:$u_{2}$, inner sep = 0pt, minimum size=.2cm] (2) at (2*360/3+90: 1cm) {};
	  \node[fill=Black, label=90:$u_{1}$, inner sep = 0pt, minimum size=.2cm] (4) at (3*360/3+90: 1cm) {};

	\draw[-,RoyalBlue] (1) edge[bend left=20, line width=.12em] (2);
  \draw[-,RoyalBlue] (1) edge[bend right=20, line width=.12em] (2);
	\draw[-,RoyalBlue] (1) edge[bend left=20, line width=.12em] (4);
  \draw[-,RoyalBlue] (1) edge[bend right=20, line width=.12em] (4);
  \draw[-,RoyalBlue] (2) edge[bend right=20, line width=.12em] (4);
  \draw[-,RoyalBlue] (2) edge[bend left=20, line width=.12em] (4);
  
	\end{tikzpicture}	
  }}
  \]
  \caption{Illustration of the procedure to bound the example term above, we start by drawing the graph and color, in red here, the edge corresponding to $\Delta$ terms and show identifications between adjacent terms. From this, we use the form of the $\Delta$ terms which replaces the red edge to two blue edges between adjacent indices, here $\{u_1,u_5\}$ and $\{u_1,u_2\}$. In the second step, we then highlight vertices that have degree 2 in this graph, $u_4$ here and we remove it and replace the two edges $\{u_4,u_5\}$ and $\{u_2,u_4\}$ for one edge $\{u_2,u_5\}$. This gives the final graph where we have 6 edges corresponding of entries (typically off-diagonal) of products of $M(m)$.}
  \label{fig:procedure}
\end{figure}
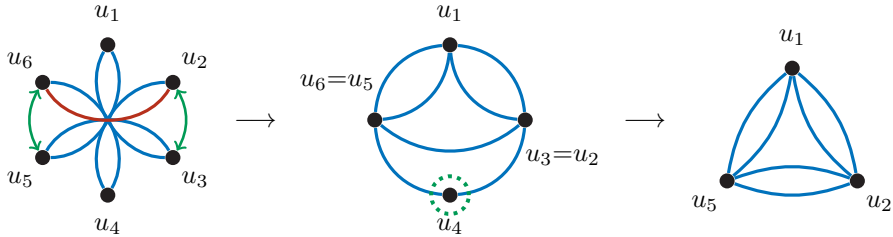
If we now consider a general graph with $\tilde{\ell}$ initial vertices $u_1,\dots,u_{\tilde{\ell}}$ and with corresponding edges $\{u_{i-1},u_{i+1}\}$. We color an arbitrary number of them (at least 1) to show $\Delta$ edges and we add identifications between an arbitrary number of adjacent vertices (possibly none). The first step consists in replacing the $\Delta$ edges by the corresponding $\{u_{i-1},u_i\}$ and $\{u_i,u_{i+1}\}$ edges and collapsing the vertices with identification together. From this, we note that if $\{u_{i-1},u_{i+1}\}$ is a colored edge, then after this procedure $u_i$ is a vertex with degree 4 since the colored edge is replaced by two edges $\{u_i,u_{i+1}\}$ and $\{u_{i-1},u_i\}$ and we also have the two pre-existing edges $\{u_{i-2},u_i\}$ and $\{u_i,u_{i+2}\}$. We note that if one of these pre-existing edge is itself colored, for instance $\{u_i,u_{i+2}\}$, the degree after the procedure of $u_i$ is stil 4 since this edge is replaced by $\{u_i,u_{i+1}\}$ and $\{u_{i+1},u_{i+2}\}$ and thus only one is incident to $u_i$. This is illustrated  in Figure \ref{fig:procedure-1}.

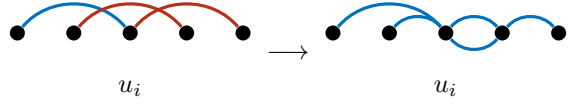
\begin{figure}[!ht]
  \[
  \vcenter{\hbox{
    \begin{tikzpicture}
      \node[fill=black, inner sep=0pt, minimum size=.2cm] (1) at (0,0) {};
      \node[fill=black, inner sep=0pt, minimum size=.2cm] (2) at (0.75,0) {};
      \node[fill=black, label = {[label distance =.95em]below:$u_{i}$}, inner sep=0pt, minimum size=.2cm] (3) at (1.5,0) {};
      \node[fill=black, inner sep=0pt, minimum size=.2cm] (4) at (2.25,0) {};
      \node[fill=black, inner sep=0pt, minimum size=.2cm] (5) at (3,0) {};

      \draw[-,RoyalBlue] (1) edge[bend left=50, line width=.12em] (3);
      \draw[-,BrickRed] (2) edge[bend left=50, line width=.12em] (4);
      \draw[-,BrickRed] (3) edge[bend left=50, line width=.12em] (5);
    \end{tikzpicture}
  }}
  \longrightarrow
  \vcenter{\hbox{
    \begin{tikzpicture}
      \node[fill=black, inner sep=0pt, minimum size=.2cm] (1) at (0,0) {};
      \node[fill=black, inner sep=0pt, minimum size=.2cm] (2) at (0.75,0) {};
      \node[fill=black, label = {[label distance =.95em]below:$u_{i}$}, inner sep=0pt, minimum size=.2cm] (3) at (1.5,0) {};
      \node[fill=black, inner sep=0pt, minimum size=.2cm] (4) at (2.25,0) {};
      \node[fill=black, inner sep=0pt, minimum size=.2cm] (5) at (3,0) {};

      \draw[-,RoyalBlue] (1) edge[bend left=50, line width=.12em] (3);
      \draw[-,RoyalBlue] (2) edge[bend left=50, line width=.12em] (3);
      \draw[-,RoyalBlue] (3) edge[bend left=50, line width=.12em] (4);
      \draw[-,RoyalBlue] (4) edge[bend left=50, line width=.12em] (3);
      \draw[-,RoyalBlue] (4) edge[bend left=50, line width=.12em] (5);

    \end{tikzpicture}
  }}
  \]
  \caption{After replacing the edge $\{u_{i-1},u_{i+1}\}$ to $\{u_{i-1},u_i\}$ and $\{u_{i},u_{i+1}\}$ doing the same for the red edge $\{u_{i},u_{i+1}\}$, we see that $u_i$ is of degree 4.}
  \label{fig:procedure-1}
\end{figure}

From collapsing the vertices with identification say $u_i=u_{i+1}=\dots=u_{i+r}$ for some $r\geqslant 1$, one should note two things: some self-loops can be created from identifying a group of adjacent vertices i.e.\ $r-1$ self-loops are created, and removing these self-loops, the degree of the collapsed vertex is itself 4 since one has two edges coming from $u_{i-1}$ and $u_{i-2}$ and towards $u_{i+r+1}$ and $u_{i+r+2}$. This is illustrated in Figure \ref{fig:procedure-2}.
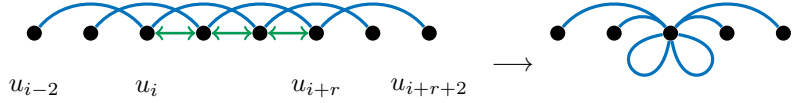
\begin{figure}[!ht]
  \[
  \vcenter{\hbox{
    \begin{tikzpicture}
      \node[fill=black, label = {[label distance = .5em]below:$u_{i-2}$}, inner sep=0pt, minimum size=.2cm] (1) at (0,0) {};
      \node[fill=black, inner sep=0pt, minimum size=.2cm] (2) at (0.75,0) {};
      \node[fill=black, label = {[label distance = .95em]below:$u_{i}$}, inner sep=0pt, minimum size=.2cm] (3) at (1.5,0) {};
      \node[fill=black, inner sep=0pt, minimum size=.2cm] (4) at (2.25,0) {};
      \node[fill=black, inner sep=0pt, minimum size=.2cm] (5) at (3,0) {};
      \node[fill=black, label = {[label distance = .5em]below:$u_{i+r}$}, inner sep=0pt, minimum size=.2cm] (6) at (3.75,0) {};
      \node[fill=black, inner sep=0pt, minimum size=.2cm] (7) at (4.5,0) {};
      \node[fill=black, label = {[label distance = 0em]below:$u_{i+r+2}$}, inner sep=0pt, minimum size=.2cm] (8) at (5.25,0) {};

      \draw[-,RoyalBlue] (1) edge[bend left=50, line width=.12em] (3);
      \draw[-,RoyalBlue] (2) edge[bend left=50, line width=.12em] (4);
      \draw[-,RoyalBlue] (3) edge[bend left=50, line width=.12em] (5);
      \draw[-,RoyalBlue] (4) edge[bend left=50, line width=.12em] (6);
      \draw[-,RoyalBlue] (5) edge[bend left=50, line width=.12em] (7);
      \draw[-,RoyalBlue] (6) edge[bend left=50, line width=.12em] (8);

      \draw[<->,ForestGreen] (4) edge[line width=.1em] (3);
      \draw[<->,ForestGreen] (4) edge[line width=.1em] (5);
      \draw[<->,ForestGreen] (5) edge[line width=.1em] (6);

    \end{tikzpicture}
  }}
  \longrightarrow
  \vcenter{\hbox{
    \begin{tikzpicture}
      \node[fill=black, inner sep=0pt, minimum size=.2cm] (1) at (0,0) {};
      \node[fill=black, inner sep=0pt, minimum size=.2cm] (2) at (0.75,0) {};
      \node[fill=black, inner sep=0pt, minimum size=.2cm] (3) at (1.5,0) {};
      \node[fill=black, inner sep=0pt, minimum size=.2cm] (4) at (2.25,0) {};
      \node[fill=black, label = {[label distance = 0em]below:$\phantom{u_{i+r+2}}$}, inner sep=0pt, minimum size=.2cm] (5) at (3,0) {};

      \draw[-,RoyalBlue] (1) edge[bend left=50, line width=.12em] (3);
      \draw[-,RoyalBlue] (2) edge[bend left=50, line width=.12em] (3);
      \draw[-,RoyalBlue] (4) edge[bend right=50, line width=.12em] (3);
      \draw[-,RoyalBlue] (5) edge[bend right=50, line width=.12em] (3);

      \path [draw, RoyalBlue, line width = .12em] (3) to [loop below, looseness=10, min distance=1cm, out=-10, in=-80] (3);
      \path [draw, RoyalBlue, line width=.12em] (3) to [loop below, looseness=10, min distance=1cm, out=-100, in=-170] (3);

    \end{tikzpicture}
  }}
  \]
  \caption{If a group of adjacent vertices are identified, the corresponding vertex has some self-loops and is of degree 4 if not counting self-loops. Note that we only drew relevant edges.}
  \label{fig:procedure-2}
\end{figure}

Since we finally remove the vertices of degree 2 and replaces their two edges $\{u_{i+2},u_i\}$ and $\{u_{i-2},u_i\}$ to $\{u_{i+2},u_{i-2}\}$ we see that the overall degree of the remaining vertices is unchanged and thus always of degree 4, up to a given number of self-loops. Thus the final graph after this procedure, is always of the same form: $\hat{\ell}\leqslant \tilde{\ell}$ vertices $u_{i_1},\dots,u_{i_{\hat{\ell}}}$ where we keep the initial ordering of the vertices such as each vertex $u_{i_j}$ has a double edge to $u_{i_{j-1}}$ and to $u_{i_{j+1}}$ with $i_{\hat{\ell}+1}=i_1$ and each vertex $u_{i_j}$ can have an arbitrary number of self-loops attached to it. We illustrate a more general procedure in Figure \ref{fig:procedure-3}.

\begin{figure}[!ht]
  \[
  \vcenter{\hbox{
    \begin{tikzpicture}
      \node[fill=Black, inner sep = 0pt, minimum size=.2cm] (2) at (360/10+18: 1cm) {};
	    \node[fill=Black, inner sep = 0pt, minimum size=.2cm] (1) at (2*360/10+18: 1cm) {};
      \node[fill=Black, inner sep = 0pt, minimum size=.2cm] (10) at (3*360/10+18: 1cm) {};
      \node[fill=Black, inner sep = 0pt, minimum size=.2cm] (9) at (4*360/10+18: 1cm) {};
      \node[fill=Black, inner sep = 0pt, minimum size=.2cm] (8) at (5*360/10+18: 1cm) {};
      \node[fill=Black, inner sep = 0pt, minimum size=.2cm] (7) at (6*360/10+18: 1cm) {};
      \node[fill=Black, inner sep = 0pt, minimum size=.2cm] (6) at (7*360/10+18: 1cm) {};
      \node[fill=Black, inner sep = 0pt, minimum size=.2cm] (5) at (8*360/10+18: 1cm) {};
      \node[fill=Black, inner sep = 0pt, minimum size=.2cm] (4) at (9*360/10+18: 1cm) {};
      \node[fill=Black, inner sep = 0pt, minimum size=.2cm] (3) at (10*360/10+18: 1cm) {};

      \draw[-,RoyalBlue] (1) edge[bend right=50, line width=.12em] (3);
      \draw[-,RoyalBlue] (2) edge[bend right=50, line width=.12em] (4);
      \draw[-,RoyalBlue] (4) edge[bend right=50, line width=.12em] (6);
      \draw[-,RoyalBlue] (5) edge[bend right=50, line width=.12em] (7);
      \draw[-,RoyalBlue] (6) edge[bend right=50, line width=.12em] (8);
      \draw[-,RoyalBlue] (7) edge[bend right=50, line width=.12em] (9);
      \draw[-,RoyalBlue] (8) edge[bend right=50, line width=.12em] (10);

      \draw[-,BrickRed] (10) edge[bend right=50, line width=.12em] (2);
      \draw[-,BrickRed] (9) edge[bend right=50, line width=.12em] (1);
      \draw[-,BrickRed] (3) edge[bend right=50, line width=.12em] (5);

      \draw[<->,ForestGreen] (10) edge[bend left, line width=.1em] (1);
      \draw[<->,ForestGreen] (2) edge[bend left, line width=.1em] (3);
      \draw[<->,ForestGreen] (3) edge[bend left, line width=.1em] (4);
      \draw[<->,ForestGreen] (4) edge[bend left, line width=.1em] (5);
      \draw[<->,ForestGreen] (6) edge[bend left, line width=.1em] (7);

    \end{tikzpicture}
  }}
  \longrightarrow
  \vcenter{\hbox{
    \begin{tikzpicture}
      \node[fill=Black, inner sep = 0pt, minimum size=.2cm] (2) at (360/10+18: 1cm) {};
	    \node[fill=Black, inner sep = 0pt, minimum size=.2cm] (1) at (2*360/10+18: 1cm) {};
      \node[fill=Black, inner sep = 0pt, minimum size=.2cm] (10) at (3*360/10+18: 1cm) {};
      \node[fill=Black, inner sep = 0pt, minimum size=.2cm] (9) at (4*360/10+18: 1cm) {};
      \node[fill=Black, inner sep = 0pt, minimum size=.2cm] (8) at (5*360/10+18: 1cm) {};
      \node[fill=Black, inner sep = 0pt, minimum size=.2cm] (7) at (6*360/10+18: 1cm) {};
      \node[fill=Black, inner sep = 0pt, minimum size=.2cm] (6) at (7*360/10+18: 1cm) {};
      \node[fill=Black, inner sep = 0pt, minimum size=.2cm] (5) at (8*360/10+18: 1cm) {};
      \node[fill=Black, inner sep = 0pt, minimum size=.2cm] (4) at (9*360/10+18: 1cm) {};
      \node[fill=Black, inner sep = 0pt, minimum size=.2cm] (3) at (10*360/10+18: 1cm) {};

      \draw[-,RoyalBlue] (1) edge[bend right=50, line width=.12em] (3);
      \draw[-,RoyalBlue] (2) edge[bend right=50, line width=.12em] (4);
      \draw[-,RoyalBlue] (4) edge[bend right=50, line width=.12em] (6);
      \draw[-,RoyalBlue] (5) edge[bend right=50, line width=.12em] (7);
      \draw[-,RoyalBlue] (6) edge[bend right=50, line width=.12em] (8);
      \draw[-,RoyalBlue] (7) edge[bend right=50, line width=.12em] (9);
      \draw[-,RoyalBlue] (8) edge[bend right=50, line width=.12em] (10);

      \draw[-,RoyalBlue] (1) edge[bend right=50, line width=.12em] (2);
      \draw[-,RoyalBlue] (3) edge[bend right=50, line width=.12em] (4);
      \draw[-,RoyalBlue] (4) edge[bend right=50, line width=.12em] (5);
      \draw[-,RoyalBlue] (9) edge[bend right=50, line width=.12em] (10);
      \draw[-,RoyalBlue] (10) edge[bend right=50, line width=.12em] (1);
      \draw[-,RoyalBlue] (1) edge[bend right=50, line width=.12em] (10);

      \draw[<->,ForestGreen] (10) edge[line width=.1em] (1);
      \draw[<->,ForestGreen] (2) edge[bend left, line width=.1em] (3);
      \draw[<->,ForestGreen] (3) edge[bend left, line width=.1em] (4);
      \draw[<->,ForestGreen] (4) edge[bend left, line width=.1em] (5);
      \draw[<->,ForestGreen] (6) edge[bend left, line width=.1em] (7);

    \end{tikzpicture}
  }}
  \longrightarrow 
  \vcenter{\hbox{
    \begin{tikzpicture}
      \node[fill=Black, inner sep = 0pt, minimum size=.2cm] (1) at (360/5+18: 1cm) {};
      \node[fill=Black, inner sep = 0pt, minimum size=.2cm] (9) at (2*360/5+18: 1cm) {};
      \node[fill=Black, inner sep = 0pt, minimum size=.2cm] (8) at (3*360/5+18: 1cm) {};
      \node[fill=Black, inner sep = 0pt, minimum size=.2cm] (6) at (4*360/5+18: 1cm) {};
      \node[fill=Black, inner sep = 0pt, minimum size=.2cm] (2) at (5*360/5+18: 1cm) {};

      \draw[-,RoyalBlue] (1) edge[bend right, line width=.12em] (2);
      \draw[-,RoyalBlue] (2) edge[bend right, line width=.12em] (1);
      \draw[-,RoyalBlue] (2) edge[bend right, line width=.12em] (6);
      \draw[-,RoyalBlue] (6) edge[bend right, line width=.12em] (2);
      \draw[-,RoyalBlue] (6) edge[bend right, line width=.12em] (8);
      \draw[-,RoyalBlue] (6) edge[bend right, line width=.12em] (9);
      \draw[-,RoyalBlue] (8) edge[bend right, line width=.12em] (1);
      \draw[-,RoyalBlue] (9) edge[bend right, line width=.12em] (1);

      \path [draw, RoyalBlue, line width = .12em] (1) to [loop above, min distance=.75cm, out=125, in=55] (1);
      \path [draw, RoyalBlue, line width = .12em] (2) to [loop below, looseness=10, min distance=.75cm, out=-23, in=-80] (2);
      \path [draw, RoyalBlue, line width = .12em] (2) to [loop below, looseness=10, min distance=.75cm, out=34, in=-23] (2);
      \path [draw, RoyalBlue, line width = .12em] (2) to [loop below, looseness=10, min distance=.75cm, out=91, in=34] (2);

      \node[draw=ForestGreen, dotted, circle, fit=(8), inner sep=2pt, line width=.15 em] {};
      \node[draw=ForestGreen, dotted, circle, fit=(9), inner sep=2pt, line width=.15 em] {};

    \end{tikzpicture}
  }}
  \longrightarrow 
  \vcenter{\hbox{
    \begin{tikzpicture}
      \node[fill=Black, inner sep = 0pt, minimum size=.2cm] (1) at (360/3-30: 1cm) {};
      \node[fill=Black, inner sep = 0pt, minimum size=.2cm] (6) at (2*360/3-30: 1cm) {};
      \node[fill=Black, inner sep = 0pt, minimum size=.2cm] (2) at (3*360/3-30: 1cm) {};

      \draw[-,RoyalBlue] (1) edge[bend right, line width=.12em] (2);
      \draw[-,RoyalBlue] (2) edge[bend right, line width=.12em] (1);
      \draw[-,RoyalBlue] (1) edge[bend right, line width=.12em] (6);
      \draw[-,RoyalBlue] (6) edge[bend right, line width=.12em] (1);
      \draw[-,RoyalBlue] (2) edge[bend right, line width=.12em] (6);
      \draw[-,RoyalBlue] (6) edge[bend right, line width=.12em] (2);

      \path [draw, RoyalBlue, line width = .12em] (1) to [loop above, min distance=.75cm, out=125, in=55] (1);
      \path [draw, RoyalBlue, line width = .12em] (2) to [loop below, looseness=10, min distance=.75cm, out=-68, in=-125] (2);
      \path [draw, RoyalBlue, line width = .12em] (2) to [loop below, looseness=10, min distance=.75cm, out=-11, in=-68] (2);
      \path [draw, RoyalBlue, line width = .12em] (2) to [loop below, looseness=10, min distance=.75cm, out=46, in=-11] (2);
    \end{tikzpicture}
  }}
  \]
  \caption{Example of a general procedure. In the first step, we replace the red edges by the two corresponding blue edges from the form of $\Delta$. In the second step, we identify the vertices according to the green arrows. In the third step, we highlight the vertices with degree 2 which we remove. The final graph is always of this form, vertices drawn around a circle and where adjacent vertices are linked with double edges and each vertex can have an arbitrary number of loops coming from adjacent identifications or identification between two adjacent red edges.}
  \label{fig:procedure-3}
\end{figure}
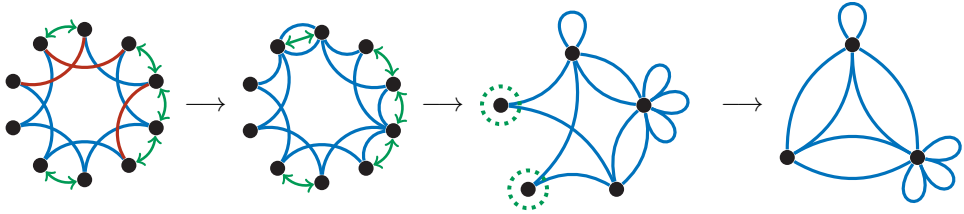

We now have two different bounds coming from these graphs. First if there $\hat{\ell}\geqslant 2$ vertices left after the procedure then we have $2\hat{\ell}$ edges between these vertices and some self-loops. We can now bound each edge between distinct vertices by $d^{-\frac{1}{2}+\varepsilon}$ and each self-loop by $d^{\varepsilon}$, with the $\frac{1}{pd}$ normalization in front of the sum we thus obtain an error bound of $\O{\frac{d^{\tilde{\ell} \varepsilon}}{pd}}$ where we used the fact that $\hat{\ell}\leqslant \tilde{\ell}.$ We also have to take into account the case where $\hat{\ell}=1$. In this case, we do not have any edges between vertices (since there is only one vertex) but only self-loops, the error bound is then $\O{\frac{d^{\tilde{\ell}\varepsilon}}{p}}$ since the sum over the value of the remaining index adds a factor of $d$ which is not canceled by off-diagonal entries. We illustrate a procedure which gives a single vertex in Figure \ref{fig:procedure-4}.

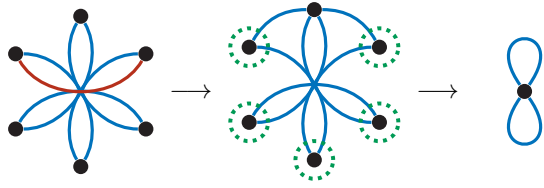
\begin{figure}[!ht]
  \centering
  \[
  \vcenter{\hbox{
    \begin{tikzpicture}
	  \node[fill=Black, inner sep = 0pt, minimum size=.2cm] (1) at (360/6+30: 1cm) {};
	  \node[fill=Black, inner sep = 0pt, minimum size=.2cm] (2) at (2*360/6+30: 1cm) {};
	  \node[fill=Black, inner sep = 0pt, minimum size=.2cm] (3) at (3*360/6+30: 1cm)  {};
	  \node[fill=Black, inner sep = 0pt, minimum size=.2cm] (4) at (4*360/6+30: 1cm) {};
	  \node[fill=Black, inner sep = 0pt, minimum size=.2cm] (5) at (5*360/6+30: 1cm) {};
	  \node[fill=Black, inner sep = 0pt, minimum size=.2cm] (6) at (6*360/6+30: 1cm) {};
	
	\draw[-,RoyalBlue] (1) edge[bend left=60, line width=.12em] (3);
	\draw[-,RoyalBlue] (2) edge[bend left=60, line width=.12em] (4);
	\draw[-,RoyalBlue] (3) edge[bend left=60, line width=.12em] (5);
	\draw[-,RoyalBlue] (4) edge[bend left=60, line width=.12em] (6);
	\draw[-,RoyalBlue] (5) edge[bend left=60, line width=.12em] (1);
	\draw[-,BrickRed] (6) edge[bend left=60, line width=.12em] (2);
	\end{tikzpicture}	
  }}
  \longrightarrow
  \vcenter{\hbox{
\begin{tikzpicture}
	   \node[fill=Black, inner sep = 0pt, minimum size=.2cm] (1) at (360/6+30: 1cm) {};
	  \node[fill=Black, inner sep = 0pt, minimum size=.2cm] (2) at (2*360/6+30: 1cm) {};
	  \node[fill=Black, inner sep = 0pt, minimum size=.2cm] (3) at (3*360/6+30: 1cm)  {};
	  \node[fill=Black, inner sep = 0pt, minimum size=.2cm] (4) at (4*360/6+30: 1cm) {};
	  \node[fill=Black, inner sep = 0pt, minimum size=.2cm] (5) at (5*360/6+30: 1cm) {};
	  \node[fill=Black, inner sep = 0pt, minimum size=.2cm] (6) at (6*360/6+30: 1cm) {};

	\draw[-,RoyalBlue] (1) edge[bend left=60, line width=.12em] (3);
	\draw[-,RoyalBlue] (2) edge[bend left=60, line width=.12em] (4);
	\draw[-,RoyalBlue] (3) edge[bend left=60, line width=.12em] (5);
	\draw[-,RoyalBlue] (4) edge[bend left=60, line width=.12em] (6);
	\draw[-,RoyalBlue] (5) edge[bend left=60, line width=.12em] (1);

  \draw[-,RoyalBlue] (1) edge[bend right, line width=.12em] (2);
	\draw[-,RoyalBlue] (1) edge[bend left, line width=.12em] (6);

  \node[draw=ForestGreen, dotted, circle, fit=(3), inner sep=2pt, line width=.15 em] {};
  \node[draw=ForestGreen, dotted, circle, fit=(2), inner sep=2pt, line width=.15 em] {};
  \node[draw=ForestGreen, dotted, circle, fit=(6), inner sep=2pt, line width=.15 em] {};
  \node[draw=ForestGreen, dotted, circle, fit=(5), inner sep=2pt, line width=.15 em] {};
  \node[draw=ForestGreen, dotted, circle, fit=(4), inner sep=2pt, line width=.15 em] {};

	\end{tikzpicture}	
  }}
  \longrightarrow 
  \vcenter{\hbox{
  \begin{tikzpicture}
	  \node[fill=Black, inner sep = 0pt, minimum size=.2cm] (1) at (0,0) {};

  \path [draw, RoyalBlue, line width = .12em] (1) to [loop above, looseness=10, min distance=1cm, out=55, in=125] (1);
  \path [draw, RoyalBlue, line width = .12em] (1) to [loop below, looseness=10, min distance=1cm, out=-55, in=-125] (1);

	\end{tikzpicture}	
  }}
  \]
  \caption{Example of a term which gives rise to a graph with a single vertex.}
  \label{fig:procedure-4}
\end{figure}

After this step, we now have
 \begin{equation*}
  \begin{aligned}
    &\frac{1}{pd}\Tr(\Q^k) 
    \\&
    =\frac{1}{p}\Tr(H^k)+
    \frac{1}{pd}
    \hspace{-.6em}\sum_{\substack{\tilde{\ell}, u_1, \dots, u_{\tilde{\ell}} \\ \tilde{n}_1+\cdots+\tilde{n}_{\tilde{\ell}+1}=k+1}}\hspace{-.6em}
    \left(       \frac{\alpha_\phi^2}{d}
    W D H^{\tilde{n}_1+\tilde{n}_{\tilde{\ell}+1}-2} DW^\top \right)_{u_0, u_2}
    \prod_{j=2}^{\tilde{\ell}}
    \left(       \frac{\alpha_\phi^2}{d}
    W D H^{\tilde{n}_j-1} DW^\top \right)_{u_{j-1}, u_{j+1}}
    +o(1)
  \end{aligned}
\end{equation*}
with very high probability, where $\tilde{\ell} \in \{1,\dots,k\}$ and $\tilde{n}_j \geq 1$ for $j \in \unn{1}{\tilde{\ell}}$ are such that $\tilde{n}_1+\cdots+\tilde{n}_{\tilde{\ell}+1}=k+1$, and where $u_i \in \unn{1}{d}$ for $i \in \unn{1}{\tilde{\ell}}$ satisfy $u_{i+1} \neq u_i$ for $i \in \unn{1}{\tilde{\ell}}$.
  \paragraph{Step 2 :Removing the constraint that $u_1 \neq u_2 \neq \dots \neq u_{\tilde{\ell}} \neq u_1$}. To remove the constraint from the sum, we actually can perform exactly the same reasoning as above except that we do not color any edge, but instead guarantee at least an identification. The reasoning is then simpler and the same procedure gives rise to the same types of graph, and thus the same bounds. This gives that we now have 
  \begin{equation}\label{eq:trqk-2}
  \begin{aligned}
    &\frac{1}{pd}\Tr(\Q^k) 
    \\&
    =\frac{1}{p}\Tr(H^k)+
    \frac{1}{pd}
    \hspace{-.2em}\sum_{\substack{\tilde{\ell}, u_1, \dots, u_{\tilde{\ell}} \\ \tilde{n}_1+\cdots+\tilde{n}_{\tilde{\ell}}=k+1}}
    \left(       \frac{\alpha_\phi^2}{d}
    W D H^{\tilde{n}_1+\tilde{n}_{\tilde{\ell}+1}-2} DW^\top \right)_{u_0, u_2}
    \prod_{j=2}^{\tilde{\ell}}
    \left(       \frac{\alpha_\phi^2}{d}
    W D H^{\tilde{n}_j-1} DW^\top \right)_{u_{j-1}, u_{j+1}}
    +o(1)
  \end{aligned}
\end{equation}
with very high probability, where $\tilde{\ell} \in \{1,\dots,k\}$ and $\tilde{n}_j \geq 1$ for $j \in \unn{1}{\tilde{\ell}}$ are such that $\tilde{n}_1+\cdots+\tilde{n}_{\tilde{\ell}+1}=k+1$, and where $u_i \in \unn{1}{d}$ for $i \in \unn{1}{\tilde{\ell}}.$ Note that the sum on the $u_i$'s is now unconstrained.
  \paragraph{Step 3: Bounding the contribution of odd $\tilde{\ell}$} Consider a fixed odd $\tilde{\ell}=2q+1$. The sum part from \eqref{eq:trqk-2} can be written in the more compact form 
  \[
  \frac{1}{pd}\sum_{\substack{u_1,\dots,u_{\tilde{\ell}}\\ \nt_1+\dots+\nt_{\tilde{\ell}}=k+1}}
  M(\nt_1+\nt_{\tilde{\ell}+1}-1)_{u_{2q+1},u_2}
  \prod_{j=2}^{\tilde{\ell}} M(\nt_j)_{u_{j-1},u_{j+1}}
  \]
  where the sum over $u_i$'s is unconstrained. This is simply a trace of product of $M(m)$ since from the structure of the entries we have that $u_0 \to u_2 \to \dots \to u_{2q}$ (wherer $u_i\to u_j$ means that the entry $M(m)_{u_i,u_j}$ appears) but since $\tilde{\ell}$ is odd, we then have that $u_{2q}\to u_1 \to \dots \to u_{2q+1} \to u_2$. Thus this term can be written as 
  \[
  \frac{1}{pd}\sum_{\nt_1+\dots+\nt_{\tilde{\ell}}=k+1}
  \mathrm{Tr}\left(
    M(\nt_1+\nt_{\tilde{\ell}+1})
    M(\nt_3)M(\nt_5)\dots M(\nt_{2q+1})M(\nt_2)M(\nt_4)\dots M(\nt_{2q})
  \right)
  =
  \O{
    \frac{d^{k\varepsilon}}{p}
  }
  \]
  with very high probability where we used the fact that the operator norm of each matrix $M$ is bounded by $d^{\varepsilon}$ and that the product of $M(m)$ is a matrix of size $d\times d$.
  \paragraph{Step 4: Conclusion}

  As we have removed the contribution of odd $\tilde{\ell}$ terms, we set $\ell = \tilde{\ell}/2$. We then have 
  \[
  \begin{aligned}
    &\frac{1}{pd}\Tr(\Q^k) 
    \\&=
    \frac{1}{p}\Tr(H^k)+
    \frac{1}{pd}
    \sum_{\substack{\ell, u_1, \dots, u_{2\ell} \\ \tilde{n}_1+\cdots+\tilde{n}_{2\ell}=k+1}}
    \left(     
      \frac{\alpha_\phi^2}{d}  
      W D H^{\tilde{n}_1+\tilde{n}_{2\ell+1}-2} DW^\top 
    \right)_{u_0, u_2}
    \prod_{j=2}^{\ell}
    \left( 
      \frac{\alpha_\phi^2}{d}
      W D H^{\tilde{n}_j-1} DW^\top 
    \right)_{u_{j-1}, u_{j+1}}
  \end{aligned}
  \] 
  This splits into two traces of matrix products, one in which we use the $u_{2i}$ and one in which we use the $u_{2i-1}$.  Hence, we set $n_1 = \tilde{n}_1+\tilde{n}_{2\ell+1}-1$, $n_j = \tilde{n}_{2j-1}$ for $j \in \unn{1}{\ell}$ and $n_j' = \tilde{n}_{2j}$ for $j \in \unn{1}{\ell}$.  We then have 
  \[
  \begin{aligned}
    &\sum_{\substack{\ell, u_1, \dots, u_{2\ell} \\ \tilde{n}_1+\cdots+\tilde{n}_{2\ell}=k+1}}
    \prod_{j=1}^{\ell}
    \left( 
      \frac{\alpha_\phi^2}{d}
      W D H^{{n}_j-1} DW^\top 
    \right)_{u_{2j-2}, u_{2j}}
    \prod_{j=1}^{\ell}
    \left( 
      \frac{\alpha_\phi^2}{d}
      W D H^{n_j'-1} DW^\top 
    \right)_{u_{2j-1}, u_{2j+1}} \\ 
    &=
    \sum_{\substack{\ell, n_1, \dots, n_{\ell} \\ n_1', \dots, n_{\ell}'}}
    n_1 \Tr[ M(n_1)M(n_2)\cdots M(n_\ell)]\Tr[M(n_1')M(n_2')\cdots M(n_\ell')],
  \end{aligned}
  \]
  where the sum is over all $\ell, n_1, \dots, n_{\ell}$ and $n_1', \dots, n_{\ell}'$ such that $n_1+\cdots+n_{\ell} + n_1'+\cdots+n_{\ell}'=k$.
  Here we have used that there are $n_1$ ways to choose the $\tilde{n}_1$ and $\tilde{n}_{\ell+1}$ (as their sum is $n_1+1$ and we should write this as a sum of $a+b$ with both $a,b\geq 1$).
\end{proof}

\begin{rmk}
Suppose that we have $\beta_\psi^2 = 0$.  Then we have from Proposition \ref{prop:TrQk} that 
\[
\begin{aligned}
\frac{1}{dp}\Tr(\Q^k) 
&=
\frac{1}{p}\Tr(H^k)+
\frac{1}{dp} \sum_{\substack{\ell, n_1,\dots,n_\ell \\ n_1', \dots, n_{\ell}'}}
n_1 \Tr( H^{n_1+\dots + n_{\ell}}) \Tr( H^{n_1'+\dots + n_{\ell}'}) \\
&=
\frac{1}{p}\Tr(H^k)+
\frac{1}{dp} \sum_{\substack{\ell, n_1,\dots,n_\ell,n_{\ell+1} \\ n_1', \dots, n_{\ell}'}}
\Tr( H^{n_1+\dots + n_{\ell}+n_{\ell+1}-1}) \Tr( H^{n_1'+\dots + n_{\ell}'}). \\
\end{aligned}
\]
Here we have used that the number of ways to express $n_1 +1 = a + b$ for $a,b \geq 1$ is $n_1.$ If we set the sum of $n_1, \dots, n_{\ell},n_{\ell+1}$ to $q$ and the sum of $n_1', \dots, n_{\ell}'$ to $q'$, then we have $q+q'=k+1$ and $q,q' \geq 1$.  The number of ways to write $q$ as a sum of $\ell+1$ natural numbers is $\binom{q-1}{\ell}$.  Hence, we have 
\[
\frac{1}{dp} \sum_{\substack{\ell, n_1,\dots,n_\ell,n_{\ell+1} \\ n_1', \dots, n_{\ell}'}}
\Tr( H^{n_1+\dots + n_{\ell}+n_{\ell+1}-1}) \Tr( H^{n_1'+\dots + n_{\ell}'})
=
\frac{1}{dp} \sum_{\substack{\ell,q,q' \\ q+q'=k+1}} 
\binom{q-1}{\ell}
\binom{q'-1}{\ell-1}
\Tr( H^{q-1}) \Tr( H^{q'}).
\]
Using the Vandermonde identity, we have 
\[
\binom{k-1}{q'}
=
\binom{q+q'-2}{q'}
=
\sum_{\ell=1}^{q'} 
\binom{q-1}{\ell} 
\binom{q'-1}{q'-\ell}
=\sum_{\ell=1}^{q'} 
\binom{q-1}{\ell} 
\binom{q'-1}{\ell-1}.
\]
Therefore the moment formula is 
\[
\begin{aligned}
\frac{1}{dp} \sum_{\substack{\ell,q,q' \\ q+q'=k+1}} 
\binom{q-1}{\ell}
\binom{q'-1}{\ell-1}
\Tr( H^{q-1}) \Tr( H^{q'})
&=
\frac{1}{dp} \sum_{q'=1}^{k-1} 
\binom{k-1}{q'}
\Tr( H^{k-q'}) \Tr( H^{q'}) \\
&=
\frac{1}{dp} \sum_{q'=1}^{k-1} 
\binom{k-1}{q'}
\Tr( H^{k-q'}) \Tr( H^{q'})
\end{aligned}
\]
We conclude that 
\[
\begin{aligned}
\frac{1}{dp}\Tr(\Q^k) 
&= \frac{1}{p}\Tr(H^k)
+
\frac{1}{dp} \sum_{q'=1}^{k-1} 
\binom{k-1}{q'}
\Tr( H^{k-q'}) \Tr( H^{q'})
+o_{\P}(1) \\
&= 
\frac{1-\gamma_2}{p}\Tr(H^k)
+
\frac{1}{dp} \sum_{q'=0}^{k-1} 
\binom{k-1}{q'}
\Tr( H H^{k-1-q'}) \Tr( H^{q'})
+o_{\P}(1) \\
\end{aligned}
\]
From the convergence of the moments of $H$, and using the equality that for iid $A,B$ real random variables,
\[
\E\left[ A(A+B)^{k-1} \right] = 
\E\left[B(A+B)^{k-1} \right] = 
\frac{1}{2}\E\left[ (A+B)^{k} \right],
\]
we can conclude that 
\[
\frac{1}{dp} \Tr(\Q^k) =
(1-\gamma_2) \E\left[ A^k\right] 
+ \frac{\gamma_2}{2} \E\left[ (A+B)^{k} \right] + o_{\P}(1),
\]
where $A$ and $B$ and follow the limit law of the empirical spectral distribution of $H$.
\end{rmk}

\subsection{Free probability formulation of the limit law}\label{subsec:freeproba}
The moment formula of Proposition \ref{prop:TrQk} can be expressed as a product of traces of non-commutative polynomials in the matrices $\frac{1}{d}W^{\top}W$ and $D^2$. It is standard that $\frac{1}{p}W^{\top}W$ converges to a Free Poisson random variable $\mathfrak{p}$ with parameter $\frac{d}{p} \to \frac{1}{\gamma_2}$ as $p \to \infty$ (see \cite{Mingo2017}*{Section 4.5.1}).  Moreover, we can consider a non-commutative probability space $(\mathcal{A}, \tau)$.  Provided the empirical distribution of $D^2$ converges to a distribution with compact support, then we can choose this space to contain $\mathfrak{p}$ and an element $\mathfrak{d}$ such that 
\[
\frac{1}{p}\Tr(D^{2k}) \to \tau(\mathfrak{d}^k)
\quad \text{and}\quad 
\frac{1}{p}\Tr\left( \left(\frac{1}{d}W^{\top}W \right)^{k} \right) \to \tau(\mathfrak{p}^k).
\]
as $p \to \infty$.  Moreover, $\{\mathfrak{d}, \mathfrak{p}\}$ are free.  In this section, we show the following.
\begin{prop}\label{prop:stieltjes}
  Suppose that $(D^2, \tfrac{1}{d}W^{\top}W)$ converge in the sense of non-commutative moments to the pair $(\mathfrak{d}, \mathfrak{p})$ of free non-commutative random variables.  Then the Stieltjes transform of $\Q$, $\frac{1}{dp}\Tr( (\Q-w)^{-1})$, converges pointwise in probability for $w \in \C^+$ to 
  \[
    \tau\left(
      (1-\mathcal{Y})\mathcal{R}
    \right)
    -w
    \tau\otimes\tau \left(
    (\mathbf{1} \otimes \mathcal{Y})
    ( (\mathcal{X}-w)^{\otimes 2} - \mathcal{Y}^{\otimes 2})^{-1}
    \right),
  \]
  where $\mathcal{X} = \beta_\psi^2 \mathfrak{d} + {\gamma_2\alpha_\phi^2}\left( \mathfrak{d}\mathfrak{p} \right)$, where $\mathcal{Y} = \gamma_2\alpha_\phi^2 \mathfrak{d}\mathfrak{p}$ and where $\mathcal{R} = (\mathcal{X}-w)^{-1}$, which defines a probability measure on $[0,\infty)$ by Stieltjes inversion.
\end{prop}

We now observe that if we define 
\[
\mathcal{X}(z) = \beta_\psi^2 \mathfrak{d} + {\gamma_2\alpha_\phi^2}\left( \mathfrak{d}\mathfrak{p}(1+z) \right),
\]
then we can represent traces of products of $M$ as coefficient extraction from $\mathcal{X}(z)$.  We let $[z^\ell]\mathcal{F}(z)$ denote the coefficient of $z^\ell$ in a power series $\mathcal{F}(z)$.
\begin{lem}\label{lem:coefficientextraction}
  Suppose that $(D^2, \tfrac{1}{d}W^{\top}W)$ converge in the sense of non-commutative moments to the pair $(\mathfrak{d}, \mathfrak{p})$ of free non-commutative random variables.  Then for any $k \geq 1$, we have 
  \[
  \begin{aligned}
    &[z^\ell]\tau(\left( \mathcal{X}(z) \right)^k) = \sum_{n_1,n_2,\dots,n_\ell} \frac{n_1}{p} \Tr( M(n_1)M(n_2)\cdots M(n_\ell)) + o_{\P}(1), \\
    &[z^{\ell-1}]\partial_z\tau(\left( \mathcal{X}(z) \right)^{k}) = \sum_{n_1,n_2,\dots,n_\ell} \frac{k}{p}\Tr( M(n_1)M(n_2)\cdots M(n_\ell)) + o_{\P}(1).
  \end{aligned}
  \]
  where the sum is over $n_1,n_2,\dots,n_\ell \geq 1$ such that $k = n_1+\cdots+n_\ell$.
\end{lem}
\begin{proof}
  We have 
  \[
  \tau(\left( \mathcal{X}(z) \right)^k) = \tau\left(    
    \left(
      \beta_\psi^2 \mathfrak{d} + {\gamma_2\alpha_\phi^2}\left( \mathfrak{d}\mathfrak{p}(1+z) \right)
    \right)^k
  \right).  
  \]
  Hence, the coefficient of $z^\ell$ in this expression is given by 
  \[
  [z^\ell]\tau(\left( \mathcal{X}(z) \right)^k) 
  = \sum_{n_1,n_2,\dots,n_\ell,n_{\ell+1}}
  \tau\left(
    \mathcal{X}(0)^{n_1-1} 
    \left(
      \gamma_2\alpha_\phi^2
      \mathfrak{d}\mathfrak{p}
    \right)
    \mathcal{X}(0)^{n_2-1} 
    \left(
      \gamma_2\alpha_\phi^2
      \mathfrak{d}\mathfrak{p}
    \right)
    \cdots 
    \mathcal{X}(0)^{n_{\ell+1}-1} 
  \right),
  \]
  where $k+1 = n_1+\cdots+n_{\ell+1}$ and where all $n_i \geq 1$.
  Thus using cyclicity of the trace, and the convergence of non-commutative moments, we have (noting that $\mathcal{X}(0) = X$)
  \[
  [z^\ell]\tau(\left( \mathcal{X}(z) \right)^k) 
  =
  \sum_{n_1,n_2,\dots,n_\ell,n_{\ell+1}} \frac{1}{p} \Tr( M(n_1+n_{\ell+1}-1)M(n_2)\cdots M(n_\ell)) + o_{\P}(1).
  \]
  The number of ways to write $a=n_1+n_{\ell+1}-1$ for a given natural number $a$ is itself $a$ and therefore, (making the replacement $n_1+n_{\ell+1}-1 \to n_1$) 
  \[
    [z^\ell]\tau(\left( \mathcal{X}(z) \right)^k) = 
    \sum_{n_1,n_2,\dots,n_\ell} 
    \frac{n_1}{p} \Tr( M(n_1)M(n_2)\cdots M(n_\ell)) + o_{\P}(1).
  \]

  As for the second claim, we have 
  \[
  \begin{aligned}
    &[z^\ell]\tau(\left( \mathcal{X}(z) \right)^{k-1} z \partial_z \mathcal{X}(z)) = \\
    &\sum_{n_1,n_2,\dots,n_\ell}
    \tau\left(
      \mathcal{X}(0)^{n_1-1} 
      \left(
        \gamma_2\alpha_\phi^2
        \mathfrak{d}\mathfrak{p}
      \right)
      \mathcal{X}(0)^{n_2-1} 
      \left(
        \gamma_2\alpha_\phi^2
        \mathfrak{d}\mathfrak{p}
      \right)
      \cdots 
      \mathcal{X}(0)^{n_{\ell}-1}   
      \left(
        \gamma_2\alpha_\phi^2
        \mathfrak{d}\mathfrak{p}
      \right)
    \right),
    \end{aligned}
\]
where $k = n_1+\cdots+n_\ell$ and $n_i \geq 1$ for $i \in \unn{1}{\ell}$.
Thus, again using cyclicity of the trace, we have 
\[
  [z^\ell]\tau(\left( \mathcal{X}(z) \right)^{k-1} z \partial_z \mathcal{X}(z)) = 
  \sum_{n_1,n_2,\dots,n_\ell} 
  \frac{k}{p} \Tr( M(n_1)M(n_2)\cdots M(n_\ell)) + o_{\P}(1),
\]
and finally that  
\[
  [z^\ell]\tau(\left( \mathcal{X}(z) \right)^{k-1} z \partial_z \mathcal{X}(z)) = 
  [z^{\ell-1}]\frac{1}{k} \partial_z \tau(\left( \mathcal{X}(z) \right)^k) + o_{\P}(1).
\]
\end{proof}

Thus we conclude with a formal moment formula for $\Tr(\Q^k).$ 
\begin{lem}\label{lem:coefficientextraction2}
  Suppose that $(D^2, \tfrac{1}{d}W^{\top}W)$ converge in the sense of non-commutative moments to the pair $(\mathfrak{d}, \mathfrak{p})$ of free non-commutative random variables.  Then for any $k \geq 1$, we have 
  \begin{multline*}
  [z^0] 
  \left( 
    \sum_{a=1}^{k-1}
    \tau[ \left( \mathcal{X}(z^{-1}) \right)^a ]\tau[ \mathcal{X}(z)^{k-a-1} z \partial_z \mathcal{X}(z) ]
  \right)
  \\=\sum \frac{n_1}{p^2} \Tr[ M(n_1)M(n_2)\cdots M(n_\ell)]\Tr[M(n_1')M(n_2')\cdots M(n_\ell')] + o_{\P}(1),
  \end{multline*}
  where $\ell \in \N$ and $n_j \geq 1$ for $j \in \unn{1}{\ell}$ and $n_j' \geq 1$ for $j \in \unn{1}{\ell}$ are such that $n_1+\cdots+n_\ell+ n_1'+\cdots+n_\ell'=k$.
\end{lem}
\begin{proof}
  We let $k = a+b$ with $a,b \geq 1$, and we divide the sum on the right according to $a= n_1+\cdots+n_\ell$ and $b= n_1'+\cdots+n_\ell'$.  Doing so, and using Lemma \ref{lem:coefficientextraction}
  \begin{multline*}
  \sum \frac{n_1}{p^2} \Tr[ M(n_1)M(n_2)\cdots M(n_\ell)]\Tr[M(n_1')M(n_2')\cdots M(n_\ell')]
  =\\
  \sum_{a,b,\ell} 
  \left(
    [z^\ell]\tau(\left( \mathcal{X}(z) \right)^a)
  \right)
  \left(
    [z^\ell]\tau(\left( \mathcal{X}(z) \right)^{b-1} z \partial_z \mathcal{X}(z))
  \right) + o_{\P}(1).
  \end{multline*}
\end{proof}

\begin{proof}[Proof of Proposition \ref{prop:stieltjes}]
Hence we introduce the formal power series $\mathfrak{f}(w) \in \mathcal{A}[[w]]$, which is convergent for $|w|$ sufficiently large, defined by
\[
\mathfrak{f}(w)
\coloneqq
-
[z^0] 
\left(
  \sum_{k=0}^\infty 
  w^{-k-1}
  \tau(\mathcal{X}(0)^k) 
  +
  w^{-1} 
  \sum_{a=1}^{k-1}
  \tau[ \left( \mathcal{X}(z^{-1})w^{-1} \right)^a ]\tau[ \left( \mathcal{X}(z)w^{-1} \right)^{k-a-1} z \partial_z \mathcal{X}(z) ]
\right).
\]
From Lemma \ref{lem:coefficientextraction2}, we have that the Stieltjes transform of $\Q$ converges pointwise in probability to $\mathfrak{f}$ for all $|w|$ sufficiently large.
We take the second sum as empty when $k=0$ and $k=1$.
Then we have that (for $|w|$ sufficiently large)
\[
\begin{aligned}
  \mathfrak{f}(w)
  &= 
  -
  [z^0] 
  \left(
    \sum_{k=0}^\infty  
    w^{-k-1}
    \tau(\mathcal{X}(0)^k) 
    +
    w^{-1} 
    \sum_{a=1}^{k-1}
    \tau[ \left( \mathcal{X}(z^{-1})w^{-1} \right)^a ]\tau[ \left( \mathcal{X}(z)w^{-1} \right)^{k-a-1} z \partial_z \mathcal{X}(z) ]
  \right) \\
  &= 
  [z^0] 
  \left(
    \tau\left( \frac{1}{\mathcal{X}(0) - w} \right)
    -
    w^{-1} 
    \tau \left( \frac{\mathcal{X}(z^{-1})w^{-1}}{1- \mathcal{X}(z^{-1})w^{-1}} \right) 
    \tau \left( \frac{z \partial_z \mathcal{X}(z)}{1-\mathcal{X}(z)w^{-1}}  \right)
  \right) \\
  &= 
  [z^0] 
  \left(
    \tau\left( \frac{1}{\mathcal{X}(0) - w} \right)    -
    \tau \left( \frac{\mathcal{X}(z^{-1})}{\mathcal{X}(z^{-1})- w} \right) 
    \tau \left( \frac{z \partial_z \mathcal{X}(z)}{\mathcal{X}(z)- w}  \right)
  \right),
\end{aligned}
\]
and so the series is convergent for $|w| \gg \max\{|z|, |z|^{-1}\}$.  

We set $\mathcal{Y} = \partial_z \mathcal{X}(z)$, which we note is independent of $z$.  Then we have 
\[
\mathcal{X}(z) = \mathcal{X}(0) + z\mathcal{Y}.
\]
For $w$ sufficiently large we have that 
\[
\tau \left( \frac{\mathcal{X}(z^{-1})}{\mathcal{X}(z^{-1})- w} \right) 
=
\tau \left( \frac{\mathcal{X}(0) + z^{-1}\mathcal{Y}}{\mathcal{X}(0) + z^{-1}\mathcal{Y}- w} \right)
=
1
+
w\sum_{k=0}^\infty (-1)^k z^{-k} \tau\left((\mathcal{X}(0) -w)^{-1} ( \mathcal{Y}(\mathcal{X}(0) -w)^{-1} )^k\right).
\]
We also have 
\[
  \tau \left( \frac{z \partial_z \mathcal{X}(z)}{\mathcal{X}(z)- w}  \right)
  =
  \tau \left( \frac{z \mathcal{Y}}{\mathcal{X}(0) + z\mathcal{Y}- w}  \right)
  =
  \sum_{k=0}^\infty (-1)^k z^{k} \tau\left(( \mathcal{Y}(\mathcal{X}(0) -w)^{-1} )( \mathcal{Y}(\mathcal{X}(0) -w)^{-1} )^{k}\right).
\]
Hence, we have the representation, with $\mathcal{R} = (\mathcal{X}(0)-w)^{-1} = (\beta_\psi^2 \mathfrak{d} + {\gamma_2\alpha_\phi^2} \mathfrak{d}\mathfrak{p}-w)^{-1}$
and $\mathcal{Y} = {\gamma_2\alpha_\phi^2} \mathfrak{d}\mathfrak{p}$,
\begin{equation}\label{eq:frepresentation}
\begin{aligned}
\mathfrak{f}(w)
&=
\tau\left(
    (1-\mathcal{Y})\mathcal{R}
\right)
-w
\tau\otimes\tau \left(
(\mathcal{R} \otimes \mathcal{Y}\mathcal{R})
(\mathbf{1} \otimes \mathbf{1} - \mathcal{Y}\mathcal{R} \otimes \mathcal{Y}\mathcal{R})^{-1}
\right) \\
&=
\tau\left(
  (1-\mathcal{Y})\mathcal{R}
\right)
-w
\tau\otimes\tau \left(
(\mathbf{1} \otimes \mathcal{Y})
( (\mathcal{X}(0)-w)^{\otimes 2} - \mathcal{Y}^{\otimes 2})^{-1}
\right).
\end{aligned}
\end{equation}
We note that this is a Stieltjes transform of a compactly supported probability measure $\nu$ on $\R$, as $\mathfrak{f}(w) \sim -\frac{1}{w}$ as $w \to \infty$ and from the Herglotz representation theorem.  From the identity theorem for analytic functions, and as Stieltjes transforms of subprobability measures are a normal family on $\C^+$, it follows that the pointwise convergence in probability holds not just for $|w|$ large but for all $w \in \C^+$.  The support of the measure follows, as it is a weak limit of the empirical spectral distribution of $\Q$.
\end{proof}

\appendix

\section{ A few properties of sample covariance matrices }

In this section we collect a few relatively standard properties of random matrices of the form 
\begin{equation}\label{a:samplecov}
  \begin{aligned}
  H & \coloneqq \beta_\psi^2 D^2 + \frac{\alpha_\phi^2}{d} D W^{\top} W D = \beta_\psi^2 D^2 + Z^{\top} Z, \\
  H(u) & \coloneqq \beta_\psi^2 D^2 +  D W^{\top}\left(\frac{\alpha_\phi^2}{d} + \mathbf{e}_u \otimes \mathbf{e}_u \right) W D = \beta_\psi^2D^2 +  Z^{\top} Z + \frac{d}{\alpha_\phi^2} Z^\top(\mathbf{e}_u \otimes \mathbf{e}_u)Z,
  \end{aligned}
\end{equation}
as well as the matrices 
\begin{equation}\label{a:samplecov2}
  \begin{aligned}
  M(n) & \coloneqq \frac{\alpha_\phi^2}{d}WD H^{n-1} D W^{\top} \\
  M(u,n) & \coloneqq \frac{\alpha_\phi^2}{d} WD H(u)^{n-1} D W^{\top},
  \end{aligned}
\end{equation}
where we recall that $D$ is a diagonal matrix with bounded entries, $W$ is i.i.d.\ standard normal, and $Z=\frac{\alpha_\phi}{\sqrt{d}}WD$. In this section, we show the following estimate.
\begin{thm}\label{a:main}
  Suppose $\epsilon > 0$ is fixed.
  Suppose $\tr(D^2) \geq \epsilon d$.
  With very high probability, there is a constant $C  > 0$ so that for any $n \in \N$,
  \[
  \max_{u,i} (M(n))_{ii} \leq C^n 
  \quad \text{and} \quad
  \max_{u,i} |(M(u,n))_{ii}| \leq C^n.
  \]
  Moreover, with very high probability, the following bounds hold
  \begin{enumerate}
    \item 
    \(
    \max_{i \neq j} |(M(u,n))_{ij}| \leq d^{\epsilon-1/2}
    \)
    \item 
    \(
    \max_{i \neq j} |(M(n))_{ij}| \leq d^{\epsilon-1/2}.
    \)
    \item 
    \(
    \max_{i \neq u} |(M(u,n))_{ii} - (M(n))_{ii}| \leq d^{\epsilon-1}.
    \)
  \end{enumerate}
\end{thm}

The matrix $H$ is essentially a sample covariance matrix.  When the factor $\beta_\psi = 0$, there is a vast literature we can apply.  The presence of $\beta_\psi^2 > 0$ does not substantially change the methods, in particular the \emph{leave-one-out} method (see \cite{LiaoCouillet} for a pleasant introduction or \cite{ElliotPaquetteNotes}*{Theorem 18} for a less pleasant one) can be used to show the bulk spectral law.

We recall the resolvent of a square matrix $A$ is defined by 
\[
R(z;A) \coloneqq (A-zI)^{-1},
\]
for all $z$ not in the spectrum of $A$.
We also recall the Sherman-Morrison-Woodbury formula, which gives that for $u$ and $v$ and all $z \in \C$ 
\begin{equation}\label{a:smw}
R\left(z ; A+ u \otimes v \right)-R(z ; A)=-\frac{R(z ; A) (u \otimes v) R(z ; A)}{1+u^\top R(z ; A) v}.
\end{equation}

This algebraic machinery powers the \emph{leave-one-out} method.
We further introduce the matrices 
\[
H^{[S]} \coloneqq \beta_\psi^2 D^2 + 
Z^{\top} \left( I - \sum_{i \in S} \mathbf{e}_i \otimes \mathbf{e}_i \right) Z 
= \frac{\alpha_\phi^2}{d} \sum_{i \not\in S} w_i \otimes w_i,
\quad w_i \coloneqq  D W^{\top}\mathbf{e}_i.
\]
For a singleton $S = \{i\}$, we abbreviate $H^{[S]} = H^{[i]}$.
Then if $S' \setminus S = \{i\}$, from \eqref{a:smw} we have that for $\Im z > 0$
\begin{equation}\label{a:leaveoneout}
  \begin{aligned}
  R(z; H^{[S]}) - R(z; H^{[S']}) &= -\frac{\alpha_\phi^2}{d}\frac{R(z; H^{[S']}) (w_i \otimes w_i) R(z; H^{[S']})}{1+\frac{\alpha_\phi^2}{d}w_i^\top R(z; H^{[S']}) w_i}, \\
  R(z; H(u)) - R(z; H^{[u]}) &= -\left(\frac{\alpha_\phi^2}{d}+1\right)\frac{R(z; H^{[u]}) (w_u \otimes w_u) R(z; H^{[u]})}{1+\left(\frac{\alpha_\phi^2}{d}+1\right)w_u^\top R(z; H^{[u]}) w_u}.
  \end{aligned}
\end{equation}

Now it is straightforward to show by standard methods that the resolvent entries of $H^{[S]}$ are with very high probability well behaved on a contour that (with very high probability) encloses the spectrum of $H$.
\begin{prop}\label{a:usual}
  Suppose $\epsilon > 0$ is fixed and $k \in \N$ is fixed.
  Suppose that $\epsilon < \frac{p}{d} < \frac{1}{\epsilon}$.
  There is a constant $C > 0$ (not depending on $\epsilon$) so that with very high probability the norm of $H^{[S]}$ is bounded by $C$ for any $S \subseteq \unn{1}{d}$.  If we let $\Gamma$ be the contour based at zero of radius least $C+1$ (and independent of $d$), then with very high probability, the following estimates hold uniformly in $S \subseteq \unn{1}{d}$ with $|S| \leq k$ and all $\epsilon >0$ sufficiently small:
  \begin{enumerate}
    \item The off-diagonal entries of the resolvent of $H^{[S]}$ are uniformly small on the contour. That is,
    \begin{equation}\label{a:usual1}
      \max_{z \in \Gamma}\max_{i \neq j} |R(z ; H^{[S]})_{ij}| \leq d^{\epsilon-1/2}.
    \end{equation}
    \item The diagonal entries of the resolvent of $H^{[S]}$ are uniformly bounded on the contour. That is,
    \begin{equation}\label{a:usual2}
      \max_{z \in \Gamma}\max_{i} |R(z ; H^{[S]})_{ii}| \leq 1.
    \end{equation}
    \item Uniformly over the contour,
    \[
    \min_{z \in \Gamma} \left| 1+ \frac{\alpha_\phi^2}{d}\tr(R(z;H^{[S]})D^2)\right| \geq \epsilon 
    \quad \text{and} \quad 
    \min_{z \in \Gamma} \left|\frac{\alpha_\phi^2}{d}\tr(R(z;H^{[S]})D^2)\right| \geq \epsilon \tr(D^2).
    \]
  \end{enumerate}
\end{prop}
\begin{proof}
  These are results that are contained in some form in most resolvent based approaches to sample covariance matrices.
  The norm bound is a direct consequence of the boundedness of the entries of $D$ and standard operator norm bounds for the Gaussian random matrix $W$.  This makes the second conclusion trivial, as the operator norm of the resolvent on the contour is bounded by $1$ (note that ).  The bound for the off-diagonal entries of the resolvent is nontrivial, but it follows the path of the usual self-consistent equation for the resolvent of a sample covariance matrix by the leave-one-out method.  See \cite{ElliotPaquetteNotes}*{Theorem 18} for a detailed proof.  The final point can be ensured by increasing the radius of the contour is sufficiently large, as the resolvent becomes close in norm to $-I/z$ (with a norm error $O(1/z^2)$).
\end{proof}

Building on this, we derive two less standard corollaries that we need. 
\begin{prop}\label{a:lessusual}
  Suppose the same setup as Proposition \ref{a:usual}, in particular with $\epsilon,$ $C$ and $\Gamma$ as before. Suppose further that $\tr(D^2) \geq \epsilon d$. Then with very high probability, 
  \begin{enumerate}
    \item The off-diagonal generalized resolvent entries of $H$ are small on the contour, that is,
    \begin{equation}\label{a:lessusual1}
      \max_{z \in \Gamma}\max_{i \neq j} |w_j^{\top}R(z ; H)w_i| \leq d^{\epsilon+1/2}.
    \end{equation}
    \item The off-diagonal generalized resolvent entries of $H(u)$ are close to those of $H$ provided we do not look in the $u$ direction, that is,
    \begin{equation}\label{a:lessusual2}
      \max_{z \in \Gamma}\max_{\substack{i \neq u \\ j}} |w_i^{\top}R(z ; H(u)) w_j-w_i^{\top}R(z ; H) w_j| \leq d^{\epsilon+1/2}.
    \end{equation}
    If neither $i$ nor $j$ is equal to $u$, the bound improves to $d^{\epsilon}$ with very high probability.
  \end{enumerate}
\end{prop}
\begin{proof}
  We begin with the first claim.
  
  We let $S' = \{i,j\}$ and $S = \{i\}$.  We claim that for the case of $H^{[S']}$ we have that with very high probability,
  \[
  |w_j^{\top} R(z; H^{[S']})w_i| \leq d^{\epsilon+1/2}
  \]
  from the Hanson-Wright inequality and the independence of the matrix $H^{[S']}$ (which has bounded operator norm with very high probability) from the vectors $w_i$ and $w_j$.  

  Now we can use \eqref{a:leaveoneout} to extend the claim to the case of $H^{[S]}=H^{[i]}$ and to $H$. In particular, 
  \[
    w_j^{\top} R(z; H^{[S]})w_i = w_j^{\top} R(z; H^{[S']})w_i + \frac{\alpha_\phi^2}{d}\frac{\left(w_j^{\top} R(z; H^{[S']})w_j\right)\left(w_j^{\top} R(z; H^{[S']})w_i\right)}{1+\frac{\alpha_\phi^2}{d}w_j^{\top} R(z; H^{[S']})w_j}.
  \]
  With very high probability, the denominator is bounded away from $0$ (from part 3 of Proposition \ref{a:usual} and Hanson-Wright), and the $(w_j,w_j)$ inner product is bounded above.  The cross term, is $d^{\epsilon+1/2}$ with very high probability, from the independence of $w_i$, and so we have extended the claim to $H^{[i]}$ (indeed the extra contribution from the leave-one-out factor is smaller by a factor of $d$ than we need).  

  We can now repeat this argument to bound the off-diagonal generalized entries of $H$.  Using \eqref{a:leaveoneout} again, we have that
  \[
  w_j^{\top} R(z; H)w_i = w_j^{\top} R(z; H^{[S]})w_i + \frac{\alpha_\phi^2}{d}\frac{\left(w_j^{\top} R(z; H^{[S]})w_i\right)\left(w_i^{\top} R(z; H^{[S]})w_i\right)}{1+\frac{\alpha_\phi^2}{d}w_i^{\top} R(z; H^{[S]})w_i}.
  \]
  The same argument now applies, which establishes the first claim at a single $z$ with very high probability.  As the operator norm of the resolvent is bounded by $1$ with very high probability on the contour, we have also that the $z$-derivative is bounded similarly (using the $z$-derivative of a resolvent is its square).  Hence by applying Hanson-Wright pointwise over a mesh of the contour of cardinality $d^{100}$, we conclude the claim uniformly on the contour with very high probability.

  We now turn to the second claim.  Using \eqref{a:leaveoneout},
  \[
  w_i^{\top} R(z; H(u)) w_j = w_i^{\top} R(z; H^{[u]}) w_j + \left(\frac{\alpha_\phi^2}{d}+1\right)\frac{\left(w_i^{\top} R(z; H^{[u]}) w_u\right)\left(w_j^{\top} R(z; H^{[u]}) w_u\right)}{1+\left(\frac{\alpha_\phi^2}{d}+1\right)w_u^\top R(z; H^{[u]}) w_u}.
  \]
  Hence for whichever of $i$ or $j$ is not equal to $u$, we have that its inner product term in the numerator is $O(d^{\epsilon+1/2})$ with very high probability.  The denominator owing to the scaling is at least $\epsilon d/2$ (in norm) with very high probability, and hence combining everything we conclude that with very high probability,
  \[
    |w_i^{\top} R(z; H(u)) w_j - w_i^{\top} R(z; H^{[u]}) w_j| \leq d^{\epsilon+1/2}.
  \]
  Note if neither is equal to $u$, the bound improves to $d^{\epsilon}$ with very high probability.  Using the same leave-one-out bound, we can compare $w_i^{\top} R(z; H^{[u]}) w_j$ to $w_i^{\top} R(z; H) w_j$, which produces the same error.
\end{proof}

These bounds now directly imply Theorem \ref{a:main}, from the representations
\[
\begin{aligned}
M(n)_{ij} &= \frac{-1}{2\pi i} \frac{\alpha_\phi^2}{d}\int_{\Gamma} w_i^{\top} R(z;H) w_j z^{n}\dif z \\
M(u,n)_{ij} &= \frac{-1}{2\pi i} \frac{\alpha_\phi^2}{d}\int_{\Gamma} w_i^{\top} R(z;H(u)) w_j z^{n}\dif z.
\end{aligned}
\]
Theorem \ref{a:main} now follows immediately from Proposition \ref{a:lessusual}.

\bibliographystyle{abbrv}
\bibliography{ntk}

\end{document}